\documentclass[12pt]{amsart}

\usepackage{amssymb}

\usepackage{amsmath}
\usepackage{latexsym}
\usepackage{extarrows}
\usepackage{enumerate}
\usepackage{txfonts}
\usepackage{mathtools}
\usepackage{comment}
\usepackage{bm}
\usepackage[numbers,sort&compress]{natbib}
\usepackage[hidelinks]{hyperref} 

\makeatletter
\@namedef{subjclassname@2020}{%
  \textup{2020} Mathematics Subject Classification}
\makeatother

\usepackage[T1]{fontenc}

\newtheorem{theorem}{Theorem}[section]
\newtheorem{corollary}[theorem]{Corollary}
\newtheorem{conjecture}[theorem]{Conjecture}
\newtheorem{lemma}[theorem]{Lemma}
\newtheorem{proposition}[theorem]{Proposition}

\newtheorem{reduction}{Reduction}

\theoremstyle{definition}
\newtheorem{definition}[theorem]{Definition}
\newtheorem{remark}[theorem]{Remark}
\newtheorem{example}[theorem]{Example}

\numberwithin{equation}{section}

\usepackage{tikz}
\usetikzlibrary{cd}
\usepackage{hyperref}

\makeatletter
\newcommand{\subalign}[1]{%
  \vcenter{%
    \Let@ \restore@math@cr \default@tag
    \baselineskip\fontdimen10 \scriptfont\tw@
    \advance\baselineskip\fontdimen12 \scriptfont\tw@
    \lineskip\thr@@\fontdimen8 \scriptfont\thr@@
    \lineskiplimit\lineskip
    \ialign{\hfil$\m@th\scriptstyle##$&$\m@th\scriptstyle{}##$\hfil\crcr
      #1\crcr
    }%
  }%
}
\makeatother

\newcommand{\eps}{\varepsilon}

\newcommand{\N}{\mathbb{N}}
\newcommand{\Z}{\mathbb{Z}}
\newcommand{\E}{\mathbb{E}}

\renewcommand{\O}{\mathcal{O}}

\usepackage{bbm} 
\newcommand{\ind}{\mathbbm{1}}

\newcommand{\UC}{\text{UC-}\lim}

\newcommand{\seminorm}[2]{{\left\vert\kern-0.25ex\left\vert\kern-0.25ex\left\vert #2 
    \right\vert\kern-0.25ex\right\vert\kern-0.25ex\right\vert}_{#1}}

\newcommand{\gen}{\textup{gen}}
\newcommand{\supp}{\textup{supp}}

\newcommand{\id}{\textup{id}}

\begin{document}


\baselineskip=17pt


\title[]{Equidistribution in 2-Nilpotent Polish Groups and triple restricted sumsets}

\author[E. Ackelsberg]{Ethan Ackelsberg}
\address{Institute of Mathematics \\ \'{E}cole Polytechnique F\'{e}d\'{e}rale de Lausanne (EPFL) \\
1015 Lausanne, Switzerland}
\email{ethan.ackelsberg@epfl.ch}

\author[A. Jamneshan]{Asgar Jamneshan}
\address{Institute of Mathematics \\ University of Bonn \\ 53113 Bonn, Germany}
\email{ajamnesh@math.uni-bonn.de}


\begin{abstract}
The aim of this paper is to establish a Ratner-type equidistribution theorem for orbits on homogeneous spaces associated with \(2\)-nilpotent locally compact Polish groups under the action of a countable discrete abelian group. We apply this result to establish the existence of triple restricted sumsets in subsets of positive density in arbitrary countable discrete abelian groups, subject to a necessary finiteness condition.
\end{abstract}



\maketitle

\setcounter{tocdepth}{1}
\tableofcontents

\section{Introduction}

This paper consists of two parts. In the first part, we establish a Ratner-type equidistribution theorem for orbits on a homogeneous space of a \( 2 \)-nilpotent locally compact Polish group under the action of a countable discrete abelian group by translations. This result builds on a recent structure theorem \cite{jst} for a certain class of measure-preserving systems over countable discrete abelian groups, known as Conze--Lesigne systems. In the \( 2 \)-nilpotent case, it generalizes equidistribution results for linear orbits on nilmanifolds under \( \mathbb{Z}^d \)-actions.

In the second part, building on a recent ergodic-theoretic approach to infinite sumsets in sets of positive density in the integers \cite{kmrr}, we use our equidistribution theorem to establish the existence of triple restricted  sumsets in sets of positive density in an arbitrary countable discrete abelian group, under a necessary finiteness assumption on the group. This advances a question and a conjecture from \cite{kmrr_survey} and extends a recent result on double restricted sumsets in sets of positive density in a countable discrete abelian group from \cite{cm}.

Accordingly, both the remainder of the introduction and the main body of the paper are structured in two corresponding parts. In what follows, we provide a more technical overview of each part.

\subsection{Equidistribution in 2-nilpotent Polish groups}

We begin by setting out our notation for dynamical systems.  

\begin{definition}
Let \(\Gamma\) be a countable discrete abelian group. 
\begin{itemize}
    \item A \emph{topological dynamical $\Gamma$-system} is a tuple $(X,T_X)$ consisting of a compact metric space $X$ and an action $T_X$ of $\Gamma$ on $X$ by homeomorphisms. A topological dynamical $\Gamma$-system $(X, T_X)$ is an \emph{extension} of another such system $(Y, T_Y)$ if there is a continuous surjection $\pi : X \to Y$, called a \emph{topological factor map}, such that $\pi\circ T^\gamma_X = T^\gamma_Y \circ \pi$ for all $\gamma\in \Gamma$. Two topological dynamical $\Gamma$-systems $(X,T_X)$ and $(Y, T_Y)$ are said to be \emph{isomorphic} if $(X,T_X)$ is an extension of $(Y, T_Y)$ and $(Y, T_Y)$ is an extension of $(X,T_X)$. 
    \item A \emph{measure-preserving dynamical $\Gamma$-system} is a tuple $(X,\Sigma_X,\mu_X,T_X)$ consisting of a Lebesgue probability space $(X,\Sigma_X,\mu_X)$ and a  near-action\footnote{Meaning that $T_X^{\gamma}\circ T_X^{\gamma'}=T_X^{\gamma+\gamma'}$ holds $\mu_X$-almost everywhere for all $\gamma,\gamma'\in\Gamma$.} $T_X$ of $\Gamma$ on $X$ by measure-preserving transformations. A measure-preserving dynamical $\Gamma$-system $(X,\Sigma_X,\mu_X, T_X)$ is an \emph{extension} of another such system $(Y,\Sigma_Y,\mu_Y, T_Y)$ if there is a measure-preserving map $\pi : X \to Y$, called a \emph{measurable factor map}, such that $\pi\circ T^\gamma_X = T^\gamma_Y \circ \pi$ holds $\mu_X$-almost surely for all $\gamma\in \Gamma$. Two measure-preserving $\Gamma$-systems $(X,\Sigma_X,\mu_X,T_X)$ and $(Y,\Sigma_Y,\mu_Y, T_Y)$ are said to be \emph{isomorphic} if $(X,\Sigma_X,\mu_X,T_X)$ is an extension of $(Y,\Sigma_Y,\mu_Y, T_Y)$ and vice versa. 
    
    A measure-preserving dynamical $\Gamma$-systems $(X,\Sigma_X,\mu_X,T_X)$ is said to be \emph{ergodic} if $\mu_X(E)\in\{0,1\}$ whenever $\max_{\gamma\in\Gamma} \mu_X(T^\gamma E\Delta E)=0$ for $E\in \Sigma_X$. 
\end{itemize}
\end{definition}

The main dynamical systems of interest in this paper come from a family of systems with nilpotent structure.
We recall and introduce some notions related to nilpotent groups in order to define this class of systems.
A \emph{filtration} on a group $G$ is a decreasing sequence of subgroups $G = G_0 = G_1 \ge G_2 \ge \ldots$ such that the commutator group $[G_i, G_j]$ is a subgroup of $G_{i+j}$ for $i,j \ge 0$.
If $G$ is a topological group, we will also assume that the groups $G_i$ are closed subgroups.
An example of a filtration is the \emph{lower central series} defined by $G_0 = G_1 = G$ and $G_{i+1} = [G,G_i]$ for $i \ge 1$.
A group is \emph{nilpotent} if it has a filtration with $G_{d+1} = \{1\}$ for some $d \in \N$.
In such a case, we call $d$ the \emph{degree} of the filtration, and the minimal degree of filtrations of $G$ (which is attained by the lower central series) is the \emph{step} of the nilpotent group $G$.
Following \cite{gt}, we denote a filtration on $G$ by $G_{\bullet}$.

We will consider dynamical systems defined on homogeneous spaces of nilpotent groups.
Let $G$ be a nilpotent locally compact Polish group.
A subgroup $\Lambda \le G$ is a \emph{lattice} in $G$ if $\Lambda$ is discrete and co-compact.
Given a filtration $G_{\bullet}$ of $G$, we say that a lattice $\Lambda \le G$ is \emph{compatible} with $G_{\bullet}$ if $\Lambda_i = G_i \cap \Lambda$ is a lattice in $G_i$ for every $i \ge 0$.
We can now define the main class of systems with which we will work.

\begin{definition}
    A \emph{translational system} is a measure-preserving dynamical $\Gamma$-system of the form $(G/\Lambda,\Sigma_{G/\Lambda},\mu_{G/\Lambda},T_{G/\Lambda})$, where $G$ is a nilpotent locally compact Polish\footnote{A Polish group is topological group whose topology is separable and completely metrizable.} group, $\Lambda \le G$ is a lattice compatible with some finite degree filtration $G_{\bullet}$ of $G$, there is a group homomorphism $\phi : \Gamma \to G$ such that $T_{G/\Lambda}^{\gamma} x = \phi(\gamma)x$ for $x \in G/\Lambda$, and $\mu_{G/\Lambda}$ is the Haar probability measure on $G/\Lambda$.
    The \emph{degree} of the translational system $(G/\Lambda,\Sigma_{G/\Lambda},\mu_{G/\Lambda},T_{G/\Lambda})$ is the degree of the filtration $G_{\bullet}$.
    We can view a translational $\Gamma$-system in the category of topological dynamical $\Gamma$-systems by \emph{forgetting} its measurable structure.
    A \emph{rotational $\Gamma$-system} is a degree 1 translational $\Gamma$-system $(G/\Lambda,\Sigma_{G/\Lambda},\mu_{G/\Lambda},T_{G/\Lambda})$. (In this case, $G$ is abelian so $\Lambda$ is normal and $G/\Lambda$ is itself a group and not just a homogeneous space.)
\end{definition}

\begin{remark}
    When $G$ is additionally assumed to be a Lie group, the homogenous space $G/\Lambda$ is called a \emph{nilmanifold} and translational systems on $G/\Lambda$ are called \emph{nilsystems}.
    In the ergodic structure theory of $\mathbb{Z}^d$-systems, nilsystems are the basic building blocks \cite{host-kra-book}, and play a crucial role in applications in ergodic Ramsey theory on $\mathbb{Z}^d$, for example in the context of sumset problems. 
    It was shown in \cite[\S 2.11]{leibman_one_variable} (see also \cite{malcev} for the case that $G$ is connected and simply connected) that every lattice $\Lambda$ in a nilpotent Lie group is compatible with the lower central series filtration.
    The proof relies heavily on the Lie structure, utilizing a basis from the associated finite-dimensional Lie algebra and the fact that $G$ has finitely many connected components to produce a finite collection of elements whose $d$th power belongs to $\Lambda$ for some $d \in \N$ and that together generate the group $G$.
    In the much wider generality where $G$ is assumed only to be locally compact and Polish, we do not know if every lattice is compatible with the lower central series filtration or even if there always exists some finite degree filtration $G_{\bullet}$ that is compatible with a given lattice.
\end{remark}

The motivation for studying translational $\Gamma$-systems comes from a recent structure theorem for Conze--Lesigne systems \cite{jst}.
A measure-preserving dynamical $\Gamma$-system is a \emph{Conze--Lesigne system} if it is isomorphic to its second Host--Kra factor.
(For a definition of the Host--Kra factors, see Section \ref{sec: useful facts} below.)

\begin{theorem}[{\cite[Theorem 1.8]{jst}}] \label{thm: structure theorem}
    Let $\Gamma$ be a countable discrete abelian group, and suppose $(X, \Sigma_X, \mu_X, T_X)$ is an ergodic $\Gamma$-system. 
    The following are equivalent:
    \begin{enumerate}[(i)]
        \item $(X, \Sigma_X, \mu_X, T_X)$ is a Conze--Lesigne $\Gamma$-system.
        \item $(X, \Sigma_X, \mu_X, T_X)$ is measurably isomorphic to an inverse limit of degree 2 translational $\Gamma$-systems $(G_n/\Lambda_n,\Sigma_{G_n/\Lambda_n},\mu_{G_n/\Lambda_n},T_{G_n/\Lambda_n})$. 
    \end{enumerate}
\end{theorem}

Let \(\Gamma\) be a countable discrete abelian group and let $(G/\Lambda,\Sigma_{G/\Lambda},\mu_{G/\Lambda},T_{G/\Lambda})$ be a translational $\Gamma$-system. Given a point $x=g\Lambda\in G/\Lambda$, we denote by $\O(x)$ the orbit closure $\overline{\{T^\gamma_{G/\Lambda} x \colon \gamma\in\Gamma\}}$ of $x$ under the action of $\Gamma$. 
The first main contribution of this paper is an equidistribution result for translational systems (Theorem \ref{thm: main} below) providing a description of the orbit closure $\O(x)$ for each point $x \in X$ in the case that the translational system is of degree 2.
It also describes the distribution of $T^\gamma_{G/\Lambda}x$ in its orbit closure $\O(x)$, and for this we recall the definition of well-distribution\footnote{There is a subtle distinction between the terms "equidistribution" and "well-distribution." The former often refers to the weaker notion of \emph{uniform distribution}, concerning the property \eqref{eq: equidistribution weak*} for a fixed F{\o}lner sequence $(\Phi_N)$ such as $\Phi_N = \{1, 2, \ldots, N\}$ in the group $\Gamma = \Z$.}. 
A sequence $(\Phi_N)$ of finite subsets of $\Gamma$ is called a \emph{F{\o}lner sequence} if it is asymptotically invariant in the sense that
\begin{equation*}
	\lim_{N \to \infty} \frac{|\Phi_N \cap (\Phi_N + \gamma)|}{|\Phi_N|} = 1
\end{equation*}
for every $\gamma \in \Gamma$. 
A $\Gamma$-sequence $(x_{\gamma})_{\gamma \in \Gamma}$ in a compact metric space $X$ is \emph{well-distributed} with respect to a Borel--Radon probability measure $\mu$ on $X$ if
\begin{equation} \label{eq: equidistribution weak*}
	\lim_{N \to \infty} \frac{1}{|\Phi_N|} \sum_{\gamma \in \Phi_N} \delta_{x_{\gamma}} = \mu
\end{equation}
in the weak* topology for every F{\o}lner sequence $(\Phi_N)$ in $\Gamma$. 
That is, for every F{\o}lner sequence $(\Phi_N)$ in $\Gamma$ and every continuous function $f \in C(X)$,
\begin{equation*}
	\lim_{N \to \infty} \frac{1}{|\Phi_N|} \sum_{\gamma \in \Phi_N} f(x_{\gamma}) = \int_X f~d\mu.
\end{equation*}

\begin{theorem} \label{thm: main}  
Let \(\Gamma\) be a countable discrete abelian group, and let a degree $2$ translational $\Gamma$-system $(G/\Lambda,\Sigma_{G/\Lambda},\mu_{G/\Lambda},T_{G/\Lambda})$ be given. 
For every \( x \in G/\Lambda \) there exists a closed subgroup \( H \leq G \) such that \( \mathcal{O}(x) = Hx \).  
Moreover, the sequence \( \left( T^\gamma_{G/\Lambda} x \right)_{\gamma \in \Gamma} \) is well-distributed in \( Hx \) with respect to the unique \( H \)-invariant probability measure on \( Hx \).  
\end{theorem}  

In the case where \( G/\Lambda \) is a nilmanifold with $G$ connected and simply connected and \( \Gamma = \mathbb{Z} \), Theorem \ref{thm: main} is a special case of a theorem of Ratner \cite{ratner} (see, e.g., \cite[Corollary 2.16.21]{tao-poincare} for a proof of this special case where $G$ is $s$-step nilpotent for arbitrary $s$) and also follows from results of Lesigne in \cite[\S 2]{lesigne_equi}. Leibman extended these results to polynomial orbits of \( \mathbb{Z}^d \)-actions on $s$-step nilmanifolds without the connectedness assumptions, see \cite{leibman, leibman1}.  To our knowledge, Theorem \ref{thm: main} provides the first equidistribution result for locally compact nilpotent non-abelian groups that are not Lie groups and for the action of arbitrary countable discrete abelian groups. 

As a first immediate application of Theorem \ref{thm: main}, we obtain the following convergence result. 

\begin{corollary}\label{cor: convergence}
    Let $\Gamma$ be a countable discrete abelian group such that $[\Gamma:6\Gamma] < \infty$, and let $(\Phi_N)$ be a F{\o}lner sequence in $\Gamma$. 
    Let $(G/\Lambda,\Sigma_{G/\Lambda},\mu_{G/\Lambda},T_{G/\Lambda})$ be a degree 2 ergodic translational $\Gamma$-system.
    Let $f_1, f_2, f_3 \in L^{\infty}(G/\Lambda)$, then for almost all $x \in G$
\[
\lim_{N \to \infty} \frac{1}{|\Phi_N|} \sum_{\gamma\in \Gamma} T_{G/\Lambda}^\gamma f_1(x\Lambda) \cdot T_{G/\Lambda}^{2\gamma} f_2(x\Lambda) \cdot T_{G/\Lambda}^{3\gamma} f_3(x\Lambda)
\]
\[
= \int_{[G,G]/[\Lambda,\Lambda]} \int_{G/\Lambda} f_1(xy\Lambda) \cdot f_2(xy^2 z \Lambda) \cdot f_3(xy^3 z^3 \Lambda) \, d\mu_{G/\Lambda}(y) \, d\mu_{[G,G]/[\Lambda,\Lambda]}(z).
\]
\end{corollary}

\begin{proof}
The proof follows exactly the same construction as that of Lesigne in \cite{lesigne_ww} (see \cite[\S 2]{ziegler1} for an abridged version) and constructs a suitable product system that is isomorphic to the orbit closure of \( x \). The main work lies in proving the unique ergodicity of this product system, which, however, follows immediately from Theorem \ref{thm: main} and Theorem \ref{thm: ue} below.
\end{proof}

We note that when $[\Gamma : 6\Gamma] < \infty$, for an arbitrary ergodic measure-preserving $\Gamma$-system $(X, \Sigma_X, \mu_X, T_X)$, limits of multiple ergodic averages of the form
\begin{equation} \label{eq: triple ergodic average}
    \frac{1}{|\Phi_N|} \sum_{\gamma \in \Phi_N} T_X^{\gamma}f_1 \cdot T_X^{2\gamma}f_2 \cdot T_X^{3\gamma}f_3
\end{equation}
are controlled by the Conze--Lesigne factor in the sense that
\begin{multline*}
    \lim_{N \to \infty} \left\| \frac{1}{|\Phi_N|} \sum_{\gamma \in \Phi_N} T_X^{\gamma}f_1 \cdot T_X^{2\gamma}f_2 \cdot T_X^{3\gamma}f_3 \right. \\
    \left. - \frac{1}{|\Phi_N|} \sum_{\gamma \in \Phi_N} T_X^{\gamma}\E[f_1 \mid Z_2] \cdot T_X^{2\gamma}\E[f_2 \mid Z_2] \cdot T_X^{3\gamma}\E[f_3 \mid Z_2] \right\|_{L^2(\mu_X)} = 0
\end{multline*}
for all $f_1, f_2, f_3 \in L^{\infty}(\mu_X)$ and every F{\o}lner sequence $(\Phi_N)$; see \cite[Theorem 6.8]{abb} and \cite[Proposition 2.3]{shalom}.
Corollary \ref{cor: convergence} together with Theorem \ref{thm: structure theorem} thus yields a new proof of $L^2$ convergence of averages of the form \eqref{eq: triple ergodic average} with additional information about the limit.

In the case of $\Gamma=\mathbb{Z}$ and $G/\Lambda$ is a $2$-step nilmanifold, the result in Corollary \ref{cor: convergence} was established by Lesigne \cite{lesigne_ww}. It was generalized to arbitrary $k$-step nilmanifolds by Ziegler \cite{ziegler1}. For $\Gamma=\mathbb{F}_p^\omega$, the infinite direct sum of a finite field of prime characteristic, limit formulas for multiple ergodic averages were obtained in \cite{btz2}.

\subsection{Triple restricted sumsets in abelian groups}

\begin{definition}
    Let $\Gamma$ be a countable discrete abelian group, let $B \subseteq \Gamma$, and let $k \in \N$.
    The \emph{$k$-fold restricted sumset of $B$} is the set
    \begin{equation*}
        B^{\oplus k} = \underbrace{B \oplus \dots \oplus B}_{k~\text{times}} = \left\{ b_1 + \dots + b_k : b_1, \dots, b_k \in B~\text{distinct} \right\}.
    \end{equation*}
\end{definition}

Resolving a conjecture of Erd\H{o}s, Kra--Moreira--Richter--Robertson \cite{kmrr_B+B} proved that every subset of $\N$ with positive upper Banach density\footnote{The upper Banach density of a subset $A$ of an abelian group $\Gamma$ is given by $d^*(A) = \sup_{\Phi} \limsup_{N \to \infty} \frac{|A \cap \Phi_N|}{|\Phi_N|}$, where the supremum is taken over all F{\o}lner sequences $\Phi = (\Phi_N)$ in $\Gamma$.} contains a shift of a set $B \oplus B$ for some infinite set $B$.
This work was recently extended by the same authors to produce shifts of sumsets $B^{\oplus k}$ for arbitrary $k$ inside of sets of positive density in the integers: 

\begin{theorem}[{\cite[Theorem 1.1]{kmrr_finite_sums}}] \label{thm:k-fold sumset integers}
    Fix $k \in \N$.
    If $A \subseteq \N$ has positive upper Banach density, then there exists a shift $t \in \Z$ and an infinite set $B \subseteq \N$ such that for all $1 \le m \le k$, one has $B^{\oplus m} \subseteq A-t$.
\end{theorem}

The $k=2$ case of Theorem \ref{thm:k-fold sumset integers} was generalized to abelian groups and certain classes of amenable groups in \cite{cm}, but for $k \ge 3$, Theorem \ref{thm:k-fold sumset integers} is known only for the integers.
As an application of our new equidistribution theorem, we give a proof of the $k=3$ case of Theorem \ref{thm:k-fold sumset integers} for abelian groups (under a technical assumption on the group that turns out to be necessary).

\begin{theorem} \label{thm:3-fold sumset}
	Let $\Gamma$ be a countably infinite abelian group, and suppose $[\Gamma:6\Gamma] < \infty$.
	If $A \subseteq \Gamma$ and $d^*(A) > 0$, then there exists $t \in \Gamma$ and an infinite set $B \subseteq \Gamma$ such that
	\begin{equation*}
		B \cup (B \oplus B) \cup (B \oplus B \oplus B) \subseteq A - t.
	\end{equation*}
\end{theorem}

If one is interested only in the triple sumset $B \oplus B \oplus B$, then the translate $t$ can be chosen as one of a given list of coset representatives of the subgroup $3\Gamma$.
That is, if $\Gamma = 3\Gamma + \{x_1, \dots, x_n\}$ and $d^*(A) > 0$, then there exists $i \in \{1, \dots, n\}$ such that $A - x_i$ contains an infinite triple restricted sumset $B \oplus B \oplus B$.
To see this, suppose we are given (by Theorem \ref{thm:3-fold sumset}) an infinite set $B_0 \subseteq \Gamma$ and $t \in \Gamma$ such that $B_0 \oplus B_0 \oplus B_0 \subseteq A - t$.
We may write $t = 3s + x_i$ for some $s \in \Gamma$ and $i \in \{1, \dots, n\}$.
Then for $B = (B_0 + s)$, we have $B \oplus B \oplus B = B_0 \oplus B_0 \oplus B_0 + 3s \subseteq A - (t - 3s) = A - x_i$.
This leads to the following corollary.

\begin{corollary}
	Let $\Gamma$ be a countably infinite abelian group, and suppose $[\Gamma:6\Gamma] < \infty$.
	Suppose $A \subseteq \Gamma$ and $d^*(A \cap 3\Gamma) > 0$.
	Then there exists an infinite set $B \subseteq \Gamma$ such that $B \oplus B \oplus B \subseteq A$.
\end{corollary}

\begin{remark}
	Note that the subgroup $3\Gamma$ has density $\frac{1}{[\Gamma:3\Gamma]}$ with respect to every F{\o}lner sequence in $\Gamma$.
	Therefore, the condition $d^*(A \cap 3\Gamma) > 0$ is automatically satisfied whenever $d^*(A) > 1 - \frac{1}{[\Gamma:3\Gamma]}$.
\end{remark}

\begin{proof}
	Let $n = [\Gamma:3\Gamma]$.
	Let $x_0 = 0$, and let $x_1, \dots, x_{n-1} \in \Gamma$ be representatives of the nonzero cosets mod $3\Gamma$ so that $\Gamma = 3\Gamma + \{x_0, x_1, \dots, x_{n-1}\}$.
	By Theorem \ref{thm:3-fold sumset} and the discussion in the paragraph immediately afterwards, there exists $i \in \{0, 1, \dots, n-1\}$ and an infinite set $B \subseteq \Gamma$ such that $B \oplus B \oplus B \subseteq (A \cap 3\Gamma) - x_i$.
	We want to show $i = 0$.
	Since distinct cosets are disjoint, it suffices to show $(B \oplus B \oplus B) \cap 3\Gamma \ne \emptyset$.
	But by the pigeonhole principle, there exist distinct elements $b_1, b_2, b_3 \in B$ that are all congruent mod $3\Gamma$, so $b_1 + b_2 + b_3 \in (B \oplus B \oplus B) \cap 3\Gamma$.
\end{proof}

Theorem \ref{thm:3-fold sumset} is optimal in the following very strong sense:

\begin{proposition} \label{prop:counterexample}
    Let $\Gamma$ be a countably infinite abelian group, and suppose $[\Gamma:6\Gamma] = \infty$.
    Then for any $\eps > 0$ and any F{\o}lner sequence $\Phi=(\Phi_N)$ in $\Gamma$, there exists $A \subseteq \Gamma$ with $\underline{d}_{\Phi}(A) > 1 - \eps$ satisfying the following property: for any $t \in \Gamma$ and any $B \subseteq \Gamma$, if $B \oplus B \oplus B \subseteq A - t$, then $B$ is finite.
\end{proposition}

\begin{proof}
    First, we note that $[\Gamma:6\Gamma] \le [\Gamma:2\Gamma][\Gamma:3\Gamma]$.
    Indeed, if $y_1, y_2, \dots$ are representatives of the cosets of $2\Gamma$ and $z_1, z_2, \dots$ are representatives of the cosets of $3\Gamma$, then we may write an arbitrary element $x \in \Gamma$ as $x = 2u + y_i$ for some $i$ and then $u = 3v + z_j$ for some $j$, whence $x = 6v + (y_i + 2z_j)$, so $\{y_i + 2z_j\}$ represents all cosets of $6\Gamma$.
    Therefore, we have $[\Gamma:2\Gamma] = \infty$ or $[\Gamma:3\Gamma] = \infty$.

    If $[\Gamma:3\Gamma] = \infty$, then the proposition follows from \cite[Theorem 1.3 and Theorem 5.1]{counterexamples}.

    If $[\Gamma:2\Gamma] = \infty$, then applying \cite[Theorem 1.2 and Theorem 5.1]{counterexamples}, we may find a set $A \subseteq \Gamma$ with $\underline{d}_{\Phi}(A) > 1 - \eps$ such that for any $t \in \Gamma$,
    \begin{equation*}
        M_t = \sup \left\{ |B| : B \oplus B \subseteq A - t \right\} < \infty.
    \end{equation*}
    Suppose $B \subseteq \Gamma$ such that $B \oplus B \oplus B \subseteq A - t$.
    Pick $b \in B$, and let $B' = B \setminus \{b\}$.
    Then $b + B' \oplus B' \subseteq B \oplus B \oplus B$, so $B' \oplus B' \subseteq A - (b+t)$.
    Thus, $|B| = |B'| + 1 \le M_{b+t} < \infty$.
\end{proof}

The counterexample in Proposition \ref{prop:counterexample} can be generalized to $k$-fold sumsets under the condition $[\Gamma:k!\Gamma] = \infty$.
There are no other known obstacles to obtaining $k$-fold sumsets in sets of positive density, so we make the following conjecture as an extension of Theorem \ref{thm:k-fold sumset integers} to abelian groups:

\begin{conjecture} \label{conj:k-fold sumset}
	Let $k \in \N$.
	Let $\Gamma$ be a countably infinite abelian group, and suppose $[\Gamma:k!\Gamma] < \infty$.
	If $A \subseteq \Gamma$ and $d^*(A) > 0$, then there exists $t \in \Gamma$ and an infinite set $B \subseteq \Gamma$ such that
	\begin{equation*}
		B^{\oplus m} \subseteq A - t
	\end{equation*}
	for every $1 \le m \le k$.
\end{conjecture}

There are two key ingredients in the proof method of Kra--Moreira--Richter--Robertson \cite{kmrr_finite_sums} for establishing Theorem \ref{thm:k-fold sumset integers} (the $\Gamma = \Z$ case of Conjecture \ref{conj:k-fold sumset}) that are not currently available for infinitely generated abelian groups.
The first of these ingredients is the structure theorem of Host and Kra, stating that the Host--Kra factors of an ergodic measure-preserving $\Z$-system are inverse limits of nilsystems.
The second is the equidistribution theorem of Leibman \cite{leibman} for orbits in nilsystems.

For infinitely generated abelian groups, a structure theorem is only known in general for order 2 (Conze--Lesigne) systems \cite{jst}, and the present paper provides the necessary equidistribution result to carry out the method of Kra--Moreira--Richter--Robertson in this case.
In order to address Conjecture \ref{conj:k-fold sumset} for $k \ge 4$, one needs to either come up with a new method of proof or to establish a structure theorem for systems of order $3$ and higher with an accompanying equidistribution result.

\subsection*{Acknowledgments}

EA was supported by Swiss National Science Foundation grant TMSGI2-211214. AJ was funded by the Deutsche Forschungsgemeinschaft (DFG, German Research Foundation) under Germany's Excellence Strategy -- EXC-2047 -- 390685813 and its Heisenberg Programme --  547294463.
We thank Trist\'{a}n Radi\'{c} for sharing with us Example \ref{eg: skew-product} to highlight a small error in an earlier version of the paper. We are grateful to the anonymous referees for their careful review of the manuscript and for their suggestions and corrections, which helped improve its quality. 

\part{Proof of the equidistribution result}\label{sec2}

We now describe the broad strategy of proof of Theorem \ref{thm: main}. 

Throughout Part \ref{sec2}, we fix an arbitrary countable discrete abelian group $\Gamma$. Let $(G/\Lambda, \Sigma_{G/\Lambda}, \mu_{G/\Lambda}, T_{G/\Lambda})$ be a translational system of degree 2.

First, using the filtration $G_{\bullet}$, we show that the topological system $(G/\Lambda, T_{G/\Lambda})$ can be built as a tower of \emph{group extensions}
\begin{equation*}
    G/\Lambda \to G/G_2\Lambda \to \{\cdot\}
\end{equation*}
so that $(G/\Lambda, T_{G/\Lambda})$ is topologically \emph{distal}.
It then follows by a general property of distal systems that each orbit closure $\O(x)$ for $x \in G/\Lambda$ is \emph{minimal}.
This first step is carried out in Section \ref{sec: distal}, where the relevant terms from topological dynamics are more fully explained.

Second, in Section \ref{sec: UE}, we prove that the minimal components $\O(x)$, $x \in G/\Lambda$, are \emph{uniquely ergodic}.
For this, we adapt an argument of Lesigne \cite{lesigne_ww} by decomposing functions on $G/\Lambda$ according to the eigenspaces of the action of $G_2/\Lambda_2$ on $G/\Lambda$ and then applying a version of the Wiener--Wintner theorem to prove convergence of certain ergodic averages.

Combining the first two steps, we produce a description of the ergodic decomposition of $(G/\Lambda, \Sigma_{G/\Lambda}, \mu_{G/\Lambda}, T_{G/\Lambda})$ in Section \ref{sec: ergodic decomp}.

These initial steps closely follow existing methods developed for describing orbits in nilmanifolds.
The remainder of the proof, however, departs substantially from previous methods.
This is out of necessity.
In the nilmanifold case, the standard approach is to reduce to the case that $G$ is generated by the connected component of the identity element $G^0$ together with the group $\phi(\Gamma)$.
This limited form of connectedness is then used in conjunction with finite-dimensionality of the Lie group $G$ in order to describe orbit closures through an inductive argument.
For example, in \cite{leibman}, connectedness (in the form $G = \langle G^0, \phi(\Gamma) \rangle$) is used to show that $T_{G/\Lambda}$ is ergodic if and only if $T_{G/[G,G]\Lambda}$ is ergodic.
(The connectedness assumption cannot be dropped in this statement, as illustrated by \cite[Exemple 1]{lesigne_equi}.)
This then enables an induction argument that goes roughly as follows (see \cite[Theorem 2.21]{leibman}).
Assume $T_{G/\Lambda}$ is not ergodic, and assume $x_0 = \Lambda \in G/\Lambda$.
Then the group rotation $T_{G/[G,G]\Lambda}$ is not ergodic, so the orbit of the point $[G,G]x_0$ in $G/[G,G]\Lambda$ is a proper subgroup $Z \le G/[G,G]\Lambda$.
The preimage $G'$ of $Z$ under the factor map $G \to G/[G,G]\Lambda$ is then a closed subgroup of $G$ with $\phi(\Gamma) \le G'$.
Crucially, $G'$ is a Lie group of dimension strictly less than the dimension of $G$, so after finitely many iterations, the resulting subnilmanifold $G'/\Lambda'$ must be ergodic.
For the setting of translational systems, arguments relying on connectedness and finite-dimensionality meet a serious obstruction: the group $G$ may be a totally disconnected group.
For example, a natural analogue of the Heisenberg nilmanifold when the acting group $\Gamma$ is a countable vector space over a finite field $\mathbb{F}_p$ is the space $G/\Lambda$ with
\begin{align*}
    G & = \left\{ \left( \begin{array}{ccc} 1 & x & z \\ 0 & 1 & y \\ 0 & 0 & 1\end{array} \right) : x, y, z \in \mathbb{F}_p((t^{-1})) \right\}
\intertext{and}
    \Lambda & = \left\{ \left( \begin{array}{ccc} 1 & a & c \\ 0 & 1 & b \\ 0 & 0 & 1\end{array} \right) : a, b, c \in \mathbb{F}_p[t] \right\},
\end{align*}
which is topologically a Cantor set. Moreover, if $\Gamma$ is not finitely generated, there are
counterexamples that show that finite-dimensionality can fail;
see the example presented after \cite[Conjecture 2.14]{shalom2024}. 

In order to complete the proof of Theorem \ref{thm: main}, we therefore use a new method.
As a pre-processing step, we reduce to the case that $G$ has a more convenient form.
Rather than the assumption $G = \langle G^0, \phi(\Gamma)\rangle$ referenced above, we reduce to the case that $G$ contains a compact abelian Lie group $L$ satisfying $[G,G] \le L \le Z(G)$ and $L \cap \Lambda = \{1\}$.
This reduction is justified in Section \ref{sec: reductions}.

We finish with an ergodic theoretic argument in Section \ref{sec: proof of equidistribution}.
The main idea is to represent $(G/\Lambda, T_{G/\Lambda})$ as a skew-product system on the space $G/L\Lambda \times L$.
Projecting onto the first component $G/L\Lambda$ produces a group rotation, in which we can show that the orbit closure of a point $xL$ is a coset $Z_x$ of a closed subgroup $Z_0 \le G/L\Lambda$.
General results in the theory of group extensions then guarantee the existence of a closed subgroup (the \emph{Mackey group}) $L' \le L$ such that the orbit closure $\O(x)$ can be represented by a skew-product on $Z_x \times L'$.
We then use arguments similar to those appearing in \cite{jst} to relate the skew-product on $Z_x \times L'$ to a translational system $H/(H \cap x\Lambda x^{-1}) \cong Hx$ for some closed subgroup $H \le G$.

\section{Nilpotent translational systems are distal} \label{sec: distal}

The first main step in the proof of Theorem \ref{thm: main} is to show that nilpotent translational systems fall into the larger class of \emph{distal} systems.
In this section, we review the definition and properties of distal systems and prove that nilpotent translational systems are indeed distal.

Let $(X, T_X)$ be a topological dynamical $\Gamma$-system, and fix a metric $d$ on $X$.
The system $(X, T_X)$ is \emph{distal} if for every pair $x, y \in X$ with $x \ne y$, one has $\inf_{\gamma \in \Gamma} d(T^{\gamma}_Xx, T^{\gamma}_Xy) > 0$.
We note that the property of being distal does not depend on the choice of metric $d$.

\subsection{Isometric systems and isometric extensions}

The simplest nontrivial examples of distal systems are provided by \emph{isometric systems}, that is, topological dynamical $\Gamma$-systems $(X,T_X)$, where $T_X$ acts by isometries with respect to some metric\footnote{In any other compatible metric, the action of $T_X$ will be uniformly equicontinuous, so such systems are also sometimes referred to as \emph{equicontinuous systems}.}.
The relative notion of an isometric extension gives a tool for constructing many other distal systems from a given one. 

\begin{definition}[Isometric extension] \label{def: isometric ext}
    Let $\pi\colon (X,T_X)\to (Y,T_Y)$ be an extension of topological dynamical $\Gamma$-systems. 
	We say that $(X,T_X)$ is an \emph{isometric extension} of $(Y,T_Y)$ if for each $y \in Y$, there is a metric $d_y$ on the fiber $\pi^{-1}(\{y\})$ such that:
	\begin{enumerate}
		\item	For any $y \in Y$, any $x_1, x_2 \in \pi^{-1}(\{y\})$, and any $\gamma \in \Gamma$, one has
			\begin{equation*}
				d_{T^{\gamma}_Y y} \left( T_X^{\gamma} x_1, T_X^{\gamma} x_2 \right) = d_y(x_1, x_2).
			\end{equation*}
		\item	The function $d : \bigcup_{y \in Y} (\pi^{-1}(\{y\}) \times \pi^{-1}(\{y\})) \to [0, \infty)$ formed by gluing the metrics $d_y$ is a continuous function on $\{(x_1, x_2) \in X \times X : \pi(x_1) = \pi(x_2)\}$.
		\item	For any $y_1, y_2 \in Y$, the metric spaces $(\pi^{-1}(\{y_1\}), d_{y_1})$ and $(\pi^{-1}(\{y_2\}), d_{y_2})$ are isometric.
	\end{enumerate}
\end{definition}

As promised, isometric extensions preserve distality:

\begin{proposition} \label{prop: isometric ext distal}
	Suppose $\pi : (X,T_X) \to (Y,T_Y)$ is an isometric extension.
	If $(Y,T_Y)$ is distal, then so is $(X,T_X)$.
\end{proposition}

\begin{proof}
	This is a standard fact about distal systems; for convenience and completeness, we include a short proof here.
	Let $d_Y$ be a metric on $Y$, and let $(d_y)_{y \in Y}$ be a family of metrics as in Definition \ref{def: isometric ext}.
	Let $d_X$ be a metric on $X$.
	Let $x_1, x_2 \in X$ with $x_1\neq x_2$. 
	We want to show $\inf_{\gamma \in \Gamma} d_X \left( T^{\gamma}_X x_1, T^{\gamma}_X x_2 \right) > 0$.
	
	If $\pi(x_1) = \pi(x_2) = y$, then $d \left( T^{\gamma}_X x_1, T^{\gamma}_X x_2 \right) = d(x_1, x_2)$ for all $\gamma \in \Gamma$, where $d$ is as in (2) in Definition \ref{def: isometric ext}.
	Since $d$ is continuous on the compact set $K = \{(x_1, x_2) \in X \times X : \pi(x_1) = \pi(x_2)\}$, it is uniformly continuous.
	Hence, there exists $\delta > 0$ such that if $(z_1, z_2), (z'_1, z'_2) \in K$ and $d_X(z_1', z_1) + d_X(z_2', z_2) < \delta$, then $|d(z'_1, z'_2) - d(z_1, z_2)| < d(x_1, x_2)$.
	It follows that $d_X \left( T^{\gamma}_X x_1, T^{\gamma}_X x_2 \right) \ge \delta$ for every $\gamma \in \Gamma$: if not, then taking $(z_1, z_2) = (T^{\gamma}_X x_1, T^{\gamma}_X x_2)$ and $(z'_1, z'_2) = (T^{\gamma}_X x_1, T^{\gamma}_X x_1)$ leads to a contradiction.
	
	Now suppose $\pi(x_1) \ne \pi(x_2)$.
	Let $y_1 = \pi(x_1)$ and $y_2 = \pi(x_2)$.
	Since $(Y,T_Y)$ is distal, we have $\eps = \inf_{\gamma \in \Gamma} d_Y \left( T^{\gamma}_Y y_1, T^{\gamma}_Y y_2 \right) > 0$.
	The map $\pi$ is (uniformly) continuous, so let $\delta > 0$ such that if $z_1, z_2 \in X$ and $d_X(z_1, z_2) < \delta$, then $d_Y(\pi(z_1), \pi(z_2)) < \eps$.
	We then have $d_X \left( T^{\gamma}_X x_1, T^{\gamma}_X x_2 \right) \ge \delta$ for every $\gamma \in \Gamma$.
\end{proof}

\subsection{Rotational systems and group extensions}

Note that every rotational $\Gamma$-system is isometric, since every compact abelian groups admits a translation-invariant metric. Moreover, rotational $\Gamma$-system enjoy many additional convenient properties.
In order to formulate these properties, we recall some terminology from topological dynamics.
Given a topological dynamical $\Gamma$-system $(X, T_X)$, a point $x \in X$ is \emph{transitive} if the orbit of $x$ is dense in $X$, that is, $\overline{\{T_X^{\gamma}x : \gamma \in \Gamma\}} = X$.
The system $(X, T_X)$ is \emph{transitive} if it has a transitive point.
We say that a topological dynamical $\Gamma$-system $(X, T_X)$ is \emph{minimal} if every point $x \in X$ is transitive.
Finally, a topological dynamical $\Gamma$-system $(X, T_X)$ is \emph{uniquely ergodic} if there is a unique $T_X$-invariant Borel probability measure on $X$.

\begin{proposition} \label{prop: Kronecker}
	Let $Z$ be a compact abelian group, let $\psi : \Gamma \to Z$ be a group homomorphism, and let $(Z,\Sigma_Z,\mu_Z,T_Z)$ be the corresponding rotational $\Gamma$-system. 
	The following are equivalent:
	\begin{enumerate}[(i)]
		\item	$(\psi(\gamma))_{\gamma \in \Gamma}$ is dense in $Z$;
		\item	$(\psi(\gamma))_{\gamma \in \Gamma}$ is well-distributed in $Z$ with respect to the Haar measure $\mu_Z$;
		\item	for any $\chi \in \hat{Z} \setminus \{1\}$, there exists $\gamma \in \Gamma$ such that $\chi(\psi(\gamma)) \ne 1$;
		\item	for any $\chi \in \hat{Z} \setminus \{1\}$ and any F{\o}lner sequence $(\Phi_N)$ in $\Gamma$, one has
			\begin{equation*}
				\lim_{N \to \infty} \frac{1}{|\Phi_N|} \sum_{\gamma \in \Phi_N} \chi(\psi(\gamma)) = 0.
			\end{equation*}
		\item	$(Z, T_Z)$ is minimal;
		\item	$(Z, T_Z)$ is uniquely ergodic;
		\item	$(Z, \Sigma_Z,\mu_Z, T_Z)$ is ergodic.
	\end{enumerate}
\end{proposition}

\begin{proof}
These equivalences are well known in the $\Gamma = \Z$ case, with proofs of most of the implications appearing, for example, in \cite[\S 2.6, \S 2.9]{tao-poincare}.
For completenss, we include a short proof in the general case.

First, since the Haar measure $\mu_Z$ has full support, being well-distributed with respect to $\mu_Z$ is a stronger property than being dense in $Z$, so (ii)$\implies$(i).
Next, by the Stone--Weierstrass theorem, linear combinations of characters $\chi \in \hat{Z}$ are uniformly dense in the space $C(Z)$ of continuous functions on $Z$, so (iv)$\implies$(ii).

We now check the implication (i)$\implies$(iii).
Suppose (i) holds.
Let $\chi \in \hat{Z}$ and assume $\chi(\psi(\gamma)) = 1$ for every $\gamma \in \Gamma$.
We will show $\chi = 1$.
Indeed, since $\chi$ is continuous, it is determined by its values on the dense set $\{\psi(\gamma) : \gamma \in \Gamma\}$, so $\chi = 1$.

Now suppose (iii) holds.
We will prove (iv).
Fix a nontrival character $\chi \in \hat{Z} \setminus \{1\}$ and a F{\o}lner sequence $(\Phi_N)$ in $\Gamma$.
We will show that every subsequence of $a_N = \frac{1}{|\Phi_N|} \sum_{\gamma \in \Phi_N} \chi(\psi(\gamma))$ has a further subsequence converging to 0, whence $\lim_{N \to \infty} a_N = 0$.
Let $N_1 < N_2 < \ldots$.
The sequence $(a_{N_j})_{j \in \N}$ is bounded, so it has a convergent subsequence $(a_{N_{j_k}})_{k \in \N}$.
By (iii), let $\gamma_0 \in \Gamma$ such that $\chi(\psi(\gamma_0)) \ne 1$.
Then applying the asymptotic invariance property of the F{\o}lner sequence $(\Phi_N)$, we have
\begin{multline*}
    \lim_{k \to \infty} a_{N_{j_k}} = \lim_{N \to \infty} \frac{1}{|\Phi_{N_{j_k}}|} \sum_{\gamma \in \Phi_{N_{j_k}}} \chi(\psi(\gamma)) \\
    = \lim_{N \to \infty} \frac{1}{|\Phi_{N_{j_k}}|} \sum_{\gamma \in \Phi_{N_{j_k}}} \chi(\psi(\gamma + \gamma_0)) \\
    = \chi(\psi(\gamma_0)) \cdot \lim_{k \to \infty} a_{N_{j_k}}.
\end{multline*}
That is, $(1 - \chi(\psi(\gamma_0)) \cdot \lim_{k \to \infty} a_{N_{j_k}} = 0$, so $\lim_{k \to \infty} a_{N_{j_k}} = 0$ as claimed.

We have thus far established the equivalence between properties (i)--(iv).
We now proceed to show the equivalence with the dynamical properties (v)--(vii).

The orbit of the point $0 \in Z$ under $T_Z$ is $(\psi(\gamma))_{\gamma \in Z}$, so (v)$\implies$(i).
For the converse, assuming (i) holds, we can write the orbit of an arbitrary point $z \in Z$ under $T_Z$ as $(z + \psi(\gamma))_{\gamma \in \Gamma}$, which is a shift of a dense set and hence dense, since the topology on $Z$ is shift-invariant.
Therefore, $(Z,T_Z)$ is minimal.

The Haar measure $\mu_Z$ is always an invariant measure for $(Z, T_Z)$, so (vi)$\implies$(vii).
Moreover, since $\mu_Z$ has full support, (vii)$\implies$(v).

Let us check the remaining implication (i)$\implies$(vi).
Assume (i) holds, and let $\mu$ be a $T_Z$-invariant Borel probability measure on $Z$.
Let $t \in Z$.
By (i), there is a sequence $(\gamma_n)_{n \in \N}$ in $\Gamma$ such that $\lim_{n \to \infty} \psi(\gamma_n) = t$.
Hence, by $T_Z$-invariance of $\mu$, we have
\begin{equation*}
    \int_Z f(z)~d\mu(z) = \int_Z f(z+\psi(\gamma_n))~d\mu(z)
\end{equation*}
for every continuous function $f \in C(Z)$ and every $n \in \N$.
Taking a limit as $n \to \infty$,
\begin{equation*}
    \int_Z f(z)~d\mu(z) = \int_Z f(z+t)~d\mu(z)
\end{equation*}
for every $f \in C(Z)$.
Thus, $\mu$ is a shift-invariant probability measure on $Z$, so $\mu = \mu_Z$.
\end{proof}

In general, as the following proposition shows, isometric systems decompose as unions of rotational systems, so the behavior of isometric systems is well-understood.

\begin{proposition}\label{prop-decomp-isometric}
	Let $(X,T_X)$ be an isometric topological dynamical $\Gamma$-system.
	Then $X$ decomposes as a disjoint union of minimal systems $X = \bigsqcup_{i \in I} X_i$.
	Moreover, each minimal system $(X_i, T_{X_i})$ is isomorphic to a rotational $\Gamma$-system.
\end{proposition}

\begin{proof}
    This follows from combining \cite[Proposition 2.6.7, Proposition 2.6.9]{tao-poincare}. These propositions are proved in \cite{tao-poincare} for $\mathbb{Z}$-systems, but the proofs extend to cover the general case of $\Gamma$-systems for an arbitrary countable discrete abelian group $\Gamma$. 
\end{proof}

Rotational systems can also be relativized, leading to the notion of a group extension.

\begin{definition}[Group extension] \label{def: group ext}
	Let $(X,T_X)$ be a topological dynamical $\Gamma$-system.
	Suppose $K$ is a compact metrizable group acting continuously on $X$ by homeomorphisms such that $kT^{\gamma}_Xx = T^{\gamma}_Xkx$ for every $k \in K$, $\gamma \in \Gamma$, and $x \in X$.
	Let $Y = K \backslash X = \{Kx : x \in X\}$.
	Then $Y$ is compact and metrizable, and $T_X$ induces an action $T_Y$ on $Y$ given by $T^{\gamma}_Y Kx = K(T^{\gamma}_X x)$.
	We say that $(X,T_X)$ is a \emph{group extension} of $(Y, T_Y)$ by $K$.
	If additionally the action of $K$ on $X$ is free, then we say that $(X,T_X)$ is a \emph{free group extension} of $(Y, T_Y)$.
\end{definition}

\begin{proposition} \label{prop: group exts are isometric}
	Every free group extension is an isometric extension.
\end{proposition}

\begin{proof}
    Let $(X,T_X)$, $K$, and $(Y,T_Y)$ be as in Definition \ref{def: group ext}, and assume the action of $K$ on $X$ is free.
    
    Fix a $K$-invariant metric $d_K$ on $K$.
    We define $d : \bigcup_{y\in Y} (\pi^{-1}(\{y\}) \times \pi^{-1}(\{y\})) \to [0,\infty)$ by
    \begin{equation*}
        d(x_1,x_2) = d_K(k,1_K),
    \end{equation*}
    where $k \in K$ is such that $kx_1 = x_2$ and $1_K$ is the identity element of $K$.

    Let us check that $d$ satisfied property (1), (2), and (3) from Definition \ref{def: isometric ext}.
    
    Suppose $Kx_1 = Kx_2$, say with $kx_1 = x_2$.
    Let $\gamma \in \Gamma$.
    Then since the action of $K$ commutes with $T_X$, we have $kT_X^{\gamma}x_1 = T_X^{\gamma}x_2$, so
    \begin{equation*}
        d(T_X^{\gamma}x_1, T_X^{\gamma}x_2) = d_K(k,1_K) = d(x_1,x_2).
    \end{equation*}
    That is, (1) holds.

    We now check that $d$ is continuous.
    Let $D = \{(x_1, x_2) \in X \times X : Kx_1 = Kx_2\}$ be the domain of $d$.
    Since $K$ is a compact group and acts continuously on $X$, there is a $K$-invariant metric $d_X$ on $X$.
    Let $(x_1, x_2) \in D$ and $\eps > 0$ be given.
    The $K$-orbit $Kx_1$ is homeomorphic to $K$ using the map $kx_1 \mapsto k$, since $K$ acts continuously and freely on $X$.
    Let $\delta > 0$ such that if $k \in K$ and $d_X(kx_1, x_1) < \delta$, then $d_K(k,1_K) < \eps$.
    Suppose $(x'_1, x'_2) \in D$ and
    \begin{equation*}
        d_X(x_1,x_1') + d_X(x_2,x_2') < \delta.
    \end{equation*}
    Let $k, k' \in K$ such that $kx_1 = x_2$ and $k'x'_1 = x'_2$.
    Then by $K$-invariance of $d_X$ and the triangle inequality
    \begin{multline*}
        d_X(k^{-1}k'x_1, x_1) = d_X(k'x_1, kx_1) \\
        \le d_X(k'x_1, k'x'_1) + d_X(k'x'_1, kx_1) \\
         = d_X(x_1, x'_1) + d_X(x_2, x'_2) < \delta,
    \end{multline*}
    so by the choice of $\delta$, we have $d_K(k,k') = d_K(k^{-1}k',1_K) < \eps$.
    Therefore,
    \begin{equation*}
        |d(x'_1, x'_2) - d(x_1,x_2)| = |d_K(k',1_K) - d_K(k,1_K)| \le d_K(k,k') < \eps
    \end{equation*}
    by the reverse triangle inequality.
    This proves that $d$ is continuous.

    Finally, suppose $y_1, y_2 \in Y$.
    Pick $x_1 \in \pi^{-1}(\{y_1\})$ and $x_2 \in \pi^{-1}(\{y_2\})$.
    Define $\xi : \pi^{-1}(\{y_1\}) \to \pi^{-1}(\{y_2\})$ by $\xi(kx_1) = kx_2$ for $k \in K$.
    The map $\xi$ is well-defined, since $K$ acts freely on $X$.
    We claim that $\xi$ is an isometry.
    Indeed, for any $k, k' \in K$, we have
    \begin{equation*}
        d_{y_2}(\xi(kx_1), \xi(k'x_1)) = d_{y_2}(kx_2, k'x_2) = d_K(k,k') = d_{y_1}(kx_1,k'x_1).
    \end{equation*}
\end{proof}

The next result follows from a straightforward extension of the proof for $\mathbb{Z}$-systems as given in \cite[Lemma 2.6.22]{tao-poincare}. 

\begin{proposition}
	Suppose $\pi : (X,T_X) \to (Y,T_Y)$ is an isometric extension.
	Suppose also that $(X,T_X)$ is minimal.
	Then there is a group extension $(Z, T_Z)$ of $(Y,T_Y)$ by a compact group $K$ and a closed subgroup $H \le K$ such that $(X,T_X)$ is isomorphic to $(H \backslash Z, T_{H \backslash Z})$, and the diagram below commutes:
	\begin{equation*}
		\begin{tikzcd}
			Z \arrow[r] \arrow[dr] & X \cong H \backslash Z \arrow[d, "\pi"] \\
			 & Y \cong K \backslash Z
		\end{tikzcd}
	\end{equation*}
\end{proposition}

\subsection{Semi-simplicity}

We have the following generalization of Proposition \ref{prop-decomp-isometric}. 

\begin{theorem}[Semi-simplicity of distal systems] \label{thm: semi-simple}
	Suppose $(X,T_X)$ is distal.
	Then $X$ decomposes as a disjoint union of minimal systems $X = \bigsqcup_{i \in I} X_i$.
\end{theorem}

\begin{proof}
	A proof for $\Z$-actions is given in \cite[Theorem 3.2]{furstenberg}.
	We give a sketch of the proof for actions of a countable discrete abelian group $\Gamma$ following the same strategy.
	
	Let $G$ be the Ellis enveloping semigroup of $T_X$.
	That is,
	\begin{equation*}
		G = \overline{\{T_X^\gamma : \gamma \in \Gamma\}} \subseteq X^X,
	\end{equation*}
	where the closure is taken in the product topology/topology of pointwise convergence.
	Since $(X,T_X)$ is distal, $G$ is a group by \cite[Theorem 1]{ellis}.
	
	Let $x \in X$.
	Since the map $g \mapsto gx$ is continuous from $G$ to $X$, we have $Gx = \O(x)$.
	Hence, given any $y \in \O(x) = Gx$, the group property implies $\O(y) = Gy = Gx = \O(x)$, so the orbit of every point in $(\O(x), T_{\O(x)})$ is dense, where $T_{\O(x)}$ is the restriction of $T$ onto $\O(x)$. 
	That is, $(\O(x), T_{\O(x)})$ is minimal. 
\end{proof}

\subsection{Distality of translational systems}

We can establish the following description of translational systems.

\begin{proposition} \label{prop: distal}
	A translational $\Gamma$-system $(G/\Lambda, T_{G/\Lambda})$ is distal. 
\end{proposition}

\begin{proof}
    Let $G_{\bullet}$ be a filtration (of degree $d \in \N$) that is compatible with $\Lambda$.
    We claim that $G_i\Lambda$ is a closed subgroup of $G$ for each $i \in \{1,\ldots,d\}$.
    To see this, suppose $(g_{i,n})_{n \in \N}$ is a sequence in $G_i$, $(\lambda_n)_{n \in \N}$ is a sequence in $\Lambda$, and $\lim_{n \to \infty} g_{i,n} \lambda_n = g \in G$.
    Since $G_i/\Lambda_i$ is compact, there exists a subsequence $(g_{i,n_k})_{k \in \N}$ such that $g_{i,n_k}\Lambda_i$ converges to some point $g_i\Lambda_i$ in $G_i/\Lambda_i$.
    Hence, there exists a sequence $(\lambda_{i,k})_{k \in \N}$ in $\Lambda_i$ so that $g_i = \lim_{k \to \infty} g_{i,n_k}\lambda_{i,k}$.
    Let $\lambda = g_i^{-1}g$.
    By continuity of the group operations on $G$, we have
    \begin{equation*}
        \lambda = \lim_{k \to \infty} \lambda_{i,k}^{-1} g_{i,n_k}^{-1} g_{i,n_k} \lambda_{n_k} = \lim_{k \to \infty} \lambda_{i,k}^{-1} \lambda_{n_k}.
    \end{equation*}
    But $\Lambda$ is a closed subgroup of $G$ and $\lambda_{i,k}^{-1} \lambda_{n_k} \in \Lambda$ for each $k \in \N$, so $\lambda \in \Lambda$ and $g = g_i\lambda \in G_i\Lambda$.
    
	We now let $Y_i = G_i \backslash X = G/G_i\Lambda$ for each $i \in \{1, \ldots, d\}$.
	Then $(G/\Lambda, T_{G/\Lambda})$ is obtained as a tower $G/\Lambda = Y_{d+1} \to Y_d \to \dots \to Y_2 \to Y_1 = \{\cdot\}$.
	By Proposition \ref{prop: isometric ext distal}, it suffices to show that each of the extensions $\pi_i : (Y_{i+1}, T_{Y_{i+1}}) \to (Y_i, T_{Y_{i}})$ is isometric. 
	We claim $(Y_{i+1}, T_{Y_{i+1}})$ is a free extension of $(Y_i, T_{Y_{i}})$ by the compact abelian group $G_i/G_{i+1}\Lambda_i$ so that $\pi_i$ is an isometric extension by Proposition \ref{prop: group exts are isometric}.
	Indeed, $K_i = G_i/G_{i+1}\Lambda_i$ acts on $(Y_{i+1}, T_{Y_{i+1}})$ by $(gG_{i+1}\Lambda_i) \cdot y = gy$.
	This action is well defined and free: if $g_1, g_2 \in G_i$ and $y = h G_{i+1}\Lambda \in Y_{i+1}$, then
	\begin{multline*}
        g_1y = g_2y \iff g_1 h G_{i+1}\Lambda = g_2 h G_{i+1} \Lambda \\
        \iff h g_1 \underbrace{[g_1^{-1},h^{-1}]}_{\in G_{i+1}} G_{i+1}\Lambda = h g_2 \underbrace{[g_2^{-1},h^{-1}]}_{\in G_{i+1}} G_{i+1} \Lambda \\
        \iff g_1 G_{i+1}\Lambda = g_2 G_{i+1}\Lambda.
	\end{multline*}
	Moreover, $K_i \backslash Y_{i+1} = G_i \backslash Y_{i+1} = Y_i$.
\end{proof}

\begin{corollary} \label{cor: minimal}
	Let $(G/\Lambda, T_{G/\Lambda})$ be a translational $\Gamma$-system.
	Then for any $x \in G/\Lambda$, the orbit closure $\O(x)$ is minimal.
\end{corollary}

\begin{proof}
	By Proposition \ref{prop: distal} and Theorem \ref{thm: semi-simple}, $G/\Lambda$ decomposes as a disjoint union of minimal systems $G/\Lambda = \bigsqcup_{i \in I} X_i$.
	Taking $i \in I$ such that $x \in X_i$, we have that $\O(x) = X_i$ is minimal.
\end{proof}

\section{Minimal components of translational systems are uniquely ergodic} \label{sec: UE}

The next step in the proof of Theorem \ref{thm: main} is to show that each of the minimal components of $(G/\Lambda, T_{G/\Lambda})$ is uniquely ergodic.

\begin{theorem} \label{thm: ue}
	Let $(G/\Lambda, T_{G/\Lambda})$ be a degree $2$ translational $\Gamma$-system. 
	For any $x \in G/\Lambda$, the orbit closure $(\O(x), T_{\O(x)})$ is uniquely ergodic.
\end{theorem}

\subsection{Unique ergodicity}

Let us first review some basic properties of uniquely ergodic systems.

\begin{proposition} \label{prop: ue characterization}
	A topological dynamical $\Gamma$-system $(X,T_X)$ is uniquely ergodic with unique invariant measure $\mu_X$ if and only if for every continuous function $f \in C(X)$, every $x \in X$, and every F{\o}lner sequence $(\Phi_N)$ in $\Gamma$,
	\begin{equation*}
		\lim_{N \to \infty} \frac{1}{|\Phi_N|} \sum_{\gamma \in \Phi_N} f(T^{\gamma}_Xx) = \int_X f~d\mu_X.
	\end{equation*}
\end{proposition}

\begin{proof}
    This is a standard fact, but we include a short proof for completeness.

    First, assume that $(X, T_X)$ is uniquely ergodic with unique invariant measure $\mu_X$.
    Let $x \in X$, and let $(\Phi_N)$ be a F{\o}lner sequence in $\Gamma$.
    We want to show that
    \begin{equation} \label{eq: weak* convergence}
        \lim_{N \to \infty} \frac{1}{|\Phi_N|} \sum_{\gamma \in \Phi_N} \delta_{T_X^{\gamma}x} = \mu_X,
    \end{equation}
    where the limit is taken in the weak* topology.
    We will establish this convergence by showing that every subsequence of $\mu_N = \frac{1}{|\Phi_N|} \sum_{\gamma \in \Phi_N} \delta_{T_X^{\gamma}x}$ has a further subsequence converging to $\mu_X$.
    Let $(N_j)_{j \in \N}$ be a strictly increasing sequence in $\N$.
    By the Banach--Alaoglu theorem, the sequence of probability measures $(\mu_{N_j})_{j \in \N}$ has a convergent subsequence $(\mu_{N_{j_k}})_{k \in \N}$.
    The limit $\mu = \lim_{k \to \infty} \mu_{N_{j_k}}$ is $T_X$-invariant by the asymptotic invariance property of $(\Phi_N)$, so by unique ergodicity, $\mu = \mu_X$.
    This proves \eqref{eq: weak* convergence}.

    Conversely, suppose that for every continuous function $f \in C(X)$, every $x \in X$, and every F{\o}lner sequence $(\Phi_N)$ in $\Gamma$,
    \begin{equation*}
		\lim_{N \to \infty} \frac{1}{|\Phi_N|} \sum_{\gamma \in \Phi_N} f(T^{\gamma}_Xx) = \int_X f~d\mu_X.
	\end{equation*}
    Let $\mu$ be a $T_X$-invariant Borel probability measure on $X$.
    Then by $T_X$-invariance, we have
    \begin{equation*}
        \int_X f~d\mu = \int_X f(T_X^{\gamma}x)~d\mu(x)
    \end{equation*}
    for every $f \in C(X)$ and $\gamma \in \Gamma$.
    Averaging over $\gamma \in \Phi_N$ for some F{\o}lner sequence $(\Phi_N)$ and taking a limit as $N \to \infty$, we have
    \begin{multline*}
        \int_X f~d\mu = \lim_{N \to \infty} \frac{1}{|\Phi_N|} \sum_{\gamma \in \Phi_N} \int_X f(T_X^{\gamma}x)~d\mu(x) \\
        = \int_X \lim_{N \to \infty} \frac{1}{|\Phi_N|} \sum_{\gamma \in \Phi_N} f(T_X^{\gamma}x)~d\mu(x) = \int_X f~d\mu_X,
    \end{multline*}
    where we have moved the limit inside the integral using the bounded convergence theorem.
    Thus, $\mu = \mu_X$, so $(X,T_X)$ is uniquely ergodic with unique invariant measure $\mu_X$.
\end{proof}

For minimal systems, a simpler criterion was given by Oxtoby (see \cite[Proposition 5.4]{oxtoby}) in the case of $\Z$-actions.
Essentially the same argument works for minimal actions of general countable discrete abelian groups.
For completeness, we record the proof in full generality here.

\begin{lemma} \label{lem: oxtoby}
	Fix a F{\o}lner sequence $(\Phi_N)$ in $\Gamma$.
	A minimal topological dynamical $\Gamma$-system $(X,T_X)$ is uniquely ergodic if and only if
	\begin{equation*}
		\frac{1}{|\Phi_N|} \sum_{\gamma \in \Phi_N} f \left( T^{\gamma}_X x \right)
	\end{equation*}
	converges for every $f \in C(X)$ and every $x \in X$.
\end{lemma}

\begin{proof}
	If $(X,T_X)$ is uniquely ergodic with unique invariant measure $\mu_X$, then
	\begin{equation*}
		\lim_{N \to \infty} \frac{1}{|\Phi_N|} \sum_{\gamma \in \Phi_N} f \left( T^{\gamma}_X x \right) = \int_X{f~d\mu_X}
	\end{equation*}
	for every $f \in C(X)$ and every $x \in X$ by Proposition \ref{prop: ue characterization}. 
	
	For the converse, fix a continuous function $f \in C(X)$.
	For $N \in \N$, put $M_Nf(x) = \frac{1}{|\Phi_N|} \sum_{\gamma \in \Phi_N} f \left( T^{\gamma}_X x \right)$, and let $Mf = \lim_{N \to \infty} M_Nf$, where the limit is taken pointwise. 
	We claim $Mf$ is a constant function.
	
	Suppose for contradiction that $Mf$ is not constant.
	Since $(\Phi_N)$ is a F{\o}lner sequence, $Mf$ is constant on each orbit.
	Then since $(X,T_X)$ is minimal (so that every orbit is dense) it follows that $Mf$ is everywhere discontinuous.
	On the other hand, $Mf$ is a pointwise limit of continuous functions $M_Nf$, so $Mf$ is discontinuous only on a meager set by the Baire--Osgood theorem; see, e.g., \cite[Theorem 7.3]{oxtoby_m&c}.
	This is a contradiction, so $Mf$ must be constant.
	
	We can thus define a measure $\mu_X$ on $X$ by $\int_X{f~d\mu_X} = Mf$.
	Since every point $x \in X$ is generic for $\mu_X$ along $(\Phi_N)$, we conclude that $(X,T_X)$ is uniquely ergodic with $\mu_X$ as the unique invariant measure.
\end{proof}

\subsection{A Wiener--Wintner theorem for abelian groups}

The final ingredient in the proof of Theorem \ref{thm: ue} is the following generalization of the classical Wiener--Wintner theorem (see  \cite[p. 121]{ow} or \cite[Corollary 4.1]{z-k}, for a proof), where we denote by $\hat\Gamma$ the Pontryagin dual of the group $\Gamma$. 

\begin{theorem}[Wiener--Wintner theorem for abelian groups] \label{thm: WW}
	Let $(X, \Sigma_X,\mu_X, T_X)$ be a measure-preserving $\Gamma$-system, and let $(\Phi_N)$ be a tempered\footnote{A F{\o}lner sequence $(\Phi_N)$ is \emph{tempered} if for some \( C > 0 \) and all \( N \)
\begin{equation*}
    \left| \bigcup_{M < N} \Phi_M^{-1} \Phi_N \right| \leq C |\Phi_N|.
\end{equation*}
Tempered F{\o}lner sequences are important for establishing pointwise ergodic theorems for amenable groups and they always exist, for instance as subsequences of any given F{\o}lner sequence; see \cite{lindenstrauss}.} F{\o}lner sequence in $\Gamma$.
	There is a set $X' \subseteq X$ with $\mu_X(X') = 1$ such that for every $\lambda \in \hat{\Gamma}$, every $f \in C(X)$, and every $x \in X'$, the sequence
	\begin{equation*}
		\frac{1}{|\Phi_N|} \sum_{\gamma \in \Phi_N} \lambda(\gamma) f \left( T^{\gamma}_X x \right)
	\end{equation*}
	converges.
\end{theorem}

\subsection{Proof of Theorem \ref{thm: ue}}

	By Corollary \ref{cor: minimal} and Lemma \ref{lem: oxtoby}, it suffices to show that there exists a F{\o}lner sequence $(\Phi_N)$ in $\Gamma$ such that for every $f \in C(G/\Lambda)$ and every $x \in G/\Lambda$, the sequence
	\begin{equation*}
		\frac{1}{|\Phi_N|} \sum_{\gamma \in \Phi_N} f \left( T^{\gamma}_{G/\Lambda} x \right)
	\end{equation*}
	converges.
	We will use an argument due to Lesigne \cite{lesigne_ww} to establish this claim.
	
	Fix a tempered F{\o}lner sequence $(\Phi_N)$.
	By the Stone--Weierstra\ss \, theorem, finite sums of functions $g$ satisfying $g(l x) = \chi(l) g(x)$ for $l \in G_2/\Lambda_2$ and $x \in G/\Lambda$ for some $\chi$ in the Pontryagin dual $\widehat{G_2/\Lambda_2}$ of $G_2/\Lambda_2$ are dense in $C(G/\Lambda)$. Therefore, by linearity and continuity, it suffices to consider $f \in C(G/\Lambda)$ such that $f(l x) = \chi(l) f(x)$ for $l \in G_2/\Lambda_2$ and $x \in G/\Lambda$. 
	
	By Theorem \ref{thm: WW}, let $x_0 \in G/\Lambda$ be such that the averages
	\begin{equation*}
		\frac{1}{|\Phi_N|} \sum_{\gamma \in \Phi_N} \lambda(\gamma) h(T^\gamma_{G/\Lambda} x_0)
	\end{equation*}
	converge for every $\lambda \in \hat{\Gamma}$ and every $h \in C(G/\Lambda)$.
	Write $x_0 = g_0 \Lambda$ with $g_0 \in G$.
	Then for any $x = g \Lambda \in G/\Lambda$, we have
	\begin{align*}
		f \left( T^\gamma_{G/\Lambda} x \right) & = f \left( \phi(\gamma) g \Lambda \right) \\
		 & = f \left( \phi(\gamma) g g_0^{-1} g_0 \Lambda \right) \\
		 & = f \left( [\phi(\gamma), g g_0^{-1}] g g_0^{-1} \phi(\gamma) g_0 \Lambda \right) \\
		 & = \chi \left( [\phi(\gamma), g g_0^{-1}] \right) f \left( g g_0^{-1} \phi(\gamma) x_0 \right).
	\end{align*}
	Taking $h(z) = f \left( g g_0^{-1} z \right)$ and $\lambda(\gamma) = \chi \left( [\phi(\gamma), g g_0^{-1}] \right)$ (which defines an element of $\hat\Gamma$, since the commutator map $[\cdot,\cdot]$ is bilinear in a $2$-nilpotent group), we then have $f(T^\gamma_{G/\Lambda} x) = \lambda(\gamma) h(T^\gamma_{G/\Lambda} x_0)$, so
	\begin{equation*}
		\frac{1}{|\Phi_N|} \sum_{\gamma \in \Phi_N} f(T^\gamma_{G/\Lambda}x) = \frac{1}{|\Phi_N|} \sum_{\gamma \in \Phi_N} \lambda(\gamma) h(T^\gamma_{G/\Lambda} x_0)
	\end{equation*}
	converges.

\section{Ergodic decomposition of translational systems} \label{sec: ergodic decomp}

Combining the results of the previous sections, we have the following description of the ergodic measures for the translational $\Gamma$-system $(G/\Lambda, T_{G/\Lambda})$:

\begin{theorem} \label{thm: ergodic decomposition}
    Let $(G/\Lambda, T_{G/\Lambda})$ be a translational $\Gamma$-system of degree 2. 
	The space $G/\Lambda$ decomposes as a disjoint union of orbit closures $G/\Lambda = \bigsqcup_{i \in I} \O(x_i)$, and each of the systems $(\O(x_i), T_{\O(x_i)})$ is minimal and uniquely ergodic. 
	In particular, if $\mu$ is an ergodic measure for $(G/\Lambda, T_{G/\Lambda})$, then $\mu = \mu_{x_i}$ for some $i \in I$, where $\mu_{x_i}$ is the unique $T_{\O(x_i)}$-invariant Borel probability measure on $\O(x_i)$.
\end{theorem}

Before giving the proof, we state an immediate corollary:

\begin{corollary} \label{cor: ergodic, minimal, ue}
	Under the assumption of Theorem \ref{thm: ergodic decomposition}, the following are equivalent:
	\begin{enumerate}[(i)]
		\item	$(G/\Lambda, \Sigma_{G/\Lambda},\mu_{G/\Lambda}, T_{G/\Lambda})$ is ergodic.
		\item	$(G/\Lambda, T_{G/\Lambda})$ is minimal.
		\item	$(G/\Lambda, T_{G/\Lambda})$ is uniquely ergodic.
	\end{enumerate}
\end{corollary}

\begin{proof}[Proof of Theorem \ref{thm: ergodic decomposition}]
	The decomposition into disjoint minimal orbit closures comes from Corollary \ref{cor: minimal}.
	That each orbit closure is uniquely ergodic follows from Theorem \ref{thm: ue}.
	For the final claim, suppose $\mu$ is an ergodic measure for $(G/\Lambda, T_{G/\Lambda})$.
	Then by Lindenstrauss's pointwise ergodic theorem \cite{lindenstrauss}, there is a F{\o}lner sequence $(\Phi_N)$ in $\Gamma$ and a point $x \in G/\Lambda$ such that for every $f \in C(G/\Lambda)$, we have
	\begin{equation*}
		\lim_{N \to \infty} \frac{1}{|\Phi_N|} \sum_{\gamma \in \Phi_N} f(T^{\gamma}_{G/\Lambda}x) = \int_{G/\Lambda} f~d\mu.
	\end{equation*}
	But by Corollary \ref{cor: minimal} and Proposition \ref{prop: ue characterization},
	\begin{equation*}
		\lim_{N \to \infty} \frac{1}{|\Phi_N|} \sum_{\gamma \in \Phi_N} f(T^{\gamma}_{G/\Lambda}x) = \int_{G/\Lambda} f~d\mu_x.
	\end{equation*}
	Hence, $\mu = \mu_x$.
\end{proof}

\section{Reduction of Theorem \ref{thm: main} to a special case} \label{sec: reductions}

The goal of this section is to reduce the general case of Theorem \ref{thm: main} to the special case in which there exists a compact abelian Lie group $L$ such that $[G,G] \subseteq L \subseteq Z(G)$ and $L \cap \Lambda = \{1\}$.\footnote{Such a reduction can always be achieved in the structure theorem (Theorem \ref{thm: structure theorem}), see \cite[Proposition 4.4]{jst}.}  

We first set up some notation.
Let $\phi : \Gamma \to G$ be the group homomorphism such that $T_{G/\Lambda}^{\gamma}x = \phi(\gamma)x$.
Let $G = G_0 = G_1 \ge G_2 \ge G_3 = \{1\}$ be a filtration of $G$ compatible with $\Lambda$.
Note that being a filtration in this case means that $[G,G] \subseteq G_2 \subseteq Z(G)$.

\begin{reduction} \label{red: faithful action}
    We may assume without loss of generality that $G$ acts faithfully on $G/\Lambda$.
\end{reduction}

\begin{proof}
    Assume that Theorem \ref{thm: main} holds under the additional assumption that $G$ acts faithfully.
    We will deduce the general case.

    Let $N = \{g \in G : gx = x~\text{for every}~x \in G/\Lambda\}$.
    It is easy to see that $N$ is a closed normal subgroup of $G$.
    Moreover, by considering the point $x_0 = \Lambda \in G/\Lambda$, we have $N \le \Lambda$.
    Let $G' = G/N$ and $\Lambda' = \Lambda/N$.
    Define $\phi' : \Gamma \to G'$ by $\phi'(\gamma) = \phi(\gamma) \bmod{N}$.
    Let $\pi : G \to G'$ be the natural projection and $\iota : G/\Lambda \to G'/\Lambda'$ the induced isomorphism.

    Fix a point $x \in G/\Lambda$.
    Let $x' = \iota(x)$.
    By assumption, since $G'$ acts faithfully on $G'/\Lambda'$, there exists a closed subgroup $H' \le G'$ such that $\O(x') = H'x'$ and $(\phi'(\gamma)x')_{\gamma \in \Gamma}$ is well-distributed in $H'x'$.
    Let $H = \pi^{-1}(H') \le G$.
    Then $H$ is a closed subgroup of $G$.
    Moreover, $\O(x) = \iota^{-1}(\O(x')) = \iota^{-1}(H'x') = Hx$ and $\phi(\gamma)x = \iota^{-1}(\phi'(\gamma)x')$ is well-distributed in $Hx$.
\end{proof}

We now observe a consequence of Reduction \ref{red: faithful action}.

\begin{lemma} \label{lem: G_2 compact}
    Assume $G$ acts faithfully on $G/\Lambda$.
    Then $G_2 \cap \Lambda = \{1\}$ and $G_2$ is compact.
\end{lemma}

\begin{proof}
    Let $\Lambda_2 = G_2 \cap \Lambda$.
    Given $\lambda \in \Lambda_2$ and $x = g\Lambda \in G/\Lambda$, we have
    \begin{equation*}
        \lambda x = \lambda g\Lambda \stackrel{(1)}{=} g\lambda \Lambda \stackrel{(2)}{=} g\Lambda = x,
    \end{equation*}
    where in step (1) we use that $\lambda \in G_2 \subseteq Z(G)$ and in step (2) that $\lambda \in \Lambda$.
    By the assumption that $G$ acts faithfully, we conclude $\lambda = 1$.
    Hence, $\Lambda_2 = \{1\}$.
    Now, since $\Lambda$ is compatible with the filtration $G_{\bullet}$, we have that
    \begin{equation*}
        G_2 \cong G_2/\{1\} = G_2/\Lambda_2
    \end{equation*}
    is compact.
\end{proof}

We now want to reduce to the case that $G_2$ is a Lie group.
The key group theoretic fact enabling this reduction is that every compact abelian group is an inverse limit of compact abelian Lie groups.

\begin{lemma} \label{lem: inverse limit of Lie groups}
    Let $K$ be a compact abelian group.
    There exists a sequence of closed subgroups $K = K_1 \ge K_2 \ge \ldots$ such that $\bigcap_{n \in \N} K_n = \{0\}$ and $L_n = K/K_n$ is a compact abelian Lie group for each $n \in \N$.
    In particular, $K$ is the inverse limit of the compact abelian Lie groups $L_n$.
\end{lemma}

\begin{proof}
    Since $K$ is compact, the group $\hat{K}$ is a countable discrete group.
    Let $\{1\} = A_1 \le A_2 \le \ldots \le \hat{K}$ be a sequence of finitely generated subgroups with $\bigcup_{n \in \N} A_n = \hat{K}$.
    Then define $K_n = A_n^{\perp} = \{k \in K : \chi(k) = 1~\text{for every}~\chi \in A_n\}$.
    By construction, $K = K_1 \ge K_2 \ge \ldots$ and $\bigcap_{n \in \N} K_n = \{0\}$.
    Moreover, $L_n = K/K_n \cong \hat{A}_n$ is the dual of a finitely-generated abelian group, so $L_n$ is a Lie group.
\end{proof}

\begin{reduction}
    We may assume without loss of generality that there exists a compact abelian Lie group $L$ such that $[G,G] \subseteq L \subseteq Z(G)$ and $L \cap \Lambda = \{1\}$.
\end{reduction}

\begin{proof}
    By Reduction \ref{red: faithful action}, assume that $G$ acts faithfully on $G/\Lambda$.
    Then by Lemma \ref{lem: G_2 compact}, the group $G_2$ is a compact abelian group and $G_2 \cap \Lambda = \{1\}$.
    Hence, by Lemma \ref{lem: inverse limit of Lie groups}, let $G_2 = K_1 \ge K_2 \ge \ldots$ be a decreasing sequence of compact subgroups of $G$ such that $\bigcap_{n \in \N} K_n = \{1\}$ and $L_n = G_2/K_n$ is a compact Lie group for each $n \in \N$.

    The group $K_n \subseteq G_2 \subseteq Z(G)$ is central, so $K_n$ is a normal subgroup of $G$.
    For each $n \in \N$, let $G^{(n)} = G/K_n$.
    We claim that $[G^{(n)}, G^{(n)}] \subseteq L_n \subseteq Z(G^{(n)})$.
    Indeed, for the first inclusion, we note that $K_n$ being central implies that if $g \equiv g' \pmod{K_n}$ and $h \equiv h' \pmod{K_n}$, then
    \begin{equation*}
        [g,h] = [g',h']
    \end{equation*}
    so $[G^{(n)}, G^{(n)}] \subseteq [G,G]/K_n \subseteq G_2/K_n = L_n$; the second inclusion follows immediately from the fact that $G_2$ is central.
    This proves that $G^{(n)} = G^{(n)}_0 = G^{(n)}_1 \ge G^{(n)}_2 = L_n \ge G^{(n)}_3 = \{1\}$ is a filtration of $G^{(n)}$.
    Since $G^{(n)}_i = G_i/K_n$ for each $n \in \N$ and $i \in \{0,1,2,3\}$, the filtration $G^{(n)}_{\bullet}$ is compatible with $\Lambda$ for $n \in \N$.

    Assume now that Theorem \ref{thm: main} holds for the systems $(G^{(n)}/\Lambda, T_{G^{(n)}/\Lambda})$, where $T_{G^{(n)}/\Lambda}$ is induced by the homomorphism $\phi^{(n)} : \Gamma \to G^{(n)}$ defined by $\phi^{(n)} = \phi \bmod{K_n}$.
    Let $x \in G/\Lambda$.
    For $n \in \N$, let $x_n = x \bmod{K_n} \in G^{(n)}/\Lambda$.
    By our assumption, there exists a closed subgroup $H_n \le G^{(n)}$ such that $\O(x_n) = H_n x_n$ and $(\phi^{(n)}(\gamma)x_n)_{\gamma \in \Gamma}$ is well-distributed in $H_n x_n$.
    To enable a construction of an inverse limit of the groups $(H_n)_{n \in \N}$, we will choose $H_n$ maximally, namely
    \begin{equation*}
        H_n = \{g \in G^{(n)} : g \cdot \O(x_n) = \O(x_n)\}.
    \end{equation*}
    Note that if $g \in H_n$ for some $n \in \N$ and $m < n$, then putting $g_m = g \bmod{K_m}$, we have
    \begin{equation*}
        g_m \cdot \O(x_m) = g \cdot \O(x_n) \bmod{K_m} = \O(x_n) \bmod{K_m} = \O(x_m),
    \end{equation*}
    so $g_m \in H_m$.
    We may therefore define $H = \{g \in G : g \bmod{K_n} \in H_n~\text{for all}~n \in \N\}$ as the inverse limit of the sequence
    \begin{equation*}
        \ldots \to H_{n+1} \stackrel{\bmod{K_n}}{\to} H_n \to \ldots \to H_3 \stackrel{\bmod{K_2}}{\to} H_2 \stackrel{\bmod{K_1}}{\to} H_1.
    \end{equation*}

    By construction, $Hx$ is the inverse limit of $H_n x_n = \O(x_n)$, so $Hx = \O(x)$.
    To establish well-distribution of $(\phi(\gamma)x)_{\gamma \in \Gamma}$ in $Hx$, one must show
    \begin{equation} \label{eq: well-distributed in Hx}
        \lim_{N \to \infty} \frac{1}{|\Phi_N|} \sum_{\gamma \in \Phi_N} f(\phi(\gamma)x) = \int_{Hx} f~d\mu_{Hx}
    \end{equation}
    for every continuous function $f \in C(G/\Lambda)$ and F{\o}lner sequence $(\Phi_N)$ in $\Gamma$.
    But $\bigcup_{n \in \N} C(G^{(n)}/\Lambda)$ is dense in $C(G/\Lambda)$, so we may assume $f \in C(G^{(n)}/\Lambda)$ for some $n \in \N$.
    The identity \eqref{eq: well-distributed in Hx} then holds by well-distribution of $(\phi^{(n)}(\gamma)x_n)_{\gamma \in \Gamma}$ in $H_n x_n$.
\end{proof}

\section{Proof of Theorem \ref{thm: main}} \label{sec: proof of equidistribution}

\subsection{Skew-product representation of translational systems}

The final ingredient that we need to establish the proof of Theorem \ref{thm: main} is the following representation result from \cite{jst}. 

\begin{definition}
Let $(Y, \Sigma_Y, \mu_Y, T_Y)$ be a measure-preserving $\Gamma$-system and $L$ a metrizable compact abelian group. A \emph{cocycle} is a map $\rho\colon \Gamma\times Y\to L$ satisfying, for all $\gamma_1, \gamma_2\in \Gamma$ and almost every $y\in Y$,
\[
\rho(\gamma_1+\gamma_2,y)=\rho(\gamma_1,y)+\rho(\gamma_2,T^{\gamma_1}_Yy).
\]
Two cocycles $\rho$ and $\rho'$ are \emph{cohomologous} if there exists a measurable function $F \colon Y \to L$ such that for every $\gamma \in \Gamma$ and almost every $y \in Y$,
\[
\rho'(\gamma,y) = \rho(\gamma,y) + F(T_Y^{\gamma}y) - F(y).
\]
Given such a cocycle, the \emph{skew-product $\Gamma$-system} $Y\rtimes_\rho L$ is defined on the product space $(Y\times L, \Sigma_Y\times \Sigma_L, \mu_Y\otimes \mu_L)$, where $\Sigma_L$ is the Borel $\sigma$-algebra and $\mu_L$ the Haar measure on $L$, by the measure-preserving action
\[
T_\rho^\gamma(y,l)\coloneqq (T^\gamma_Y y, \rho(\gamma,y)+l), \quad \text{for all } y\in Y, l\in L.
\]
\end{definition}

If $\rho, \rho' : \Gamma \times Y \to L$ are cohomologous cocycles, then the skew-product systems $Y \rtimes_{\rho} L$ and $Y \rtimes_{\rho'} L$ are measurably isomorphic.

\begin{proposition} \label{lem: $2$-step nil group extension}
Let $(G/\Lambda,\Sigma_{G/\Lambda},\mu_{G/\Lambda},T_{G/\Lambda})$ be translational $\Gamma$-system of degree 2, where and the action $T_{G/\Lambda}$ is induced by the homomorphsim $\phi\colon\Gamma\to G$.
Assume that there is a closed subgroup $L \le G$ such that $L$ is a compact abelian Lie group with $[G,G] \le L \le Z(G)$ and $L \cap \Lambda = \{1\}$.
Then there exist a rotational $\Gamma$-system $(Z,\Sigma_Z,\mu_Z,T_Z)$ and a cocycle $\rho\colon \Gamma \times Z \to L$ such that the skew-product $\Gamma$-system $Z \rtimes_\rho L$ is isomorphic to the translational system $(G/\Lambda,\Sigma_{G/\Lambda},\mu_{G/\Lambda},T_{G/\Lambda})$. 

Moreover, the cocycle $\rho$ satisfies the following \emph{Conze--Lesigne equation}: there exists a measurable function $F : G \times Z \to L$ such that 
\begin{equation}\label{cl-property}
    \rho(\gamma, z+\pi(g)) - \rho(\gamma,z) = F(g, T^\gamma_Z(z)) - F(g, z) - [g,\phi(\gamma)]
\end{equation}
for all $g \in G$, $\gamma \in \Gamma$, and $z \in Z$, where $\pi\colon G \to Z$ denotes the projection map (see below for the definition). 
\end{proposition}

\begin{proof}[Sketch of proof]
For a proof, we refer the interested reader to \cite[Proposition 4.1]{jst}, where the groups denoted by \( G_2 \) and \( K \) in \cite{jst} are both taken as \( L \). While \cite[Proposition 4.1]{jst} is originally stated for ergodic translational systems and its proof treats the isomorphism between the translational system and the skew-product group extension modulo null sets, a careful examination of the proof reveals that the same argument extends to non-ergodic translational systems and that the isomorphism constructed therein is a bi-measurable bijection defined everywhere, a fact we will use in the sequel.  
\end{proof}

As preparation for the proof of Theorem \ref{thm: main}, we collect some notation from the proof of Proposition \ref{lem: $2$-step nil group extension} in \cite{jst}.   

We define \( Z = G/L\Lambda \) and equip the metrizable compact abelian group \( Z \) with the Haar measure \( \mu_Z \). Let \( \pi\colon G\to Z \) be the projection homomorphism, define the homomorphism \( \psi\colon  \Gamma \to Z\) as the composition $\pi\circ \phi$, and let $T_Z$ be the induced measure-preserving action on $Z$.  Then, \( (Z,\Sigma_Z,\mu_Z,T_Z) \) forms a rotational \( \Gamma \)-system.  

Additionally, we have a quotient map \( \varphi\colon G\to G/\Lambda \), which satisfies \( \pi= \tilde{\pi} \circ \varphi \), where \( \tilde{\pi}\colon G/\Lambda\to Z \) is the natural continuous surjection.  
Let \( s\colon Z\to G \) be a Borel cross-section of \( \pi\colon G\to Z \) (see, e.g., \cite[Theorem 1.2.4]{becker1996}), and define \( \tilde{s}=\varphi\circ s \), which serves as a Borel cross-section of \( \tilde{\pi} \).  

Since \( L \) is central and satisfies \( L\cap \Lambda=\{1\} \), it acts freely on \( G/\Lambda \). We express this action additively, writing \( l + x \) for \( x\in G/\Lambda \) and \( l \in L \).  
For \( z \in Z \) and \( g \in G \), we have  
\[
\tilde{\pi}(g \cdot \tilde{s}(z)) = z + \pi(g),
\]
which implies the existence of a unique element \( F(g, z) \in L \) such that  
\begin{equation}\label{def-F0}
    g \cdot \tilde{s}(z) = F(g,z) + \tilde{s}(z + \pi(g)).
\end{equation}
The function \( F\colon G\times Z\to L \) defined in this manner is jointly Borel measurable.  

Since \( L \) commutes with \( G \), we also have  
\begin{equation} \label{eq: action of g on product space}
    g \cdot (l + \tilde{s}(z)) = (l + F(g,z)) + \tilde{s}(z + \pi(g))
\end{equation}
for all \( l \in L \).  

Define \( \rho: \Gamma \times Z \to L \) by  
\[
\rho(\gamma,z) = F(\phi(\gamma),z).
\]
The map \( Z \times L \to G/\Lambda \), given by \( (z,l) \mapsto l + \tilde{s}(z) \), is a bijection that is Borel measurable with a Borel measurable inverse.  
By \eqref{eq: action of g on product space}, this establishes an isomorphism between the translational \( \Gamma \)-system \( (G/\Lambda, \Sigma_{G/\Lambda},\mu_{G/\Lambda}, T_{G/\Lambda}) \) and the skew-product \( \Gamma \)-system \( Z\rtimes_\rho L \).  

\subsection{A measurability lemma}

We denote by $\mathcal{M}(Z,L)$ the space of equivalence classes of measurable functions from $Z$ to $L$, identified $\mu_Z$-almost surely, with the topology of convergence in measure.
Given $f \in L^2(Z)$, $\chi \in \hat{L}$, and $\eps > 0$, let  
\begin{equation*}
    V(f,\chi,\eps) = \left\{ h \in \mathcal{M}(Z,L) : \|\chi \circ h - \chi \circ f\|_{L^2(Z)} < \eps \right\}.
\end{equation*}
The basic open sets for the topology on $\mathcal{M}(Z,L)$ are of the form $\bigcap_{i=1}^k V(f, \chi_i, \eps)$; see, e.g., \cite[Lemma 7.28]{abb}.

\begin{lemma} \label{lem:measurability}
    Suppose $F : G \times Z \to L$ is jointly Borel measurable.
    Then the map $\eta : G \to \mathcal{M}(Z,L)$ defined by $\eta(g) = F(g,\cdot)$ is Borel measurable.
\end{lemma}

\begin{proof}
    It suffices to show that $\eta^{-1}(V(f,\chi,\eps))$ is a Borel set for each $f \in L^2(Z)$, $\chi \in \hat{L}$, and $\eps > 0$.  
    Fix $f \in L^2(Z)$, $\chi \in \hat{L}$, and $\eps > 0$.  
    Note that  
    \begin{equation*}
        \eta^{-1}(V(f,\chi,\eps)) = \left\{ g \in G : \int_Z \left| \chi(F(g,z)) - \chi(f(z))\right|^2~dz < \eps^2 \right\}.
    \end{equation*}
    Let $H : G \times Z \to [0,\infty)$ be the map defined by  
    \[
    H(g,z) = \left| \chi(F(g,z)) - \chi(f(z)) \right|^2.
    \]
    Since $H$ is a composition of Borel measurable functions, it is Borel measurable.  
    By Tonelli's theorem, the function $I : G \to [0,\infty)$ defined by  
    \begin{equation*}
        I(g) = \int_Z H(g,z)~dz
    \end{equation*}
    is also Borel measurable.  
    Therefore,  
    \begin{equation*}
        \eta^{-1}(V(f,\chi,\eps)) = I^{-1}([0,\eps^2))
    \end{equation*}
    is a Borel subset of $G$.
\end{proof}

\subsection{Mackey groups}

We need the following fact about the Mackey range of abelian skew-product systems (see, e.g., \cite[Proposition 2.3(iii)]{jst}, and the references mentioned therein).  

\begin{proposition}\label{prop-mackey}
Let $(Y,\Sigma_Y,\mu_Y,T_Y)$ be an ergodic measure-preserving $\Gamma$-system, let $L$ be a metrizable compact abelian group, and let $\rho\colon \Gamma\times Y\to L$ be a cocycle.  
There exists a closed subgroup $\tilde{L}\leq L$ and an ergodic\footnote{A cocycle $\rho$ is said to be \emph{ergodic} if the corresponding skew-product $\Gamma$-system $Y\rtimes_\rho L$ is ergodic.} cocycle $\tilde{\rho}\colon \Gamma\times Y\to \tilde{L}$ such that $\rho$ is cohomologous to $\tilde{\rho}$ if both are viewed as cocycles with values in $L$. 
\end{proposition}

\subsection{Finishing the proof}

We are ready to prove Theorem \ref{thm: main}: 

\begin{proof}	
By the results of Section \ref{sec: reductions}, we may assume that there exists a compact abelian Lie group $L$ such that $[G,G] \le L \le Z(G)$ and $L \cap \Lambda = \{1\}$.

Fix a point \( x=g\Lambda \in G/\Lambda \). 
Let 
\[
Z_0=\overline{\{\psi(\gamma)L\Lambda : \gamma \in \Gamma\}}
\]
be the closed subgroup of $Z$ corresponding to the orbit closure of the identity in $Z$. 
By Proposition \ref{prop: Kronecker}, $(Z_0,T_{Z_0})$ is a uniquely ergodic rotational \( \Gamma \)-system. 

Defining \( Z_x=\tilde{\pi}(x)+Z_0 \), the resulting $\Gamma$-system \( (Z_x,T_{Z_x}) \) is also uniquely ergodic. 
In particular, we have \( \tilde{\pi}_* \mu_x = \mu_{Z_x} \), where \( \mu_x \) is the ergodic component of \( \mu_{G/\Lambda} \) as in Proposition \ref{thm: ergodic decomposition}.  

On \( Z_x \times L \), we consider the restriction of the \( \Gamma \)-action \( T_\rho \) by restricting the cocycle \( \rho \) to \( Z_x \), noting that \( \rho|_{Z_x} \) still satisfies the Conze--Lesigne equation \eqref{cl-property}.  

The skew-product system \( Z_x \rtimes_\rho L \) may be non-ergodic. By Proposition \ref{prop-mackey}, there exists a closed subgroup \( L_x \leq L \) and an ergodic cocycle \( \rho_x : \Gamma \times Z_x \to L_x \) cohomologous to \( \rho \).  
Moreover, \( Z_x \rtimes_{\rho_x} L_x \) is isomorphic to the measure-preserving $\Gamma$-system \( (\O(x), \Sigma_{\O(x)},\mu_x, T_{\O(x)}) \), as it is the ergodic component containing the point corresponding to \( x \).  

Since \( \rho_x \) is cohomologous to \( \rho \) on \( Z_x \), there exists a Borel measurable function \( \Phi : Z_x \to L \) such that for every \( \gamma \in \Gamma \),
\begin{equation} \label{eq: cohomology}
    \rho_x(\gamma,z) = \rho(\gamma,z) + \Phi(T^\gamma_Z(z))-\Phi(z)
\end{equation}
$\mu_{Z_x}$-almost surely. 

Define $\tilde{F} : G \times Z_x \to L$ by $$\tilde{F}_g(z) = F(g,z) + \Phi(z + \pi(g))-\Phi(z).$$
Note that $\rho_x(\gamma, \cdot) = \tilde{F}_{\phi(\gamma)}$ for $\gamma \in \Gamma$, and we have the Conze--Lesigne equation
\begin{equation} \label{eq:F tilde CL}
    \rho_x(\gamma,z+\pi(g))-\rho_x(\gamma,z)= \tilde{F}_g(T^\gamma_Z(z))-\tilde{F}_g(z) - [g,\phi(\gamma)]
\end{equation}
for almost every $z\in Z_x$, every $g \in G$, and every $\gamma \in \Gamma$.

Moreover, $\tilde{F}$ satisfies the identities
\begin{equation} \label{eq: F tilde group operations}
    \tilde{F}_{hh'}(z) = \tilde{F}_h(z + \pi(h')) + \tilde{F}_{h'}(z), \quad  
    \tilde{F}_{h^{-1}}(z) = - \tilde{F}_h(z - \pi(h)).
\end{equation}

Define a subgroup $H$ of $G$ by
\begin{equation*}
    H = \left\{ h \in G : \pi(h) \in Z_0~\text{and}~\tilde{F}_h(z) \in L_x \, \mu_{Z_x}\text{-a.s.}\right\}.
\end{equation*}
We may write $H = H_1 \cap H_2$, where $H_1 = \pi^{-1}(Z_0)$ and
\begin{equation*}
    H_2 = \left\{ h \in G : \tilde{F}_h(z) \in L_x \; \mu_{Z_x}\text{-a.s.}\right\}.
\end{equation*}
Since $Z_0$ is a closed subgroup of $Z$ and $\pi : G \to Z$ is a continuous homomorphism, the preimage $H_1$ is a closed subgroup of $G$.
To see that $H_2$ is also a subgroup, we use the identities \eqref{eq: F tilde group operations} and the fact that $L_x$ is a subgroup of $L$.
We claim that $H_2$ is also closed.
Since $L_x$ is closed, it suffices to prove that the map $\eta : G \to \mathcal{M}(Z_x,L)$ defined by $\eta(h) = \tilde{F}_h$ is continuous.
By Lemma \ref{lem:measurability}, $\eta$ is Borel measurable.
Therefore, by Lusin's theorem, there exists a closed subset \( E \subseteq G \) with \( \mu_G(E) > 0 \) such that \( \eta|_E \) is continuous.  
Using the identities in \eqref{eq: F tilde group operations}, we get  
\begin{equation} \label{eq: eta group operations}
    \eta(hh') = \eta(h) \circ \tau_{\pi(h')} + \eta(h'), \quad  
    \eta(h^{-1}) = - \eta(h) \circ \tau_{-\pi(h)}.
\end{equation}
Since translation is continuous on \( L^2(Z_x) \), it follows that \( \eta \) is continuous on the group generated by \( E \).  
By Weil's theorem \cite{weil}, the set \( EE^{-1} \) contains a neighborhood of the identity in \( G \), so \( \eta \) is continuous at \( 1 \).  
Using \eqref{eq: eta group operations} again, we conclude that \( \eta \) is continuous on \( G \).  
This proves that \( H_2 \) is closed, and hence \( H \) is a closed subgroup of \( G \).

We now wish to show that $\pi(H) = Z_0$.
Let $\text{Hom}(Z_0,L/L_x)$ be the (countable) group of continuous homomorphisms from $Z_0$ to $L/L_x$. \\

\noindent \underline{Claim 1}: There is a continuous homomorphism $\Theta\colon H_1\to \text{Hom}(Z_0,L/L_x)$ sending $h \mapsto \theta_h$ such that $[h,\phi(\gamma)] \equiv \theta_h(\psi(\gamma)) \pmod{L_x}$ for every $h \in H_1$ and $\gamma \in \Gamma$.

\begin{proof}[Proof of Claim 1]
    Since $\rho_x$ is $L_x$-valued, taking the Conze--Lesigne equation \eqref{eq:F tilde CL} mod $L_x$, we have
\begin{equation}\label{eq:F tilde CL0}
    \tilde{F}_h(T_{Z_x}^\gamma(z))-\tilde{F}_h(z)  \equiv [h, \phi(\gamma)] \pmod{L_x}.
\end{equation}
That is, $\tilde{F}_h \mod{L_x}$ is an eigenfunction for $(Z_x,\Sigma_{Z_x},\mu_{Z_x},T_{Z_x})$ with eigenvalue $[h, \phi(\gamma)] \mod{L_x}$.
But $(Z_x,\Sigma_{Z_x},\mu_{Z_x}, T_{Z_x})$ is isomorphic to the ergodic group rotation $(Z_0,\Sigma_{Z_0},\mu_{Z_0},T_{Z_0})$, whose eigenvalues are of the form $\theta \circ \psi$ for homomorphisms $\theta$ on $Z_0$.
Since $\{\psi(\gamma) : \gamma \in \Gamma\}$ is dense in $Z_0$, the homomorphism $\theta_h : Z_0 \to L/L_x$ is uniquely determined. (Note that we have used in the previous paragraph that $L/L_x$ is a compact abelian Lie group, that is, isomorphic to a group of the form $\mathbb{T}^d \oplus W$ where $d \in \N$ and $W$ is a finite abelian group.)

Since $G$ is 2-nilpotent, the map $\Theta$ is a homomorphism.
Moreover, $\Theta$ is Borel measurable, as can be seen by noting that for each $\theta \in \text{Hom}(Z_0,L/L_x)$,
\begin{equation*}
    \Theta^{-1}(\{\theta\}) = \left\{ h \in H_1 : \theta_h = \theta \right\} = \bigcap_{\gamma \in \Gamma} \left\{ h \in H_1 : [h,\phi(\gamma)] \equiv \theta(\psi(\gamma)) \pmod{L_x} \right\}
\end{equation*}
is a closed subset of $H_1$, since $h \mapsto [h,\phi(\gamma)]$ is a continuous function on $G$.
By automatic continuity (see, e.g., \cite[Theorem 2.2]{rosendal}), it follows that $\Theta$ is a continuous homomorphism.
\end{proof}

\noindent \underline{Claim 2}: $\Theta(\Lambda) = \Theta(H_1)$.

\begin{proof}[Proof of Claim 2.]
    Note that $H_1 = \overline{\phi(\Gamma) L \Lambda} \subseteq G$.
Therefore, since $\Theta$ is continuous, we have $\Theta(H_1) = \overline{\Theta(\phi(\Gamma)L\Lambda)} \subseteq \text{Hom}(Z_0,L/L_x)$.
But $\text{Hom}(Z_0,L/L_x)$ is discrete, so every subset is already closed.
Hence, $\Theta(H_1) = \Theta(\phi(\Gamma)L\Lambda)$.
Finally, from the definition of $\Theta$, since $\phi(\Gamma)$ is an abelian subgroup of $G$ and $L$ is central, we have $\Theta(\phi(\Gamma)) = \Theta(L) = \{0\}$.
The claim then follows from the fact that $\Theta$ is a homomorphism.
\end{proof}

Now we can show $\pi(H) = Z_0$.
Fix $u \in Z_0$.
By Claim 2, there exists $\lambda \in \Lambda$ such that $\theta_{\lambda} = \theta_{s(u)}$.
Therefore, by the Conze--Lesigne equation \eqref{eq:F tilde CL0} and Claim 1, for almost every $z\in Z_x$,
\begin{multline*}
    \tilde{F}_{s(u)\lambda^{-1}}(T_{Z_x}^\gamma(z)) - \tilde{F}_{s(u)\lambda^{-1}}(z) \equiv [s(u)\lambda^{-1},\phi(\gamma)] \\ \equiv \theta_{s(u)\lambda^{-1}}(\psi(\gamma)) \equiv \theta_{s(u)}(\psi(\gamma)) - \theta_{\lambda}(\psi(\gamma)) \equiv 0 \pmod{L_x}.
\end{multline*}
In other words, $\tilde{F}_{s(u)\lambda^{-1}} \bmod{L_x}$ is an invariant function for the ergodic system $(Z_x,\Sigma_{Z_x},\mu_{Z_x}, T_{Z_x})$, so $\tilde{F}_{s(u)\lambda^{-1}}$ is equal to a constant $l_0$ mod $L_x$ $\mu_{Z_x}$-almost surely.
Let $l \in L$ with $l \equiv l_0 \pmod{L_x}$.
Then
\begin{equation*}
       \tilde{F}_{s(u)l^{-1}\lambda^{-1}} = \tilde{F}_{s(u)\lambda^{-1}} - l
\end{equation*}
takes values in $L_x$ for almost every $z \in Z_x$.
That is, $h = s(u)l^{-1}\lambda^{-1} \in H_2$.
Moreover, $\pi(h) = \pi(s(u)) = u$, so we are done. \\

Let \( K = \{h \in H : \pi(h) = 0\} \). Then we have the short exact sequence  
\[
0 \to K \to H \xrightarrow{\pi} Z_0 \to 0.
\]

Define \( \Lambda_x\coloneqq H\cap g\Lambda g^{-1} \), where \( g \in G \) is such that \( x = g\Lambda \).  
For \( \lambda \in \Lambda \), we have \( g \lambda g^{-1} = [g, \lambda] \lambda \in L\Lambda \), so \( \pi(g \lambda g^{-1}) = 0 \), implying \( \Lambda_x \leq K \).  
Since \( L_x\leq L \) is a central subgroup, we have \( L_x \cap \Lambda_x = \{1\} \).  

From \eqref{eq: F tilde group operations}, we see that \( F(l,z) = l \) for all \( l\in L \) and every \( z\in Z \), implying \( \tilde{F}_l(z) = l \) for all \( l\in L \).  
Thus, \( l\in H \) if and only if \( l\in L_x \).  
Since \( L_x\subset K \) and \( L_x\cap\Lambda_x=\{1\} \), we conclude that  $K = L_x\cdot \Lambda_x$. 
Hence, \( H/\Lambda_x \) is compact because both \( H/K \cong Z_x \) and \( K/\Lambda_x \cong L_x \) are compact.  

Now we show that \( \O(x) = Hx \).  Define an action of \( H \) on \( Z_x \times L_x \) by  
\[
h\cdot (z,l)\coloneqq (z + \pi(h), l + \tilde{F}_h(z)).
\]
This defines a valid group action by \eqref{eq: F tilde group operations}, which is also continuous.  
One verifies that the action is transitive.  

To find the stabilizer of \( x \) in \( H \), note that \( h\in H \) stabilizes \( x \) if and only if \( \pi(h)=0 \) and \( F(h,\tilde{\pi}(x)) = 0 \), i.e., \( hx=x \).  
The collection of such \( h \) is precisely \( \Lambda_x \), implying that \( H/\Lambda_x \) is isomorphic to \( Z_x \times L_x \) as measure-preserving \( H \)-systems.  
Applying the group homomorphism \( \phi\colon \Gamma\to H \), we deduce that \( H/\Lambda_x \) is isomorphic to the skew-product $\Gamma$-system \( Z_x\rtimes_{\rho_x} L_x \) as measure-preserving \( \Gamma \)-systems.  

Since \( \Lambda_x \) stabilizes \( x \) under the action of \( H \) by left multiplication, we obtain a continuous bijection \( \xi : H/\Lambda_x \to Hx \) given by \( \xi(h\Lambda_x) = hx \).  
Since \( H/\Lambda_x \) is compact, \( \xi \) is a homeomorphism, providing a topological isomorphism between \( (H/\Lambda_x, \phi) \) and \( (Hx, \phi) \):  
\[
\xi(\phi(\gamma)h\Lambda_x) = \phi(\gamma)hx = \phi(\gamma) \xi(h\Lambda_x).
\]

Using the isomorphism \[ (\O(x), \Sigma_{\O(x)}, \mu_x, T_{\O(x)}) \cong Z_x\rtimes_{\rho_x} L_x\cong (H/\Lambda_x, \Sigma_{H/\Lambda_x}, \mu_{H/\Lambda_x}, T_{H/\Lambda_x}), \] we conclude that \( (H/\Lambda_x, \Sigma_{H/\Lambda_x}, \mu_{H/\Lambda_x}, T_{H/\Lambda_x})  \) is ergodic.  
By Corollary \ref{cor: ergodic, minimal, ue}, the topological system \( (H/\Lambda_x, T_{H/\Lambda_x}) \) is minimal, implying that \( (Hx, T_{Hx}) \) is minimal.  
Since \( x \in Hx \), we deduce that \( \O(x) = Hx \), completing the proof.
\end{proof}

\part{Proof of the sumset result}

We adapt the strategy of Kra--Moreira--Richter--Robertson \cite{kmrr_finite_sums} from the integer setting to the context of abelian groups to establish Theorem \ref{thm:3-fold sumset}.
The first step is to translate Theorem \ref{thm:3-fold sumset} from a combinatorial statement to a dynamical statement in order to apply tools from ergodic theory.

\section{A dynamical encoding of sumsets}

\begin{definition}
    Let $(X, T_X)$ be a topological dynamical $\Gamma$-system, and let $k \in \N$.
    A tuple $(x_0, x_1, \dots, x_{k-1}) \in X^k$ is a \emph{$k$-term Erd\H{o}s progression} if there exists a sequence $(\gamma_n)_{n \in \N}$ of distinct elements of $\Gamma$ such that
    \begin{equation*}
        \lim_{n \to \infty} \left( T_X^{\gamma_n} x_0, \dots, T_X^{\gamma_n} x_{k-2} \right) = (x_1, \dots, x_{k-1}).
    \end{equation*}
    We denote the set of all $k$-term Erd\H{o}s progressions by $EP_k$.
    Given a point $x_0 \in X$, we write
    \begin{equation*}
        EP_k(x_0) = \left\{ (x_1, \dots, x_{k-1}) : (x_0, x_1, \dots, x_{k-1}) \in EP_k \right\}.
    \end{equation*}
\end{definition}

The following lemma shows that Erd\H{o}s progressions in a topological dynamical $\Gamma$-system lead to restricted sumset configurations in corresponding subsets of $\Gamma$.

\begin{lemma} \label{lem:EP sumset}
    Let $(X, T_X)$ be a topological $\Gamma$-system, and let $x_0 \in X$.
    Suppose $U_1, \dots, U_{k-1} \subseteq X$ are open sets and $EP_k(x_0) \cap (U_1 \times \dots \times U_{k-1}) \ne \emptyset$.
    Then letting $A_j = \{\gamma \in\Gamma: T_X^{\gamma} x_0 \in U_j\}$, there exists an infinite set $B \subseteq \Gamma$ such that
    \begin{equation} \begin{split} \label{eq: sumset locations}
        B & \subseteq A_1 \\
        B \oplus B & \subseteq A_2 \\
        \vdots \\
        B^{\oplus (k-1)} & \subseteq A_{k-1}.
    \end{split} \end{equation}
\end{lemma}

\begin{proof}
    The $\Gamma = \Z$ case is shown in \cite[Lemma 2.2]{kmrr_finite_sums} and the same strategy works for a general abelian group $\Gamma$.
    We include the details below for completeness.
    
    Let $(x_1, \dots, x_{k-1}) \in EP_k(x_0) \cap (U_1 \times \dots \times U_{k-1})$.
    By the definition of an Erd\H{o}s progression, let $(\gamma_n)_{n \in \N}$ be a sequence of distinct elements such that
    \begin{equation*}
        \lim_{n \to \infty} \left( T_X^{\gamma_n}x_0, \dots, T_X^{\gamma_n}x_{k-2} \right) = (x_1, \dots, x_{k-1}).
    \end{equation*}
    Refining to a subsequence if necessary, we may assume $T_X^{\gamma_n}x_i \in U_{i+1}$ for every $n \in \N$ and $i \in \{0, 1, \dots, k-2\}$.

    We construct a sequence $b_m = \gamma_{n_m}$ inductively to have the property
    \begin{equation} \label{eq: x location}
        x_i \in T_X^{-b_m} U_{i+1} \cap \bigcap_{j < m} T_X^{-b_m - b_j} U_{i+2} \cap \dots \cap \bigcap_{j_1 < j_2 < \dots < j_{k-2-i} < m} T_X^{-b_m - \sum_{l=1}^{k-2-i} b_{j_l}} U_{k-1}
    \end{equation}
    for each $i \in \{0, 1, \dots, k-2\}$.
    For $m = 1$, we take $n_1 = 1$.
    Then $x_i \in T_X^{-b_1} U_{i+1}$ by our assumption on the sequence $(\gamma_n)_{n \in \N}$.
    The other terms in \eqref{eq: x location} are trivial for $m=1$, so \eqref{eq: x location} is satisfied.

    Suppose we have chosen $n_1 < n_2 < \dots < n_m$ such that \eqref{eq: x location} is satisfied for $j \le m$ with $b_j = \gamma_{n_j}$.
    Let
    \begin{equation*}
        V_{j,i} = T_X^{-b_j} U_{i+1} \cap \bigcap_{j' < j} T_X^{-b_j - b_{j'}} U_{i+2} \cap \dots \cap \bigcap_{j_1 < j_2 < \dots < j_{k-2-i} < j} T_X^{-b_j - \sum_{l=1}^{k-2-i} b_{j_l}} U_{k-1}
    \end{equation*}
    for $j \in \{1, \dots, m\}$ and $i \in \{0, 1, \dots, k-2\}$.
    By the inductive hypothesis, $V_{j,i}$ is an open neighborhood of $x_i$ for each $j \in \{1, \dots, m\}$ and $i \in \{0, 1, \dots, k-2\}$.
    Thus, for all large enough $n$, we have
    \begin{equation*}
        \left( T_X^{\gamma_n}x_0, T_X^{\gamma_n}x_1, \dots, T_X^{\gamma_n}x_{k-3} \right) \in \left( \bigcap_{j=1}^m V_{j,1} \right) \times \left( \bigcap_{j=1}^m V_{j,2} \right) \times \dots \times \left( \bigcap_{j=1}^m V_{j,k-2} \right).
    \end{equation*}
    Choose $n_{m+1} > n_m$ sufficiently large so that for $b_{m+1} = \gamma_{n_{m+1}}$, we have $T_X^{b_{m+1}} x_i \in \bigcap_{j=1}^m V_{j,i+1}$ for every $i \in \{0, 1, \dots, k-3\}$.
    Then for $i \in \{0, \dots, k-2\}$,
    \begin{equation*}
        x_i \in T_X^{-b_{m+1}} U_{i+1}
    \end{equation*}
    and for $i \in \{0, \dots, k-3\}$,
    \begin{multline*}
        x_i \in \bigcap_{j=1}^m T_X^{-b_{m+1}} V_{j,i+1} \\
         = T_X^{-b_{m+1}} \left( \bigcap_{j < m+1} T_X^{-b_j} U_{i+2} \cap \bigcap_{j_1 < j_2 < m+1} T_X^{- b_{j_1} - b_{j_2}} U_{i+3} \cap \right. \dots  \\ \left. \cap \bigcap_{j_1 < \dots < j_{k-2-i} < m} T_X^{- \sum_{l=1}^{k-2-i} b_{j_l}} U_{k-1} \right).
    \end{multline*}
    That is, \eqref{eq: x location} is satisfied.

    Now let $B = \{b_m : m \in \N\}$.
    Applying \eqref{eq: x location} with $i = 0$ gives \eqref{eq: sumset locations}.
\end{proof}

Using an appropriate version of the Furstenberg correspondence principle and Lemma \ref{lem:EP sumset}, we can reduce Theorem \ref{thm:3-fold sumset} to a statement in ergodic theory.
In order to formulate the dynamical version of Theorem \ref{thm:3-fold sumset}, we introduce the following notation.
Let $(X, T_X)$ be a topological dynamical $\Gamma$-system.
Given a $T_X$-invariant Borel probability measure $\mu_X$ on $X$ and a F{\o}lner sequence $\Phi = (\Phi_N)$ in $\Gamma$, we say that a point $x \in X$ is \emph{generic for $\mu_X$ along $\Phi$}, written $x \in \gen(\mu_X,\Phi)$, if
\begin{equation*}
    \lim_{N \to \infty} \frac{1}{|\Phi_N|} \sum_{\gamma \in \Phi_N} \delta_{T_X^{\gamma}x} = \mu_X
\end{equation*}
in the weak* topology.

\begin{theorem}[Dynamical Formulation of Theorem \ref{thm:3-fold sumset}] \label{thm:dynamical sumset}
    Let $\Gamma$ be a countably infinite abelian group such that $[\Gamma:6\Gamma] < \infty$.
    Let $(X, T_X)$ be a topological dynamical $\Gamma$-system.
    Suppose $a \in X$, $\Phi=(\Phi_N)$ is a F{\o}lner sequence in $\Gamma$, and $\mu_X$ is a $T_X$-invariant ergodic measure such that $a \in \textup{gen}(\mu_X, \Phi)$.
    Let $E \subseteq X$ be an open set with $\mu_X(E) > 0$.
    Then there exist $(x_1, x_2, x_3) \in E \times E \times E$ and $t \in \Gamma$ such that $(T_X^ta, x_1, x_2, x_3)$ is a 4-term Erd\H{o}s progression.
\end{theorem}

\begin{proof}[Proof that Theorem \ref{thm:dynamical sumset} and Theorem \ref{thm:3-fold sumset} are equivalent]
	We use a standard argument combining a version of the Furstenberg correspondence principle and Lemma \ref{lem:EP sumset}.
    
	Suppose Theorem \ref{thm:dynamical sumset} holds.
	Let $\Gamma$ be a countably infinite abelian group such that $[\Gamma:6\Gamma] < \infty$, and let $A \subseteq \Gamma$ with $d^*(A) > 0$.
	By the Furstenberg correspondence principle (for the appropriate version, see \cite[Theorem 2.15]{cm}), there exists a topological dynamical $\Gamma$-system $(X, T_X)$, an ergodic $T_X$-invariant measure $\mu_X$, a F{\o}lner sequence $\Phi = (\Phi_N)$ in $\Gamma$, a point $a \in \gen(\mu_X, \Phi)$, and a clopen set $E \subseteq X$ such that $A = \{\gamma \in \Gamma : T_X^{\gamma}a \in E\}$ and $\mu_X(E) \ge d^*(A)$.
	Hence, by Theorem \ref{thm:dynamical sumset}, there exist $(x_1, x_2, x_3) \in E \times E \times E$ and $t \in \Gamma$ such that $(T_X^ta, x_1, x_2, x_3)$ is a 4-term Erd\H{o}s progression.
	Then noting that $A - t = \{\gamma \in \Gamma : T_X^{\gamma} (T_X^ta) \in E\}$, we have by Lemma \ref{lem:EP sumset} that there exists an infinite set $B \subseteq \Gamma$ such that
	\begin{equation*}
		B \cup (B \oplus B) \cup (B \oplus B \oplus B) \subseteq A - t.
	\end{equation*}
	That is, Theorem \ref{thm:3-fold sumset} holds.

    Conversely, suppose Theorem \ref{thm:3-fold sumset} holds.
    As in the setup of Theorem \ref{thm:dynamical sumset}, let $\Gamma$ be a countably infinite abelian group with $[\Gamma:6\Gamma] < \infty$, let $(X, T_X)$ be a topological dynamical $\Gamma$-system, let $\mu_X$ be an ergodic $T_X$-invariant measure, let $\Phi = (\Phi_N)$ be a F{\o}lner sequence in $\Gamma$, let $a \in \gen(\mu_X,\Phi)$, and let $E \subseteq X$ be an open set with $\mu_X(E) > 0$.
    By inner regularity of the measure $\mu_X$, there exists a compact subset $K \subseteq E$ such that $\mu_X(K) > 0$.
    Let $U \subseteq X$ be an open subset of $X$ with $K \subseteq U \subseteq \overline{U} \subseteq E$.
    Define $A = \{\gamma \in \Gamma : T_X^{\gamma}a \in U\}$.
    Then
    \begin{equation*}
        \liminf_{N \to \infty} \frac{|A \cap \Phi_N|}{|\Phi_N|} = \liminf_{N \to \infty} \frac{1}{|\Phi_N|} \sum_{\gamma \in \Phi_N} \ind_U(T_X^{\gamma}a) \ge \mu_X(U)
    \end{equation*}
    by the portmanteau lemma (see, e.g., \cite[Theorem 2.1]{billingsley}), so $d^*(A) \ge \mu_X(U) > 0$.
    Hence, by Theorem \ref{thm:3-fold sumset}, there exists $t \in \Gamma$ and an infinite set $B \subseteq \Gamma$ such that $B \cup (B \oplus B) \cup (B \oplus B \oplus B) \subseteq A - t$.
    Enumerate $B = \{b_1, b_2, \dots\}$.
    By compactness of $X$, after passing to a subsequence, we may assume that $(T_X^{b_n+t}a)_{n \in \N}$ converges to a point $x_1 \in X$.
    Passing again to a subsequence if necessary, we may also assume that $(T_X^{b_n}x_1)_{n \in \N}$ converges to a point $x_2 \in X$.
    Applying the same reasoning one final time, we may additionally assume that $(T_X^{b_n}x_2)_{n \in \N}$ converges to a point $x_3 \in X$.
    Thus, $(T_X^ta, x_1, x_2, x_3)$ is an Erd\H{o}s progression.
    It remains to check that $(x_1, x_2, x_3) \in E \times E \times E$.
    For each $n \in \N$, $T_X^{b_n+t}a \in U$ from the definition of the set $A$.
    Therefore, $x_1 \in \overline{U} \subseteq E$.
    Similarly,
    \begin{align*}
        x_2 = \lim_{n \to \infty} \lim_{m \to \infty} T_X^{b_n+b_m+t}a \in \overline{U} \subseteq E
        \intertext{and}
        x_3 = \lim_{n \to \infty} \lim_{m \to \infty} \lim_{k \to \infty} T_X^{b_n+b_m+b_k+t}a \in \overline{U} \subseteq E.
    \end{align*}
\end{proof}


\section{Useful facts from ergodic theory} \label{sec: useful facts}

Our goal is now to prove Theorem \ref{thm:dynamical sumset}.
In the course of the proof, we will utilize several results from ergodic theory, which we collect in this short section.

\subsection{The van der Corput lemma}

A fundamental result used for ``complexity reduction'' in the analysis of multiple ergodic averages is the van der Corput lemma, which we will use in the following form.

\begin{lemma}[van der Corput lemma {\cite[Lemma 2.1]{shalom}}] \label{lem:vdC}
    Let $\mathcal{H}$ be a Hilbert space, let $\Gamma$ be a countable discrete abelian group, and let $u : \Gamma \to \mathcal{H}$ be a bounded sequence.
    Let $(\Phi_N)$ be a F{\o}lner sequence in $\Gamma$.
    Suppose that:
    \begin{itemize}
        \item the limit
            \begin{equation*}
                z(\delta) = \lim_{N \to \infty} \frac{1}{|\Phi_N|} \sum_{\gamma \in \Phi_N} \left\langle u(\gamma+\delta), u(\gamma) \right\rangle
            \end{equation*}
            exists for every $\delta \in \Gamma$, and
        \item there exists $K < \infty$ such that for any F{\o}lner sequence $(\Psi_M)$ in $\Gamma$,
            \begin{equation*}
                \limsup_{M \to \infty} \frac{1}{|\Psi_M|} \left| \sum_{\delta \in \Psi_M} z(\delta) \right| \le K.
            \end{equation*}
    \end{itemize}
    Then
    \begin{equation*}
        \limsup_{N \to \infty} \left\| \frac{1}{|\Phi_N|} \sum_{\gamma \in \Phi_N} u(\gamma) \right\|^2 \le K.
    \end{equation*}
\end{lemma}

\subsection{Fubini property of uniform Ces\`{a}aro limits}

Another useful tool for analyzing ergodic averages is a version of Fubini's theorem that holds for \emph{uniform Ces\`{a}ro limits}.
Given a countable discrete abelian group and a function $v$ on $\Gamma$ taking values in a Banach space $V$, we say that the \emph{uniform Ces\`{a}ro limit} of $(v(\gamma))_{\gamma \in \Gamma}$ exists and is equal to an element $v_0 \in V$ if for every F{\o}lner sequence $(\Phi_N)$ in $\Gamma$,
\begin{equation*}
    \lim_{N \to \infty} \frac{1}{|\Phi_N|} \sum_{\gamma \in \Phi_N} v(\gamma) = v_0.
\end{equation*}
In this case, we will write $v_0 = \UC_{\gamma \in \Gamma} v(\gamma)$.

\begin{lemma}[{\cite[Lemma 1.1]{bl-fubini}}] \label{lem:fubini}
    Let $V$ be a Banach space, let $\Gamma_1, \Gamma_2$ be countable discrete abelian groups, and let $v : \Gamma_1 \times \Gamma_2 \to V$ be bounded.
    Suppose that
    \begin{equation*}
        \UC_{(\gamma_1, \gamma_2) \in \Gamma_1 \times \Gamma_2} v(\gamma_1, \gamma_2)
    \end{equation*}
    exists and for each $\gamma_1 \in \Gamma_1$,
    \begin{equation*}
        \UC_{\gamma_2 \in \Gamma_2} v(\gamma_1, \gamma_2)
    \end{equation*}
    exists.
    Then
    \begin{equation*}
        \UC_{\gamma_1 \in \Gamma_1} \left( \UC_{\gamma_2 \in \Gamma_2} v(\gamma_1, \gamma_2) \right) = \UC_{(\gamma_1, \gamma_2) \in \Gamma_1 \times \Gamma_2} v(\gamma_1, \gamma_2).
    \end{equation*}
\end{lemma}

\subsection{Host--Kra seminorms}

We now define the Host--Kra seminorms associated to an ergodic $\Gamma$-system $(X, \Sigma_X, \mu_X, T_X)$.

\begin{definition}
    Let $(X, \Sigma_X, \mu_X, T_X)$ be an ergodic $\Gamma$-system.
    Define the sequence of \emph{Host--Kra seminorms} $\seminorm{U^k}{\cdot}$ on $L^{\infty}(\mu_X)$ inductively by
    \begin{itemize}
        \item $\seminorm{U^1}{f} = \left| \int_X f~d\mu_X \right|$, and
        \item $\seminorm{U^{k+1}}{f}^{2^{k+1}} = \UC_{\gamma \in \Gamma} \seminorm{U^k}{\overline{f} \cdot T^{\gamma}f}^{2^k}$ for $k \ge 1$.
    \end{itemize}
\end{definition}

Note that
\begin{equation*}
    \seminorm{U^1}{f}^2 = \UC_{\gamma \in \Gamma} \int_X \overline{f} \cdot T^{\gamma}f~d\mu_X
\end{equation*}
by the mean ergodic theorem.
One can then check by induction (using Lemma \ref{lem:fubini}) that for every $k \in \N$, one has
\begin{equation*}
    \seminorm{U^k}{f}^{2^k} = \UC_{\bm{\gamma} \in \Gamma^k} \int_X \prod_{\bm{\omega} \in \{0,1\}^k} T^{\bm{\omega} \cdot \bm{\gamma}} C^{|\bm{\omega}|} f~d\mu_X,
\end{equation*}
where $C$ is the complex conjugation map and $|\omega|=\sum_{i=1}^k \omega_i$

Host and Kra \cite{host-kra} proved that $\seminorm{U^k}{\cdot}$ is a seminorm for each $k \in \N$, and there is a corresponding sequence of factors (the \emph{Host--Kra factors} $(Z_k)_{k \ge 0}$) determined by the relation
\begin{equation*}
    \E[f \mid Z_{k-1}] = 0 \iff \seminorm{U^k}{f} = 0.
\end{equation*}
(Strictly speaking, Host and Kra only proved their results for $\Z$-actions.
However, the arguments are easily adapted to abelian groups; see \cite[Appendix A]{btz1} for details.)
It is well known that $Z=Z_1$ is isomorphic to the Kronecker factor of $X$, and $Z_2$ is called its \emph{Conze--Lesigne factor}.

\begin{lemma} \label{lem:uniformity seminorms}
Let $(X, \Sigma_X, \mu_X, T_X)$ be an ergodic $\Gamma$-system. 
    Let $f \in L^{\infty}(\mu_X)$ with $\|f\|_{\infty} \le 1$.
    Then for $k \in \N$,
    \begin{equation*}
        \seminorm{U^k}{f}^{2^k} \le \left\| \E[f \mid Z_{k-1}] \right\|_{L^1(\mu_X)}.
    \end{equation*}
\end{lemma}

\begin{proof}
    Let $\tilde{f} = \E[f \mid Z_{k-1}]$.
    Then $\seminorm{U^k}{f - \tilde{f}} = 0$, so
    \begin{multline*}
        \seminorm{U^k}{f}^{2^k} = \seminorm{U^k}{\tilde{f}}^{2^k} = \UC_{\bm{\gamma} \in \Gamma^k} \int_X \prod_{\bm{\omega} \in \{0,1\}^k} T^{\omega \cdot \bm{\gamma}} C^{|\bm{\omega}|} \tilde{f}~d\mu_X \\ = \int_X \tilde{f} \cdot \UC_{\bm{\gamma} \in \Gamma^k} \prod_{\bm{\omega} \ne \bm{0}} T^{\bm{\omega} \cdot \bm{\gamma}} C^{|\bm{\omega}|}\tilde{f}~d\mu_X.
    \end{multline*}
    The claim then follows immediately by H\"{o}lder's inequality.
\end{proof}

\begin{lemma} \label{lem:U^2 Fourier}
Let $(X, \Sigma_X, \mu_X, T_X)$ be an ergodic $\Gamma$-system. 
    Let $f \in L^{\infty}(\mu_X)$, and let $g \in L^{\infty}(Z)$ be the function defined by $g \circ \pi_{Z} = \E[f \mid Z]$.
    Then
    \begin{equation*}
        \seminorm{U^2}{f} = \|\hat{g}\|_{\ell^4(\hat{Z})}.
    \end{equation*}
\end{lemma}

\begin{proof}
    Since $\seminorm{U^2}{f - \E[f \mid Z]} = 0$, it suffices to prove
    \begin{equation*}
        \seminorm{U^2}{g} = \|\hat{g}\|_{\ell^4(\hat{Z})},
    \end{equation*}
    where we compute the $U^2$ seminorm in the rotational $\Gamma$-system $(Z, \Sigma_Z, \mu_Z, T_Z)$.
    Let $\psi : \Gamma \to Z$ be the homomorphism inducing the action $T_Z$ by $T_Z^{\gamma}z = z + \psi(\gamma)$.
    Writing out the expression for the $U^2$ seminorm and using unique ergodicity of $(Z, T_Z)$ (Proposition \ref{prop: Kronecker}), we have
    \begin{align*}
        \seminorm{U^2}{g}^4 & = \UC_{(\gamma,\delta) \in \Gamma^2} \int_Z g(z) \cdot \overline{g(z + \psi(\gamma))} \cdot \overline{g(z + \psi(\delta))} \cdot g(z + \psi(\gamma) + \psi(\delta))~d \mu_Z \\
         & = \int_{Z^3} g(z) \cdot \overline{g(z+u)} \cdot \overline{g(z+v)} \cdot g(z+u+v)~d \mu_Z(z)~d \mu_Z(u)~d \mu_Z(v).
    \end{align*}
    Then, expanding $g$ as a Fourier series and applying orthogonality of characters,
    \begin{equation*}
        \seminorm{U^2}{g}^4 = \sum_{\chi \in \hat{Z}} |\hat{g}(\chi)|^4 = \|\hat{g}\|_{\ell^4(\hat{Z})}^4.
    \end{equation*}
\end{proof}

\subsection{A uniform Szemer\'{e}di theorem}

Finally, we will need the following variant of Szemer\'{e}di's theorem.

\begin{theorem}[Uniform Szemer\'{e}di theorem] \label{thm:uniform Sz}
    Let $\Gamma$ be a countable discrete abelian group.
    Let $k \in \N$ and $\delta > 0$.
    Then there exists $c > 0$ with the following property: if $(X, \Sigma_X, \mu_X, T_X)$ is a measure-preserving $\Gamma$-system, $E \subseteq X$ is a measurable set with $\mu_X(E) \ge \delta$, and $F \subseteq \Z$ is a finite subset of cardinality $|F| \le k$, then
    \begin{equation*}
        \UC_{\gamma \in \Gamma} \mu_X \left( \bigcap_{j \in F} T_X^{-j\gamma}E \right) \ge c.
    \end{equation*}
\end{theorem}

\begin{proof}
    The limit $\UC_{\gamma \in \Gamma} \mu_X \left( \bigcap_{j \in F} T_X^{-j\gamma}E \right)$ exists by \cite{z-k_Sz}.
    To obtain a lower bound, we use the main result of \cite{jp}.
    By \cite[Theorem 1.5]{jp}, there exist constants $c'$ and $K$ depending only on $k$ and $\delta$ such that the following holds: if $\Gamma$ is an abelian group, $T_1, \dots, T_{k-1}$ are commuting measure-preserving actions of $\Gamma$ on a probability space $(X, \Sigma_X,\mu_X)$, and $E$ is a measurable set with $\mu_X(E) \ge \delta$, then $\Gamma$ is covered by at most $K$ translates of the set
    \begin{equation*}
        \left\{ \gamma \in \Gamma : \mu_X \left( E \cap T_1^{-\gamma}E \cap (T_1T_2)^{-\gamma}E \cap \dots \cap (T_1 \dots T_{k-1})^{-\gamma}E \right) \ge c' \right\}.
    \end{equation*}

    Let $(X, \Sigma_X, \mu_X, T_X)$ be a measure-preserving $\Gamma$-system, let $E \subseteq X$ be a measurable set with $\mu_X(E) \ge \delta$, and let $F \subseteq \Z$ with $|F| \le k$.
    Write $F = \{j_1, j_2, \dots, j_l\}$ with $l \le k$.
    Put $T_i^{\gamma} = T_X^{(j_{i+1} - j_i)\gamma}$ for $i \in \{1, \dots, l-1\}$ and $T_i = \id_X$ for $i \in \{l, \dots, k-1\}$.
    Then since $T_X$ is measure-preserving, we have
    \begin{align*}
        \mu_X & \left( E \cap T_1^{-\gamma}E \cap (T_1T_2)^{-\gamma}E \cap \dots \cap (T_1 \dots T_{k-1})^{-\gamma}E \right) \\
         & = \mu_X \left( E \cap T_X^{-(j_2-j_1)\gamma} \cap T_X^{-(j_3-j_2)\gamma}T_X^{-(j_2-j_1)\gamma}E \cap \dots \cap T_X^{-(j_l - j_{l-1})\gamma} \dots T_X^{-(j_2-j_1)\gamma}E \right) \\
         & = \mu_X \left( T_X^{-j_1\gamma}E \cap T_X^{-j_2\gamma}E \cap \dots \cap T_X^{-j_l\gamma}E \right),
    \end{align*}
    so at most $K$ many translates of
    \begin{equation*}
        R_{c'} = \left\{ \gamma \in \Gamma : \mu_X \left( \bigcap_{j \in F} T_X^{-j\gamma}E \right) \ge c'\right\}
    \end{equation*}
    are needed to cover $\Gamma$.
    Therefore,
    \begin{equation*}
        \UC_{\gamma \in \Gamma} \mu_X \left( \bigcap_{j \in F} T_X^{-j\gamma}E \right) \ge \frac{c'}{K}.
    \end{equation*}
    Thus, the theorem holds by taking $c = \frac{c'}{K}$.
\end{proof}


\section{Progressive measures}

We are unable to give an explicit description of the space of Erd\H{o}s progressions outside of very simple examples.
Instead, we prove existence of Erd\H{o}s progessions indirectly using an appropriate measure.
Recall that the \emph{support} of a Borel probability measure $\nu$ on a compact metric space $Y$, denoted $\supp(\nu)$, is the smallest closed subset of $Y$ with full measure.

\begin{definition} \label{defn:progressive}
    Let $(X, T_X)$ be a topological dynamical $\Gamma$-system, and let $a \in X$.
    A Borel probability measure $\sigma$ on $X^{k-1}$ is called \emph{$k$-progressive from $a$} if
    \begin{equation*}
        \text{supp}(\sigma) \subseteq \overline{EP_k(a)}.
    \end{equation*}
\end{definition}

While Definition \ref{defn:progressive} is inspired by \emph{progressive measures} as defined in \cite[Definition 3.1]{kmrr_finite_sums}, the two notions are not identical.
We explain the relationship in more detail after proving a criterion for checking that a measure $\sigma$ on $X^{k-1}$ is $k$-progressive from $a$ (Proposition \ref{prop:progressive criterion} below).
The utility of Definition \ref{defn:progressive} is expressed by the following lemma:

\begin{lemma} \label{lem:positive sigma measure implies EP}
    Let $(X, T_X)$ be a topological dynamical $\Gamma$-system.
    Let $k \ge 2$.
    Suppose $a \in X$ and $U_1, \dots, U_{k-1}$ are open subsets of $X$.
    Suppose $\sigma$ is $k$-progressive from $a$.
    If $\sigma \left( U_1 \times \dots \times U_{k-1} \right) > 0$, then there exists a $k$-term Erd\H{o}s progression $(a, x_1, \dots, x_{k-1})$ with $x_i \in U_i$ for $i \in \{1, \dots, k-1\}$.
\end{lemma}

\begin{proof}
    This lemma follows almost immediately from the definition of a progressive measure.
    Since $\text{supp}(\sigma) \subseteq \overline{EP_k(a)}$ and $\sigma \left( U_1 \times \dots \times U_{k-1} \right) > 0$, we have
    \begin{equation*}
        \left( U_1 \times \dots \times U_{k-1} \right) \cap \overline{EP_k(a)} \ne \emptyset.
    \end{equation*}
    From the definition of the topological closure of a set, we then have
    \begin{equation*}
        \left( U_1 \times \dots \times U_{k-1} \right) \cap EP_k(a) \ne \emptyset.
    \end{equation*}
    That is, there exists a $k$-term Erd\H{o}s progression $(a, x_1, \dots, x_{k-1})$ with $x_i \in U_i$ for $i \in \{1, \dots, k-1\}$.
\end{proof}

The next proposition gives a criterion for a measure to be $k$-progressive from a point $a$ in terms of a recurrence property for the measure.

\begin{proposition} \label{prop:progressive criterion}
    Let $(X, T_X)$ be a topological dynamical $\Gamma$-system, and let $a \in X$.
    Let $\sigma$ be a Borel probability measure on $X^{k-1}$.
    Suppose that for any open sets $U_1, \dots, U_{k-1} \subseteq X$ with $\sigma(U_1 \times \dots \times U_{k-1}) > 0$, there exist infinitely many $\gamma \in \Gamma$ such that $T^{\gamma}a \in U_1$ and
    \begin{equation*}
        \sigma \left( \left( U_1 \times \dots \times U_{k-2} \times U_{k-1} \right) \cap \left( T^{-\gamma}U_2 \times \dots \times T^{-\gamma}U_{k-1} \times X \right) \right) > 0.
    \end{equation*}
    Then $\sigma$ is $k$-progressive from $a$.
\end{proposition}

The property of $\sigma$ in the hypothesis of Proposition \ref{prop:progressive criterion} is very close to the definition of a progressive measure from \cite{kmrr_finite_sums}.
The natural generalization of \cite[Definition 3.1]{kmrr_finite_sums} to the context of abelian groups says that a measure $\tau$ on $X^k$ is progressive if for any open sets $U_1, \dots, U_{k-1} \subseteq X$ with $\tau(X \times U_1 \times \dots \times U_{k-1}) > 0$, there are infinitely many $\gamma \in \Gamma$ such that
\begin{equation*}
    \tau \left( \left( X \times U_1 \times \dots \times U_{k-2} \times U_{k-1} \right) \cap \left( T^{-\gamma}U_1 \times T^{-\gamma}U_2 \times \dots \times T^{-\gamma}U_{k-1} \times X \right) \right) > 0.
\end{equation*}
With this terminology, Proposition \ref{prop:progressive criterion} can be rephrased as saying: if $\tau = \delta_a \times \sigma$ is progressive, then $\sigma$ is $k$-progressive from $a$.
A proof of this fact is given in \cite[Proposition 3.2]{kmrr_finite_sums} in the case $\Gamma = \Z$, and the argument generalizes to arbitrary abelian groups without difficulty, so we omit the proof.

Proposition \ref{prop:progressive criterion} has a simple modification in terms of continuous functions.
Namely, if for any nonnegative continuous functions $f_1, \dots, f_{k-1} \in C(X)$ with $\int_{X^{k-1}} \bigotimes_{i=1}^{k-1} f_i~d\sigma > 0$, there are infinitely many $\gamma \in \Gamma$ such that
\begin{equation*}
    f_1(T_X^{\gamma}a) \int_{X^{k-1}} \bigotimes_{i=1}^{k-1} \left( f_i \cdot T_X^{\gamma} f_{i+1} \right)~d\sigma > 0,
\end{equation*}
where $f_k = 1$, then $\sigma$ is $k$-progressive from $a$.
This observation leads immediately to the following corollary.

\begin{corollary} \label{cor:progressive criterion average}
    Let $(X, T_X)$ be a topological dynamical $\Gamma$-system, and let $a \in X$.
    Let $\sigma$ be a Borel probability measure on $X^{k-1}$, and suppose there is a F{\o}lner sequence $(\Phi_N)$ such that for any nonnegative continuous functions $f_1, \dots, f_{k-1} \in C(X)$, one has
    \begin{equation*}
        \int_{X^{k-1}} \bigotimes_{i=1}^{k-1} f_i~d\sigma > 0 \implies \lim_{N \to \infty} \frac{1}{|\Phi_N|} \sum_{\gamma \in \Phi_N} f_1(T_X^{\gamma}a) \int_{X^{k-1}} \bigotimes_{i=1}^{k-1} \left( f_i \cdot T_X^{\gamma} f_{i+1} \right)~d\sigma > 0,
    \end{equation*}
    where $f_k = 1$.
    Then $\sigma$ is $k$-progressive from $a$.
\end{corollary}

\section{Erd\H{o}s progressions in nilpotent translational systems}

Before constructing progressive measures in general systems, we begin with a construction in degree 2 nilpotent translational systems, which turn out to play a fundamental role in the proof of Theorem \ref{thm:dynamical sumset}.

As we have already seen in Theorem \ref{thm: semi-simple}, one important property of nilpotent translational systems is that they are semi-simple.
This leads to enhanced recurrence properties in translational systems.

\begin{proposition} \label{prop: semi-simple uniformly recurrent}
    Let $(X, T_X)$ be a topological dynamical $\Gamma$-system, and assume that $(X, T_X)$ is semi-simple.
    Then every point $x \in X$ is uniformly recurrent.
    That is, if $x \in X$ and $U \subseteq X$ is an open neighborhood of $x$, then
    \begin{equation*}
        \{\gamma \in \Gamma: T_X^{\gamma}x \in U\}
    \end{equation*}
    is syndetic.\footnote{A set $S \subseteq \Gamma$ is syndetic if $\Gamma$ can be covered by finitely many translates of $S$, that is, $\Gamma = \bigcup_{i=1}^K (S - t_i)$ for some $K \in \N$ and $t_1, \dots, t_K \in \Gamma$.}
\end{proposition}

\begin{proof}
    By assumption, the orbit closure $\O(x) = \overline{\{T_X^{\gamma}x : \gamma \in \Gamma\}}$ is minimal.
    It then follows from \cite[Theorem 1.15]{furstenberg_book} that $x$ is uniformly recurrent.
\end{proof}

A consequence for Erd\H{o}s progressions in translational systems is the following:

\begin{proposition} \label{prop:APs are EPs}
    Let $(G/\Lambda, T_{G/\Lambda})$ be a degree 2 nilpotent translational $\Gamma$-system, where the homomorphism $\phi\colon \Gamma\to G$ induces the $\Gamma$-action $T_{G/\Lambda}$.
    Let $x \in G/\Lambda$ and $\gamma \in \Gamma$.
    Then $(x, \phi(\gamma)x, \phi(2\gamma)x, \phi(3\gamma)x) \in EP_4$.
\end{proposition}

\begin{proof}
    We want to find a sequence $(\gamma_n)_{n \in \N}$ of distinct elements of $\Gamma$ such that
    \begin{equation*}
        \lim_{n \to \infty} \left( \phi(\gamma_n)x, \phi(\gamma_n)\phi(\gamma)x, \phi(\gamma_n)\phi(2\gamma)x \right) = \left( \phi(\gamma)x, \phi(2\gamma)x, \phi(3\gamma)x \right).
    \end{equation*}
    The action of $\Gamma$ on $X^3$ given by $\gamma \cdot (x_1,x_2,x_3) = (\phi(\gamma)x_1, \phi(\gamma)x_2, \phi(\gamma)x_3)$ is distal, so every point (in particular, the point $(\phi(\gamma)x, \phi(2\gamma)x, \phi(3\gamma)x)$) is uniformly recurrent under this action by Proposition \ref{prop: semi-simple uniformly recurrent}.
    Hence, there exists a sequence $(\delta_n)_{n \in \N}$ of distinct elements of $\Gamma$ such that
    \begin{equation*}
        \lim_{n \to \infty} (\phi(\delta_n) \phi(\gamma)x, \phi(\delta_n) \phi(2\gamma)x, \phi(\delta_n) \phi(3\gamma)x) = (\phi(\gamma)x, \phi(2\gamma)x, \phi(3\gamma)x).
    \end{equation*}
    Taking $\gamma_n = \delta_n + \gamma$ completes the proof.
\end{proof}

By Theorem \ref{thm: main}, given a point $x_0 \in X$, the orbit closure
\begin{equation*}
    \O_{\phi \times \phi^2 \times \phi^3}(x_0) = \overline{\{(\phi(\gamma)x_0, \phi(2\gamma)x_0, \phi(3\gamma)x_0) : \gamma \in \Gamma\}}
\end{equation*}
is of the form $H(x_0, x_0, x_0)$ for some closed subgroup $H \subseteq G^3$ and supports a unique $H$-invariant probability measure, which we will denote by $\sigma^{(4)}_{x_0}$.
By Proposition \ref{prop:APs are EPs}, the measure $\sigma^{(4)}_{x_0}$ is 4-progressive from $x_0$.
In the following theorem, we prove several additional properties of the measure $\sigma^{(4)}_{x_0}$.

\begin{theorem} \label{thm:properties of sigma}
    Let $(G/\Lambda,T_{G/\Lambda})$ be a minimal degree 2 nilpotent translational system, where the homomorphism $\phi\colon \Gamma\to G$ induces the $\Gamma$-action $T_{G/\Lambda}$.
    Let $x_0 \in G/\Lambda$.
    Let $\sigma^{(4)}_{x_0}$ be the Haar measure on the orbit closure
    \begin{equation*}
        \O_{\phi \times \phi^2 \times \phi^3}(x_0,x_0,x_0) = \overline{\{(\phi(\gamma)x_0, \phi(2\gamma)x_0, \phi(3\gamma)x_0) : \gamma \in \Gamma\}}
    \end{equation*}
    as defined above.
    \begin{enumerate}[(1)]
        \item For each $j \in \{1,2,3\}$, if $[\Gamma:j\Gamma] < \infty$, then the measure $\mu_{X_j}\coloneqq (\pi_j)_* \sigma^{(4)}_{x_0}$ is absolutely continuous with respect to the Haar measure $\mu_{G/\Lambda}$ on $G/\Lambda$, where $\pi_j$ is the projection from $\O_{\phi \times \phi^2 \times \phi^3}(x_0,x_0,x_0)$ onto $\O_{\phi^j}(x_0) = \overline{\{\phi(j\gamma) x_0 : \gamma \in \Gamma\}}$. 
        \item If $[\Gamma,6\Gamma]<\infty$, let $f_1, f_2, f_3 \in C(X)$, then
            \begin{multline*}
                \int_{X^3} \bigotimes_{j=1}^3 f_j~d\sigma^{(4)}_{x_0} \\ = \int_{[G,G]/[\Lambda,\Lambda]} \int_{G/\Lambda} f_1(x_0y\Lambda) f_2(x_0y^2z\Lambda)f_3(x_0y^3z^3\Lambda)~d\mu_{G/\Lambda}(y)~d\mu_{[G,G]/[\Lambda,\Lambda]}(z).
            \end{multline*}
        \item If $(\Phi_N)$ is a F{\o}lner sequence in $\Gamma$, then
            \begin{equation*}
                \lim_{N \to \infty} \frac{1}{|\Phi_N|} \sum_{\gamma \in \Phi_N} \sigma^{(4)}_{T_{G/\Lambda}^{\gamma}x_0} = \lim_{N \to \infty} \frac{1}{|\Phi_N|} \sum_{\gamma \in \Phi_N} \left( \id_{G/\Lambda} \times T_{G/\Lambda}^{\gamma} \times T_{G/\Lambda}^{2\gamma} \right)_* \mu_{\Delta}.
        \end{equation*}
        where both limits are taken in the weak* topology and $\mu_{\Delta}$ is the Haar measure on the diagonal $\Delta = \{(x,x,x) : x \in G/\Lambda\} \subseteq (G/\Lambda)^3$.
    \end{enumerate}
\end{theorem}

\begin{proof}
    (1) Let $X_j = \O_{\phi^j}(x_0) = \overline{\{\phi(j\gamma) x_0 : \gamma \in \Gamma\}}$.
    Suppose $[\Gamma:j\Gamma] = n < \infty$, and let $t_1, \dots, t_n \in \Gamma$ such that $\Gamma = \bigcup_{k=1}^n (j\Gamma + t_k)$.
    Then by minimality of the system $(G/\Lambda, T_{G/\Lambda})$, we have
    \begin{equation*}
        G/\Lambda = \overline{\{\phi(\gamma)x_0 : \gamma \in \Gamma\}} = \bigcup_{k=1}^n \overline{\{\phi(j\gamma + t_k)x_0 : \gamma \in \Gamma\}} = \bigcup_{k=1}^n \phi(t_k) X_j.
    \end{equation*}
    Therefore, $\mu_{G/\Lambda}(X_j) \ge \frac{1}{n} > 0$.
    By uniqueness of the Haar measure, it follows that $\mu_{X_j}$ is the measure
    \begin{equation*}
        \mu_{X_j}(B) = \frac{\mu_{G/\Lambda}(B \cap X_j)}{\mu_{G/\Lambda}(X_j)}.
    \end{equation*}
    In particular, $\mu_{X_j} \le n \cdot \mu_{G/\Lambda}$.

    (2) See Corollary \ref{cor: convergence}.

    (3) We use the Fubini property of uniform Ces\`{a}ro averages (Lemma \ref{lem:fubini}).
    Consider the $\Gamma \times \Gamma$ sequence
    \begin{equation*}
        \rho_{(\gamma_1,\gamma_2)} = \delta_{\phi(\gamma_1)\phi(\gamma_2)x_0} \times \delta_{\phi(\gamma_1)\phi(2\gamma_2)x_0} \times \delta_{\phi(\gamma_1)\phi(3\gamma_2)x_0}.
    \end{equation*}
    Putting $Y = X \times X \times X$, $y_0 = (x_0,x_0,x_0) \in Y$ and $\psi : \Gamma \times \Gamma \to G \times G \times G$ defined by $\psi(\gamma_1,\gamma_2) = (\phi(\gamma_1 + \gamma_2), \phi(\gamma_1 + 2\gamma_2), \phi(\gamma_1 + 3\gamma_2))$, we have
    \begin{equation*}
        \rho_{(\gamma_1,\gamma_2)} = \delta_{\psi(\gamma_1,\gamma_2)y_0}.
    \end{equation*}
    By Theorem \ref{thm: main}, the sequence $(\psi(\gamma_1,\gamma_2)y_0)_{(\gamma_1,\gamma_2) \in \Gamma \times \Gamma}$ is well-distributed in its orbit closure, so the uniform Ces\`{a}ro average
    \begin{equation*}
        \UC_{(\gamma_1,\gamma_2) \in \Gamma \times \Gamma} \rho_{(\gamma_1,\gamma_2)}
    \end{equation*}
    exists in the weak* topology.
    
    Now, for fixed $\gamma_1 \in \Gamma$, the limit
    \begin{equation*}
        \UC_{\gamma_2 \in \Gamma} \rho_{(\gamma_1,\gamma_2)}
    \end{equation*}
    exists and is equal to the Haar measure $\sigma^{(4)}_{T_{G/\Lambda}^{\gamma_1}x_0}$ on the orbit closure
    \begin{equation*}
        \O_{\phi \times \phi^2 \times \phi^3}(\phi(\gamma_1)x_0, \phi(\gamma_1)x_0, \phi(\gamma_1)x_0)
    \end{equation*}
    by another application of Theorem \ref{thm: main}.
    
    On the other hand, for fixed $\gamma_2 \in \Gamma$, the limit
    \begin{equation*}
        \UC_{\gamma_1 \in \Gamma} \rho_{(\gamma_1,\gamma_2)}
    \end{equation*}
    exists and is equal to
    \begin{equation*}
        (T_{G/\Lambda}^{\gamma_2} \times T_{G/\Lambda}^{2\gamma_2} \times T_{G/\Lambda}^{3\gamma_2})_* \mu_{\Delta},
    \end{equation*}
    since $(\phi(\gamma)x_0, \phi(\gamma)x_0, \phi(\gamma)x_0)_{\gamma \in \Gamma}$ is well-distributed in $\Delta$.
    The measure $\mu_{\Delta}$ is $(T_{G/\Lambda}^{\gamma_2} \times T_{G/\Lambda}^{\gamma_2} \times T_{G/\Lambda}^{\gamma_2})$-invariant, so we may eliminate a factor of $\gamma_2$ to obtain
    \begin{equation*}
        \UC_{\gamma_1 \in \Gamma} \rho_{(\gamma_1,\gamma_2)} = (\id_{G/\Lambda} \times T_{G/\Lambda}^{\gamma_2} \times T_{G/\Lambda}^{2\gamma_2})_* \mu_{\Delta}.
    \end{equation*}
    
    Thus, by Lemma \ref{lem:fubini}, we have
    \begin{equation*}
        \UC_{\gamma \in \Gamma} \sigma^{(4)}_{T_{G/\Lambda}^{\gamma}x_0} = \UC_{\gamma \in \Gamma} \left( \id_{G/\Lambda} \times T_{G/\Lambda}^{\gamma} \times T_{G/\Lambda}^{2\gamma} \right)_* \mu_{\Delta}.
    \end{equation*}
\end{proof}


\section{Lifting to a progressive measure} \label{sec:lifting}

Throughout this section, we fix the following data:

\begin{itemize}
    \item $\Gamma$ is a countably infinite abelian group with $[\Gamma : 6\Gamma] < \infty$.
    \item $(X, T_X)$ is a topological dynamical $\Gamma$-system.
    \item $\mu_X$ is a $T_X$-invariant ergodic Borel probability measure.
    \item $\Phi = (\Phi_N)$ is a F{\o}lner sequence in $\Gamma$.
    \item $a \in X$ is a point such that $a \in \text{gen}(\mu_X,\Phi)$.
\end{itemize}

Additionally\footnote{These additional assumptions are addressed in Section \ref{sec:3.6} below.}, we assume that the Kronecker factor $(Z, \Sigma_Z, \mu_Z, T_Z)$ and Conze--Lesigne factor (order 2 Host--Kra factor) $(Z_2, \Sigma_{Z_2}, \mu_{Z_2}, T_{Z_2})$ of $(X, \Sigma_X, \mu_X, T_X)$ arise topologically in the following sense:

\begin{itemize}
    \item $(Z_2, T_{Z_2})$ is a uniquely ergodic topolgical dynamical $\Gamma$-system equal to an inverse limit of degree 2 nilpotent translational systems.
    \item $(Z, T_Z)$ is a uniquely ergodic rotational system.
    \item There are topological factor maps $\pi_{Z_2} : X \to Z_2$ and $\pi_Z : X \to Z$.
\end{itemize}
Since $(Z_2, T_{Z_2})$ is topologically an inverse limit of degree 2 nilpotent translational systems, we may define a measure $\tilde{\sigma}^{(4)}_{\pi_{Z_2}(a)}$ on $Z_2^3$ via the limit
\begin{equation*}
    \tilde{\sigma}^{(4)}_{\pi_{Z_2}(a)} = \UC_{\gamma \in \Gamma} \delta_{T_{Z_2}^{\gamma} \pi_{Z_2}(a)} \times \delta_{T_{Z_2}^{2\gamma} \pi_{Z_2}(a)} \times \delta_{T_{Z_2}^{3\gamma} \pi_{Z_2}(a)}
\end{equation*}
as in the previous section.
Fix a disintegration
\begin{equation*}
    \mu_X = \int_{Z_2} \eta_z~ d\mu_{Z_2}
\end{equation*}
with respect to the factor map $\pi_{Z_2} : X \to Z_2$.
We then lift $\tilde{\sigma}^{(4)}_{\pi_{Z_2}(a)}$ to a measure $\sigma^{(4)}_a$ on $X^3$ by
\begin{equation*}
    \sigma^{(4)}_a = \int_{Z_2 \times Z_2 \times Z_2} \left( \eta_{z_1} \times \eta_{z_2} \times \eta_{z_3} \right)~d\tilde{\sigma}^{(4)}_{\pi_{Z_2}(a)}(z_1,z_2,z_3).
\end{equation*}
Note that the measure $\sigma^{(4)}_a$ does not depend on the choice of disintegration, since $\tilde{\sigma}^{(4)}_{\pi_{Z_2}(a)}$ has absolutely continuous marginals by the assumption $[\Gamma:6\Gamma] < \infty$ and Theorem \ref{thm:properties of sigma}.

The goal of this section is to prove the following theorem:

\begin{theorem} \label{thm:sigma is progressive}
    The measure $\sigma^{(4)}_a$ is 4-progressive from $a$.
\end{theorem}

We will prove Theorem \ref{thm:sigma is progressive} using Corollary \ref{cor:progressive criterion average}.
We break the proof into several lemmas.

\subsection{A limit formula for double ergodic averages}

\begin{lemma} \label{lem:double limit controlled by Kronecker}
    Let $f_1, f_2 \in C(X)$, and let $\tilde{f}_i \in L^{\infty}(Z)$ be the function satisfying $\tilde{f}_i \circ \pi_Z = \E[f_i \mid Z]$ for $i \in \{1,2\}$.
    Then
    \begin{equation*}
        \lim_{N \to \infty} \frac{1}{|\Phi_N|} \sum_{\gamma \in \Phi_N} f_1(T_X^{\gamma}a) f_2(T_X^{-\gamma}x) = \int_Z \tilde{f}_1(\pi_Z(a)+z) \cdot \tilde{f}_2(\pi_Z(x)-z)~d \mu_Z
    \end{equation*}
    in $L^2(\mu_X)$.
\end{lemma}

\begin{proof}
    By rescaling, we may assume $|f_i| \le 1$.
    Let $u(\gamma) = f_1(T_X^{\gamma}a) \cdot T_X^{-\gamma}f_2 \in L^2(\mu_X)$.
    Then since $\mu_X$ is $T_X$-invariant,
    \begin{align*}
        \left\langle u(\gamma+\delta), u(\gamma) \right\rangle & = \overline{f_1(T_X^{\gamma}a)} f_1(T_X^{\gamma+\delta}a) \int_X T_X^{-\gamma} \left( \overline{f_2} \cdot T_X^{-\delta}f_2 \right)~d\mu_X \\
        & = \overline{f_1(T_X^{\gamma}a)} f_1(T_X^{\gamma+\delta}a) \int_X \overline{f_2} \cdot T_X^{-\delta}f_2~d\mu_X.
    \end{align*}
    The point $a$ is generic for $\mu_X$ along $(\Phi_N)$, so
    \begin{align*}
        z(\delta) & = \lim_{N \to \infty} \frac{1}{|\Phi_N|} \sum_{\gamma \in \Phi_N} \left\langle u(\gamma+\delta), u(\gamma) \right\rangle \\
        & = \left( \int_X \overline{f_1} \cdot T_X^{\delta}f_1~d\mu_X \right) \left( \int_X \overline{f_2} \cdot T_X^{-\delta}f_2~d\mu_X \right).
    \end{align*}
    Therefore,
    \begin{equation*}
        |z(\delta)| \le \min \left\{ \left| \left\langle T_X^{\delta}f_1, f_1 \right\rangle \right|, \left| \left\langle T_X^{\delta}f_2, f_2 \right\rangle \right| \right\} = \min \left\{ \seminorm{U^1}{\overline{f_1} \cdot T_X^{\delta} f_1}, \seminorm{U^1}{\overline{f_2} \cdot T_X^{\delta} f_2} \right\} .
    \end{equation*}
    (Note that we have used here that $\left| \left\langle T_X^{-\delta}f_2, f_2 \right\rangle \right| = \left| \left\langle T_X^{\delta}f_2, f_2 \right\rangle \right|$, since $T_X^{\delta}$ is a unitary operator on $L^2(\mu_X)$.)
    Hence, for each $i \in \{1,2\}$ and any F{\o}lner sequence $(\Psi_M)$,
    \begin{multline*}
        \limsup_{M \to \infty} \frac{1}{|\Psi_M|} \left| \sum_{\delta \in \Psi_M} z(\delta) \right| \le \limsup_{M \to \infty} \frac{1}{|\Psi_M|} \sum_{\delta \in \Psi_M} \seminorm{U^1}{\overline{f_i} \cdot T_X^{\delta} f_i} \\ \le \limsup_{M \to \infty} \left( \frac{1}{|\Psi_M|} \sum_{\delta \in \Psi_M} \seminorm{U^1}{\overline{f_i} \cdot T_X^{\delta} f_i}^2 \right)^{1/2} = \seminorm{U^2}{f_i}^2.
    \end{multline*}
    By Lemma \ref{lem:vdC}, we conclude that
    \begin{equation} \label{eq:U^2 control}
        \limsup_{N \to \infty} \left\| \frac{1}{|\Phi_N|} \sum_{\gamma \in \Phi_N} f_1(T_X^{\gamma}a) \cdot T_X^{-\gamma}f_2 \right\|_{L^2(\mu_X)} \le \min_{i \in \{1,2\}} \seminorm{U^2}{f_i}.
    \end{equation}

    Let $\eps > 0$.
    Since the factor map $\pi_Z : X \to Z$ is continuous, we may decompose $f_i = g_i \circ \pi_Z + h_i$ as a sum of continuous functions with $g_i \in C(Z)$, $h_i \in C(X)$, and $\|\E[h_i \mid Z]\|_{L^2(\mu_X)} < \eps$.
    Then $\seminorm{U^2}{h_i} \ll \eps^{1/4}$ by Lemma \ref{lem:uniformity seminorms}, so by \eqref{eq:U^2 control},
    \begin{multline*}
        \limsup_{N \to \infty} \left\| \frac{1}{|\Phi_N|} \sum_{\gamma \in \Phi_N} f_1(T_X^{\gamma}a) \cdot T_X^{-\gamma}f_2 - \frac{1}{|\Phi_N|} \sum_{\gamma \in \Phi_N} g_1(T_Z^{\gamma} \pi_Z(a)) \cdot T_X^{-\gamma}(g_2 \circ \pi_Z) \right\|_{L^2(\mu_X)} \\ \ll \eps^{1/4}.
    \end{multline*}
    The limit on the Kronecker factor may be computed pointwise using unique ergodicity of $(Z, T_Z)$.
    Namely,
    \begin{equation*}
        \lim_{N \to \infty} \frac{1}{|\Phi_N|} \sum_{\gamma \in \Phi_N} g_1(T_Z^{\gamma} \pi_Z(a)) \cdot g_2(T_Z^{-\gamma} \pi_Z(x)) = \int_Z g_1(\pi_Z(a)+z) g_2(\pi_Z(x) - z)~d \mu_Z
    \end{equation*}
    for every $x \in X$.
    Now, since $\|\E[h_i \mid Z]\|_{L^2(\mu_X)} < \eps$, replacing $g_i$ in this integral by $\tilde{f}_i$ introduces an error of size $O(\eps)$.
    Thus,
    \begin{multline*}
        \limsup_{N \to \infty} \left\| \frac{1}{|\Phi_N|} \sum_{\gamma \in \Phi_N} f_1(T_X^{\gamma}a) f_2(T_X^{-\gamma}x) - \int_Z \tilde{f}_1(\pi_Z(a)+z) \cdot \tilde{f}_2(\pi_Z(x)-z)~d \mu_Z \right\|_{L^2(\mu_X)} \\ \ll \eps + \eps^{1/4}.
    \end{multline*}
    But $\eps > 0$ was arbitrary, so this completes the proof.
\end{proof}

\subsection{A limit formula for triple ergodic averages}

Before stating the next lemma, we need to introduce two auxiliary measures on $X^2$.
First, we define a measure $\nu$ on $X^2$ by
\begin{equation*}
    \nu(E) = \sigma^{(4)}_a(E \times X)
\end{equation*}
for Borel subsets $E \subseteq X^2$.
That is, $\nu$ is the projection of $\sigma^{(4)}_a$ onto the first two coordinates.

The second measure we will define is a measure previously studied by Kra--Moreira--Richter--Robertson \cite{kmrr_B+B} and Charamaras--Mountakis \cite{cm} in the context of the Erd\H{o}s $B+B+t$ problem (over $\Z$ and abelian groups, respectively).
Let 
\begin{equation*}
    \mu_X = \int_Z \zeta_z~d \mu_Z
\end{equation*}
be a disintegration of $\mu_X$ with respect to the factor map $\pi_Z : X \to Z$, and let
\begin{equation*}
    \tilde{\sigma}^{(3)}_a = \UC_{\gamma \in \Gamma} \delta_{T_Z^{\gamma}\pi_Z(a)} \times \delta_{T_Z^{2\gamma}\pi_Z(a)} =  \int_Z \delta_{\pi_Z(a) + t} \times \delta_{\pi_Z(a) + 2t}~d\mu_Z(t).
\end{equation*}
Then we define $\sigma^{(3)}_a$ on $X^2$ by
\begin{equation*}
    \sigma^{(3)}_a = \int_{Z^2} \zeta_u \times \zeta_v~d\tilde{\sigma}^{(3)}_a(u,v) = \int_Z \zeta_{\pi_Z(a) + t} \times \zeta_{\pi_Z(a)+2t}~d\mu_Z(t).
\end{equation*}
Though we will not use this fact directly, it was shown in \cite[Theorems 4.9 and 4.10]{cm} that $\sigma^{(3)}_a$ is 3-progressive from $a$.

The measures $\nu$ and $\sigma^{(3)}_a$ are closely related and often coincide, but they may differ for certain choices of the point $a$, as demonstrated by the following example.

\begin{example} \label{eg: skew-product}
    This example was communicated to us by Trist\'{a}n Radi\'{c}.
    We consider the $\Z$-action on the torus $\mathbb{T}^2$ given by $T(x,y) = (x+\alpha, y+2x+\alpha)$ for an irrational number $\alpha$.
    It can be checked that $(\mathbb{T}^2, T)$ is a minimal uniquely ergodic topological dynamical system, and $(\mathbb{T}^2, \mu_{\mathbb{T}^2}, T)$ is a measure-preserving system of order 2.
    Let $a = (0,0)$.
    Then the orbit $(T^na, T^{2n}a)$ can be computed explicitly as
    \begin{equation*}
        (T^na, T^{2n}a) = (n\alpha, n^2\alpha, 2n\alpha, 4n^2\alpha),
    \end{equation*}
    which equidistributes in the subtorus $H = \{(z,t,2z,4t) : z, t \in \mathbb{T}\}$ by Weyl's equidistribution theorem.
    Thus, $\nu$ is the Haar measure on $H$, while $\sigma^{(3)}_a$ is the Haar measure on the $3$-dimensional subtorus $\{(z,u,2z,v) : z, u, v \in \mathbb{T}\}$.
\end{example}

Even allowing for the possibility that $\nu$ and $\sigma^{(3)}_a$ differ, we have the following lemma.

\begin{lemma} \label{lem: reduce to Kronecker}
    Let $f, g \in L^2(\mu)$, and suppose $g$ is measurable with respect to the Kronecker factor $Z$.
    Then
    \begin{equation*}
        \int_{X^2} f \otimes g~d\nu = \int_{X^2} f \otimes g~d\sigma^{(3)}_a.
    \end{equation*}
\end{lemma}

\begin{proof}
    By construction,
    \begin{equation*}
        \int_{X^2} \E[f \mid Z] \otimes g~d\nu = \int_{X^2} f \otimes g~d\sigma^{(3)}_a.
    \end{equation*}
    Thus, decomposing $f = f_1 + f_2$ with $f_1 = \E[f \mid Z]$ and $f_2 = f - \E[f \mid Z]$, we have
    \begin{equation*}
        \int_{X^2} f \otimes g~d\nu = \int_{X^2} f_1 \otimes g~d\nu + \int_{X^2} f_2 \otimes g~d\nu = \int_{X^2} f \otimes g~d\sigma^{(3)}_a + \int_{X^2} f_2 \otimes g~d\nu.
    \end{equation*}
    Therefore, it suffices to prove $\int_{X^2} f_2 \otimes g~d\nu = 0$.

    We use a similar argument to the one appearing in \cite[Lemma 6.11]{kmrr_finite_sums}.
    Define a measure $\lambda$ on $X \times Z$ by
    \begin{equation*}
        \int_{X \times Z} h \otimes k~d\lambda = \int_{X^2} h \otimes (k \circ \pi_Z)~d\nu
    \end{equation*}
    for $h \in C(X)$ and $k \in C(Z)$.
    Since $g$ is assumed to be measurable with respect to $Z$, we can write $g = \tilde{g} \circ \pi_Z$ for some $\tilde{g} \in L^2(\mu_Z)$, whence
    \begin{equation*}
        \int_{X^2} f_2 \otimes g~d\nu = \int_{X \times Z} f_2 \otimes \tilde{g}~d\lambda.
    \end{equation*}
    Note that $\lambda$ is a joining of $(X, \mu_X, T)$ and $(Z, \mu_Z, T_Z^2)$.
    Hence, the function $\ind \otimes \tilde{g}$ is measurable with respect to the Kronecker factor of $(X \times Z, \lambda, T \times T_Z^2)$, while $f_2 \otimes \ind$ satisfies $\seminorm{U^2}{f_2 \otimes \ind} = \seminorm{U^2}{f_2} = 0$, so $(\ind \otimes g)$ and $(f_2 \otimes \ind)$ are orthogonal in $L^2(\lambda)$.
    We conclude that
    \begin{equation*}
        \int_{X^2} f_2 \otimes g~d\nu = \int_{X \times Z} f_2 \otimes \tilde{g}~d\lambda = \int_{X \times Z} (f_2 \otimes \ind) (\ind \otimes \tilde{g})~d\lambda = 0.
    \end{equation*}
\end{proof}

\begin{lemma} \label{lem:limit formula}
    Let $f_1, f_2, f_3 \in C(X)$.
    Then
    \begin{equation*}
        \lim_{N \to \infty} \frac{1}{|\Phi_N|} \sum_{\gamma \in \Phi_N} f_1(T_X^{\gamma}a) f_2(T_X^{\gamma} x_1) f_3(T_X^{\gamma} x_2) = \int_{Z_2^3} \tilde{f}_1\otimes\tilde{f}_2\otimes \tilde{f}_3~d\mu_{\O(\pi_{Z_2}(a),\pi_{Z_2}(x_1),\pi_{Z_2}(x_2))}
    \end{equation*}
    in $L^2(\nu)$, where $\tilde{f}_i \circ \pi_{Z_2} = \E[f_i \mid Z_2]$ and $\mu_{\O(\pi_{Z_2}(a),\pi_{Z_2}(x_1),\pi_{Z_2}(x_2))}$ is the unique diagonally-invariant measure on the orbit closure
    \begin{equation*}
        \O(\pi_{Z_2}(a),\pi_{Z_2}(x_1),\pi_{Z_2}(x_2)) = \overline{\left\{ \left(T_{Z_2}^{\gamma} \pi_{Z_2}(a), T_{Z_2}^{\gamma} \pi_{Z_2}(x_1), T_{Z_2}^{\gamma} \pi_{Z_2}(x_2)\right) : \gamma \in \Gamma \right\}} \subseteq Z_2^3.
    \end{equation*}
    \end{lemma}

\begin{remark}
    The same limit in $L^2(\mu_X \times \mu_X)$ is controlled by the Kronecker factor.
    The measure $\nu$ is supported on points $(x_1, x_2)$ such that $(a,x_1,x_2)$ projects to a 3-term arithmetic progression in the Kronecker factor, and this correlation between $x_1$ and $x_2$ in the Kronecker factor introduces higher-level correlations, leading to the limit in Lemma \ref{lem:limit formula} being controlled by the second order Host--Kra factor.
\end{remark}

\begin{proof}
    By rescaling, we may assume $|f_i| \le 1$.
    Let $u(\gamma) = f_1(T_X^{\gamma}a) \cdot T_X^{\gamma}f_2 \otimes T_X^{\gamma}f_3 \in L^2(\nu)$.
    Then since the measure $\nu$ is $T_X \times T_X^2$ invariant,
    \begin{align*}
        \left\langle u(\gamma+\delta), u(\gamma) \right\rangle & = \overline{f_1(T_X^{\gamma}a)} f_1(T_X^{\gamma+\delta}a) \int_{X^2} T_X^{\gamma} \left( \overline{f_2} \cdot T_X^{\delta}f_2 \right) \otimes T_X^{\gamma} \left( \overline{f_3} \cdot T_X^{\delta}f_3 \right)~d\nu \\
        & = \overline{f_1(T_X^{\gamma}a)} f_1(T_X^{\gamma+\delta}a) \int_{X^2} \left( \overline{f_2} \cdot T_X^{\delta}f_2 \right) \otimes T_X^{-\gamma} \left( \overline{f_3} \cdot T_X^{\delta}f_3 \right)~d\nu.
    \end{align*}
    We will use Lemma \ref{lem:double limit controlled by Kronecker} to average this expression over $\gamma$.
    First we set up some notation: for each $i \in \{1,2,3\}$, let $k_i \in L^{\infty}(Z)$ be the function satisfying $k_i \circ \pi_Z = \E[\overline{f_i} \cdot T_X^{\delta} f_i \mid Z]$.
    The second marginal of $\nu$ is absolutely continuous with respect to $\mu_X$ by Theorem \ref{thm:properties of sigma}, so applying Lemma \ref{lem:double limit controlled by Kronecker},
    \begin{align*}
        z(\delta) & = \lim_{N \to \infty} \frac{1}{|\Phi_N|} \sum_{\gamma \in \Phi_N} \left\langle u(\gamma+\delta), u(\gamma) \right\rangle \\
        & = \int_{X^2} \left( \overline{f_2} \cdot T_X^{\delta}f_2 \right)(x_1) \int_{Z} k_1(\pi_Z(a) + z) \cdot k_3(\pi_Z(x_2) - z)~dz~d\nu(x_1,x_2).
    \end{align*}
    The function
    \begin{equation*}
        x_2 \mapsto \int_{Z} k_1(\pi_Z(a) + z) \cdot k_3(\pi_Z(x_2) - z)~dz
    \end{equation*}
    is measurable with respect to the Kronecker factor $Z$.
    Therefore, by Lemma \ref{lem: reduce to Kronecker},
    \begin{align*}
        z(\delta) & = \int_{X^2} \left( \overline{f_2} \cdot T_X^{\delta}f_2 \right)(x_1) \int_Z k_1(\pi_Z(a)+z) \cdot k_3(\pi_Z(x_2)-z)~d \mu_Z~d\sigma^{(3)}_a(x_1,x_2) \\
        & = \int_{Z^2} k_1(\pi_Z(a)+z) \cdot k_2(\pi_Z(a)+t) \cdot k_3(\pi_Z(a)+2t-z)~d \mu_Z(z)~d\mu_Z(t) \\
        & = \int_{Z^2} k_1(u) \cdot k_2(u+v) \cdot k_3(u+2v)~d\mu_Z(u)~d\mu_Z(v),
    \end{align*}
    where in the last step we made a change of variables $u = \pi_Z(a) + z$, $v = t-z$.
    
    Expanding each $k_i$ as a Fourier series and using orthogonality of characters,
    \begin{equation*}
        z(\delta) = \sum_{\chi \in \hat{Z}} \hat{k}_1(\chi) \hat{k}_2(\chi^{-2}) \hat{k}_3(\chi).
    \end{equation*}
    By H\"{o}lder's inequality, Parseval's identity, and Lemma \ref{lem:U^2 Fourier}, we may bound
    \begin{multline*}
        |z(\delta)| \ll \|\hat{k}_{i_1}\|_{\ell^4(\hat{Z})} \cdot \|\hat{k}_{i_2}\|_{\ell^2(\hat{Z})} \cdot \|\hat{k}_{i_3}\|_{\ell^2(\hat{Z})} = \|\hat{k}_{i_1}\|_{\ell^4(\hat{Z})} \cdot \|k_{i_2}\|_{L^2(\mu_X)} \cdot \|k_{i_3}\|_{L^2(\mu_X)} \\ \le \|\hat{k}_{i_1}\|_{\ell^4(\hat{Z})} = \seminorm{U^2}{\overline{f_{i_1}} \cdot T_X^{\delta} f_{i_1}}
    \end{multline*}
    for any permutation $(i_1, i_2, i_3)$ of $(1,2,3)$.
    Note that this bound only holds up to a constant, since the Fourier coefficient $\hat{k}_2(\xi)$ may contribute to $z(\delta)$ with multiplicity
    \begin{equation*}
        \left| \left\{ \chi \in \hat{Z} : \chi^2 = \xi \right\} \right| \le [Z : 2Z] \le [\Gamma : 2\Gamma].
    \end{equation*}
    We thus conclude
    \begin{equation*}
        |z(\delta)| \ll \min_{i \in \{1,2,3\}} \seminorm{U^2}{\overline{f}_i \cdot T_X^{\delta} f_i}.
    \end{equation*}
    Hence, for any F{\o}lner sequence $(\Psi_M)$,
    \begin{align*}
        \limsup_{M \to \infty} \frac{1}{|\Psi_M|} \left| \sum_{\delta \in \Psi_M} z(\delta) \right| & \ll \min_{i \in \{1,2,3\}} \limsup_{M \to \infty} \frac{1}{|\Psi_M|} \sum_{\delta \in \Psi_M} \seminorm{U^2}{\overline{f}_i \cdot T_X^{\delta} f_i} \\
        & \le \min_{i \in \{1,2,3\}} \limsup_{M \to \infty} \left( \frac{1}{|\Psi_M|} \sum_{\delta \in \Psi_M} \seminorm{U^2}{\overline{f}_i \cdot T_X^{\delta} f_i}^2 \right)^{1/2} \\
        & = \min_{i \in \{1,2,3\}} \seminorm{U^3}{f_i}^2.
    \end{align*}
    By Lemma \ref{lem:vdC}, it follows that
    \begin{equation} \label{eq:U^3 control}
        \limsup_{N \to \infty} \left\| \frac{1}{|\Phi_N|} \sum_{\gamma \in \Phi_N} f_1(T_X^{\gamma}a) \cdot T_X^{\gamma}f_2 \otimes T_X^{\gamma}f_3 \right\|_{L^2(\nu)} \ll \min_{i \in \{1,2,3\}} \seminorm{U^3}{f_i}.
    \end{equation}

    To finish the proof, we use the same strategy as in Lemma \ref{lem:double limit controlled by Kronecker}.
    Decomposing $f_i = g_i \circ \pi + h_i$ with $g_i \in C(Z_2)$ and $h_i \in C(X)$ with $g_i$ approximating $\E[f_i \mid Z_2]$ in $L^2$, the inequality \eqref{eq:U^3 control} combined with Lemma \ref{lem:uniformity seminorms} shows that the terms involving $h_i$ are negligible.
    The original average can therefore be approximated by the corresponding average for the functions $g_i$, which in turn can be computed pointwise by Theorem \ref{thm: main}.
\end{proof}

\subsection{A limit formula for averages against $\sigma^{(4)}_a$}

\begin{lemma} \label{lem:characteristic factor}
    Let $f_1, f_2, f_3 \in C(X)$.
    Then
    \begin{multline*}
        \lim_{N \to \infty} \frac{1}{|\Phi_N|} \sum_{\gamma \in \Phi_N} f_1(T_X^{\gamma}a) \int_{X^3} f_1(x_1) f_2(T_X^{\gamma}x_1) f_2(x_2) f_3(T_X^{\gamma}x_2) f_3(x_3)~d\sigma^{(4)}_a(x_1, x_2, x_3) \\
          = \int_{X^3} f_1(x_1) f_2(x_2) f_3(x_3)\cdot \\ \left( \int_{Z_2^3} \tilde{f}_1\otimes \tilde{f}_2\otimes \tilde{f}_3~d\mu_{\O(\pi_{Z_2}(a),\pi_{Z_2}(x_1),\pi_{Z_2}(x_2))} \right)~d\sigma^{(4)}_a(x_1,x_2,x_3).
    \end{multline*}
\end{lemma}

\begin{proof}
    For $N \in \N$,
    \begin{multline*}
        \frac{1}{|\Phi_N|}  \sum_{\gamma \in \Phi_N} f_1(T_X^{\gamma}a) \int_{X^3} f_1(x_1) f_2(T_X^{\gamma}x_1) f_2(x_2) f_3(T_X^{\gamma}x_2) f_3(x_3)~d\sigma^{(4)}_a(x_1, x_2, x_3) \\
          - \int_{X^3} f_1(x_1) f_2(x_2) f_3(x_3)\cdot  \\ \left( \int_{Z_2^3} \tilde{f}_1\otimes \tilde{f}_2\otimes \tilde{f}_3~d\mu_{\O(\pi_{Z_2}(a),\pi_{Z_2}(x_1),\pi_{Z_2}(x_2))} \right)~d\sigma^{(4)}_a(x_1,x_2,x_3) \\
          = \int_{X^3} f_1(x_1) f_2(x_2) f_3(x_3)  \mathcal{E}_N(x_1, x_2)~d\sigma^{(4)}_a(x_1,x_2,x_3),
    \end{multline*}
    where
    \begin{multline*}
        \mathcal{E}_N(x_1, x_2) = \frac{1}{|\Phi_N|} \sum_{\gamma \in \Phi_N} f_1(T_X^{\gamma}a) f_2(T_X^{\gamma}x_1) f_3(T_X^{\gamma}x_2) \\- \int_{Z_2^3} \tilde{f}_1\otimes \tilde{f}_2\otimes \tilde{f}_3~d\mu_{\O(\pi_{Z_2}(a),\pi_{Z_2}(x_1),\pi_{Z_2}(x_2))}.
    \end{multline*}

    By the Cauchy--Schwarz inequality,
    \begin{align*}
        &\left| \int_{X^3} f_1(x_1) f_2(x_2) f_3(x_3)  \mathcal{E}_N(x_1, x_2)~d\sigma^{(4)}_a(x_1,x_2,x_3) \right| \\ & \le \|f_3\|_{\infty} \left| \int_{X^2} (f_1 \otimes f_2) \mathcal{E}_N~d\nu \right| \\
         & \le \|f_3\|_{\infty} \|f_1 \otimes f_2\|_{L^2(\nu)} \|\mathcal{E}_N\|_{L^2(\nu)}.
    \end{align*}
    By Lemma \ref{lem:limit formula}, we have $\|\mathcal{E}_N\|_{L^2(\nu)} \to 0$, and this completes the proof.
\end{proof}

\subsection{A version of Szemer\'{e}di's theorem}

The final ingredient is the following version of Szemer\'{e}di's theorem.

\begin{theorem} \label{thm:szemeredi}
    Suppose $B \subseteq \Gamma$ satisfies $d^*(B) = \delta > 0$.
    Then for any finite set $F \subseteq \Z$,
    \begin{equation*}
        d^* \left( \left\{ (\gamma_1, \gamma_2) \in \Gamma^2 : \{\gamma_1+k\gamma_2 : k \in F\} \subseteq B\right\} \right) \gg_{\delta, |F|} 1.
    \end{equation*}
\end{theorem}

\begin{proof}
    Let $R = \left\{ (\gamma_1, \gamma_2) \in \Gamma^2 : \{\gamma_1 + k\gamma_2\} \subseteq B\right\}$.
    
    By the Furstenberg correspondence principle, let $(Y, T_Y)$ be a topological dynamical $\Gamma$-system, $\mu_Y$ an ergodic $T_Y$-invariant measure, $\Psi = (\Psi_N)$ a F{\o}lner sequence, $b \in \text{gen}(\mu_Y, \Psi)$, and $E \subseteq Y$ clopen such that $B = \{\gamma \in \Gamma : T_Y^{\gamma}b \in E\}$ and $\mu_Y(E) \ge \delta$.
    By Theorem \ref{thm:uniform Sz},
    \begin{equation*}
        \UC_{\gamma \in \Gamma} \mu_Y \left( \bigcap_{k \in F} T_Y^{-k\gamma} E \right) \gg_{\delta, |F|} 1.
    \end{equation*}
    Note that
    \begin{align*}
        \mu_Y \left( \bigcap_{k \in F} T_Y^{-k\gamma_2} E \right) & = \lim_{N \to \infty} \frac{1}{|\Psi_N|} \sum_{\gamma_1 \in \Psi_N} \delta_{T_Y^{\gamma_1}b} \left( \bigcap_{k \in F} T_Y^{-k\gamma_2} E \right) \\
        & = \lim_{N \to \infty} \frac{1}{|\Psi_N|} \sum_{\gamma_1 \in \Psi_N} \prod_{k \in F} \ind_B(\gamma_1 + k\gamma_2) \\
        & = \lim_{N \to \infty} \frac{1}{|\Psi_N|} \sum_{\gamma_1 \in \Psi_N} \ind_R(\gamma_1, \gamma_2).
    \end{align*}
    Thus,
    \begin{equation*}
        d^*(R) \ge \UC_{\gamma_2 \in \Gamma} \left( \lim_{N \to \infty} \frac{1}{|\Psi_N|} \sum_{\gamma_1 \in \Psi_N} \ind_R(\gamma_1, \gamma_2) \right) \gg_{\delta, |F|} 1.
    \end{equation*}
\end{proof}

\subsection{Proof of Theorem \ref{thm:sigma is progressive}}

\begin{proof}[Proof of Theorem \ref{thm:sigma is progressive}]
    We use Corollary \ref{cor:progressive criterion average}.
    Let $f_1, f_2, f_3 \in C(X)$ be nonnegative continuous functions such that
    \begin{equation*}
        c = \int_{X^3} \bigotimes_{i=1}^3 f_i~d\sigma^{(4)}_a > 0.
    \end{equation*}
    We want to show
    \begin{equation} \label{eq:desired positivity}
        \lim_{N \to \infty} \frac{1}{|\Phi_N|} \sum_{\gamma \in \Phi_N} f_1(T_X^{\gamma}a) \int_{X^3} \bigotimes_{i=1}^3 \left( f_i \cdot T_X^{\gamma} f_{i+1} \right)~d\sigma^{(4)}_a > 0,
    \end{equation}
    where $f_4 = 1$.

    By scaling the functions $f_i$ if needed, we may assume without loss of generality that $0 \le f_i \le 1$.
    Let $\eps > 0$.
    We may write $f_i = g_i \circ \pi_{Z_2} + h_i$, where $g_i \in C(Z_2)$ with $0 \le g_i \le 1$, $h_i \in C(X)$, and $\left\| \E[h_i \mid Z_2] \right\|_{L^2(\mu_X)} < \eps$.
    By Lemma \ref{lem:characteristic factor}, we may compute the limit on the left hand side of \eqref{eq:desired positivity} as
    \begin{multline*}
        I(f_1, f_2, f_3) = \int_{X^3} f_1(x_1) f_2(x_2) f_3(x_3) \\ \left( \int_{Z_2^3} \tilde{f}_1\otimes \tilde{f}_2\otimes \tilde{f}_3~d\mu_{\O(\pi_{Z_2}(a),\pi_{Z_2}(x_1),\pi_{Z_2}(x_2))} \right)~d\sigma^{(4)}_a(x_1,x_2,x_3).
    \end{multline*}
    Since the measure $\sigma^{(4)}_a$ is lifted from the factor $Z_2$, expanding $f_i = g_i \circ \pi_{Z_2} + h_i$, each term in the integral expression for $I(f_1,f_2,f_3)$ involving $h_i$ can be bounded by $O(\eps)$, where the implicit constant depends on $[\Gamma:6\Gamma]$, since this controls the behavior of the marginals of $\sigma^{(4)}_a$.
    Hence,
    \begin{equation*}
        I(f_1, f_2, f_3) = I(g_1 \circ \pi_{Z_2}, g_2 \circ \pi_{Z_2}, g_3 \circ \pi_{Z_2}) + O(\eps).
    \end{equation*}
    
    Now, on the factor $Z_2$, we may use the construction of the measure $\tilde{\sigma}^{(4)}_{\pi_{Z_2}(a)}$ as a uniform Ces\`{a}ro average to compute $I(g_1 \circ \pi_{Z_2}, g_2 \circ \pi_{Z_2}, g_3 \circ \pi_{Z_2})$:
    \begin{multline*}
        I(g_1 \circ \pi_{Z_2}, g_2 \circ \pi_{Z_2}, g_3 \circ \pi_{Z_2}) \\ = \UC_{(\gamma_1, \gamma_2) \in \Gamma^2} u_1(\gamma_1) u_2(2\gamma_1) u_3(3\gamma_1) u_1(\gamma_2) u_2(\gamma_2 + \gamma_1) u_3(\gamma_2 + 2\gamma_1),
    \end{multline*}
    where $u_i(\gamma) = g_i(T_{Z_2}^{\gamma} \pi_{Z_2}(a))$ for $\gamma \in \Gamma$ and $i \in \{1,2,3\}$.
    
    By assumption,
    \begin{equation*}
        \UC_{\gamma \in \Gamma} \prod_{i=1}^3 u_i(i\gamma) = \int_{Z^3_2} \bigotimes_{i=1}^3 g_i~d\tilde{\sigma}^{(4)}_{\pi(a)} = c + O(\eps).
    \end{equation*}
    Let
    \begin{equation*}
        B = \left\{\gamma \in \Gamma : \prod_{i=1}^3 u_i(i\gamma) \ge \frac{c}{2} \right\},
    \end{equation*}
    Fix a F{\o}lner sequence $(\Psi_N)$ along which $B$ has positive density.
    Then
    \begin{equation*}
        \UC_{\gamma \in \Gamma} \prod_{i=1}^3 u_i(i\gamma) \le \lim_{N \to \infty} \frac{1}{|\Psi_N|} \sum_{\gamma \in \Psi_N} \left( \ind_B(\gamma) + \frac{c}{2} \ind_{\Gamma \setminus B}(\gamma) \right) = d_{\Psi}(B) + \frac{c}{2} \left( 1 - d_{\Psi}(B) \right).
    \end{equation*}
    Therefore, $d_{\Psi}(B) \ge \frac{c+O(\eps)}{2-c}$.
    Taking $\eps$ sufficiently small (compared to $c$), we may assume $d_{\Psi}(B) \ge \frac{c}{2}$.
    Define
    \begin{equation*}
        R = \left\{ (\gamma_1, \gamma_2) \in \Gamma^2 : \{\gamma_1 + k\gamma_2 : k \in \{0,2,3,6\}\} \subseteq B \right\}.
    \end{equation*}
    Then by Theorem \ref{thm:szemeredi}, $d^*(R) \gg_c 1$.
    Note that if $(\gamma_1, \gamma_2) \in R$, then
    \begin{equation*}
       u_1(\gamma_1) u_2(2\gamma_1) u_3(3\gamma_1) u_1(\gamma_1 + 6\gamma_2) u_2(2(\gamma_1 + 3\gamma_2)) u_3(3(\gamma_1 + 2\gamma_2)) \ge \frac{c^4}{16}.
    \end{equation*}
    Thus, performing a change of variables $(\gamma_1, \gamma_2) \to (\gamma_1, \gamma_1 + 6\gamma_2)$, we have
    \begin{multline*}
         \UC_{(\gamma_1, \gamma_2) \in \Gamma^2} u_1(\gamma_1) u_2(2\gamma_1) u_3(3\gamma_1) u_1(\gamma_2) u_2(\gamma_2 + \gamma_1) u_3(\gamma_2 + 2\gamma_1) \\
          \ge \frac{1}{[\Gamma:6\Gamma]} \UC_{(\gamma_1, \gamma_2) \in \Gamma^2} u_1(\gamma_1) u_2(2\gamma_1) u_3(3\gamma_1) \\ u_1(\gamma_1 + 6\gamma_2) u_2(2(\gamma_1 + 3\gamma_2)) u_3(3(\gamma_1 + 2\gamma_2)) 
          \ge \frac{c^4 d^*(R)}{16 [\Gamma:6\Gamma]} \gg_c 1.
    \end{multline*}

    Putting everything together, $I(f_1, f_2, f_3) \gg_c 1 + O(\eps)$.
    Hence, letting $\eps \to 0$, we have $I(f_1, f_2, f_3) \gg_c 1$, so \eqref{eq:desired positivity} holds.
\end{proof}


\section{Completing the proof}\label{sec:3.6}

We now have all of the ingredients to prove Theorem \ref{thm:dynamical sumset}.

\begin{proof}[Proof of Theorem \ref{thm:dynamical sumset}]
    By passing to an extension if needed, we may assume without loss of generality that all of the conditions at the start of Section \ref{sec:lifting} are satisfied.
    In the context of $\Z$-systems, this was shown in \cite[Lemma 5.8]{kmrr}.
    We carry out the necessary modifications for general $\Gamma$-systems in the appendix; see Theorem \ref{thm: topological factors} (the application of which might change the F{\o}lner sequence).
    
    Now by Theorem \ref{thm:sigma is progressive}, the measure $\sigma^{(4)}_{T_X^ta}$ is 4-progressive from $T_X^ta$ for every $t \in \Gamma$.
	Fix a F{\o}lner sequence $(\Psi_N)$ in $\Gamma$.
	We claim that
	\begin{equation} \label{eq:average of shifts}
		\liminf_{N \to \infty} \frac{1}{|\Psi_N|} \sum_{t \in \Psi_N} \sigma^{(4)}_{T_X^ta}(E \times E \times E) > 0.
	\end{equation}
	To see this, let $\eps > 0$.
    Since $E$ is an open set with $\mu_X(E) > 0$, we may find a continuous function $0 \le f \le 1$ with $\supp(f) \subseteq E$ and $\int_X f~d\mu_X \ge \frac{\mu_X(E)}{2} > 0$.
	Then, approximating $\E[f \mid Z_2]$ by a continuous function, we may decompose $f = g \circ \pi_{Z_2} + h$ with $g \in C(Z_2)$, $h \in C(X)$ and $\| \E[h \mid Z_2] \|_{L^1(\mu_X)} < \eps$.
    We may assume $\int_{Z_2} g~d\mu_{Z_2} = \int_X f~d\mu_X > 0$ and $\int_X h~d\mu_X = 0$.
	From the definition of $\sigma^{(4)}_{T_X^ta}$, we then have
	\begin{multline*}
		\sigma^{(4)}_{T_X^ta}(E \times E \times E) \ge \int_{X^3} f \otimes f \otimes f~d\sigma^{(4)}_{T_X^ta} \\
        = \int_{X^3} \left( (g \circ \pi_Z + h) \otimes (g \circ \pi_Z + h) \otimes (g \circ \pi_Z + h) \right)~d\sigma^{(4)}_{T_X^ta} \\ = \int_{Z_2^3} (g \otimes g \otimes g)~d\tilde{\sigma}^{(4)}_{T_X^ta} + O(\eps),
	\end{multline*}
    where the implicit constant in the $O(\eps)$ term depends on $[\Gamma : 6\Gamma]$.
	Then by Theorem \ref{thm:properties of sigma}(3) and Theorem \ref{thm:uniform Sz},
	\begin{equation*}
		\lim_{N \to \infty} \frac{1}{|\Psi_N|} \sum_{t \in \Psi_N} \int_{Z_2^3} (g \otimes g \otimes g)~d\tilde{\sigma}^{(4)}_{T_X^ta} 
        = \lim_{N \to \infty} \frac{1}{|\Psi_N|} \sum_{\gamma \in \Psi_N} \int_{Z_2} g \cdot T_{Z_2}^{\gamma}g \cdot T_{Z_2}^{2\gamma}g~d\mu_{Z_2} \ge c,
	\end{equation*}
    where the constant $c > 0$ depends only on $\int_{Z_2} g~d\mu_{Z_2} = \int_X f~d\mu_X \ge \frac{\mu_X(E)}{2}$.
	But $\eps > 0$ was arbitrary, so
    \begin{equation*}
        \liminf_{N \to \infty} \frac{1}{|\Psi_N|} \sum_{t \in \Psi_N} \sigma^{(4)}_{T_X^ta}(E \times E \times E) \ge c > 0.
    \end{equation*}
    That is, \eqref{eq:average of shifts} holds.
	In particular, there exists $t \in \Gamma$ such that $\sigma^{(4)}_{T_X^ta}(E \times E \times E) > 0$.
	Thus, by Lemma \ref{lem:positive sigma measure implies EP}, there exists a 4-term Erd\H{o}s progression $(T_X^ta, x_1, x_2, x_3) \in X^4$ such that $(x_1, x_2, x_3) \in E \times E \times E$.
\end{proof}


\part*{Appendix}
\appendix

\section{Conze--Lesigne systems as topological inverse limits of translational systems}

Throughout this appendix, we fix an arbitrary countable discrete abelian group $\Gamma$. 
The main result of \cite{jst} says that ergodic Conze--Lesigne $\Gamma$-systems are isomorphic to inverse limits of degree 2 nilpotent translational $\Gamma$-systems in the category of measure-preserving $\Gamma$-systems (see Theorem \ref{thm: structure theorem}). 
The goal of this appendix is to realize this inverse limit within the category of topological dynamical $\Gamma$-systems.
First we need a lemma allowing us to replace measurable factor maps with topological factor maps. 

\begin{lemma} \label{lem: make factor map continuous}
     Let $\pi\colon (G/\Lambda, \Sigma_{G/\Lambda}, \mu_{G/\Lambda}, T_{G/\Lambda})\to (G'/\Lambda', \Sigma_{G'/\Lambda'}, \mu_{G'/\Lambda'}, T_{G'/\Lambda'})$ be a measurable factor map of ergodic degree 2 nilpotent translational $\Gamma$-systems. 
    Then there is a degree 2 nilpotent translational $\Gamma$-system $(\tilde{G}/\tilde{\Lambda}, \Sigma_{\tilde{G}/\tilde{\Lambda}}, \mu_{\tilde{G}/\tilde{\Lambda}}, T_{\tilde{G}/\tilde{\Lambda}})$ with the following properties. 
    \begin{enumerate}
        \item There is a measurable isomorphism $\iota : \tilde{G}/\tilde{\Lambda} \to G/\Lambda$. 
        \item There is a topological factor map $\tilde{\pi} : \tilde{G}/\tilde{\Lambda} \to G'/\Lambda'$ such that $\tilde{\pi} = \pi \circ \iota$ $\mu_{\tilde{G}/\tilde{\Lambda}}$-almost surely.
    \end{enumerate}
\end{lemma}

\begin{proof}
    We essentially follow the arguments in \cite[Proposition 3.20]{kmrr} and \cite[Proposition 3.7]{cm} to obtain a measurably isomorphic extension of $G/\Lambda$ that has a continuous factor map to $G'/\Lambda'$.
    The new observation is that the extension system is still a translational system.
    
    Fix a point $x_0 \in G/\Lambda$.
    Then $x_0$ is a transitive point, since ergodic translational systems are minimal (see Corollary \ref{cor: ergodic, minimal, ue}).
    Moreover, the system $(G'/\Lambda', T_{G'/\Lambda'})$ is distal by Proposition \ref{prop: distal}.
    Hence, by \cite[Lemma 3.5]{cm}, there exists a point $y_0 \in G'/\Lambda'$ and a F{\o}lner sequence $(\Psi_N)$ such that for any $f_1 \in C(G/\Lambda)$ and $f_2 \in C(G'/\Lambda')$, we have
    \begin{equation} \label{eq: limit from Host-Kra lemma}
        \lim_{N \to \infty} \frac{1}{|\Psi_N|} \sum_{\gamma \in \Phi_N} f_1(T_{G/\Lambda}^{\gamma}x_0) f_2(T_{G'/\Lambda'}^{\gamma}y_0) = \int_{G/\Lambda} f_1 \cdot (f_2 \circ \pi)~d\mu_{G/\Lambda}.
    \end{equation}
    Let $H = G \times G'$ and $\Delta = \Lambda \times \Lambda'$.
    Let $T_{H/\Delta}^{\gamma} = T_{G/\Lambda}^{\gamma} \times T_{G'/\Lambda'}^{\gamma}$ for $\gamma \in \Gamma$.
    Then $(H/\Delta, T_{H/\Delta})$ is a degree 2 nilpotent translational system, so by Theorem \ref{thm: main}, there exists a closed subgroup $\tilde{G} \subseteq H$ such that the orbit of the point $z_0 = (x_0,y_0) \in H/\Delta$ is well-distributed in $\tilde{G}z_0$.
    That is,
    \begin{equation*}
        \UC_{\gamma \in \Gamma} f_1(T_{G/\Lambda}^{\gamma}x_0) f_2(T_{G'/\Lambda'}^{\gamma}y_0) = \int_{\tilde{G}z_0} f_1 \otimes f_2~d\mu_{\tilde{G}z_0}
    \end{equation*}
    for $f_1 \in C(G/\Lambda)$ and $f_2 \in C(G'/\Lambda')$.
    Comparing with \eqref{eq: limit from Host-Kra lemma}, we have
    \begin{equation} \label{eq: expression for measure on orbit}
        \int_{\tilde{G}z_0} f_1 \otimes f_2~d\mu_{\tilde{G}z_0} = \int_{G/\Lambda} f_1 \cdot (f_2 \circ \pi)~d\mu_{G/\Lambda}
    \end{equation}
    for $f_1 \in C(G/\Lambda)$ and $f_2 \in C(G'/\Lambda')$.
    Let $\tilde{\Lambda} = \tilde{G} \cap z_0\Delta z_0^{-1}$ be the stabilizer of the point $z_0$ for the action of $\tilde{G}$ on $H/\Delta$.
    Then the map $\xi : h\tilde{\Lambda} \mapsto hz_0$ provides a (topological) isomorphism between the degree 2 nilpotent translational system $(\tilde{G}/\tilde{\Lambda}, T_{\tilde{G}/\tilde{\Lambda}})$ and the topological dynamical system $(\tilde{G}z_0, T_{\tilde{G}z_0})$.

    Let $\rho : G/\Lambda \to H/\Delta$ be the map $\rho(x) = (x, \pi(x))$.
    Given continuous functions $f_1 \in C(G/\Lambda)$ and $f_2 \in C(G'/\Lambda')$, we have
    \begin{equation*}
        \int_{G/\Lambda} (f_1 \otimes f_2) \circ \rho~d\mu_{G/\Lambda} = \int_{G/\Lambda} f_1 \cdot (f_2 \circ \pi)~d\mu_{G/\Lambda},
    \end{equation*}
    so by \eqref{eq: expression for measure on orbit},
    \begin{equation*}
        \int_{G/\Lambda} (f_1 \otimes f_2) \circ \rho~d\mu_{G/\Lambda} = \int_{\tilde{G}z_0} f_1 \otimes f_2~d\mu_{\tilde{G}z_0}.
    \end{equation*}
    In other words, $\rho$ establishes an isomorphism between $(G/\Lambda, \Sigma_{G/\Lambda}, \mu_{G/\Lambda}, T_{G/\Lambda})$ and $(H/\Delta, \Sigma_{H/\Delta}, \mu_{\tilde{G}z_0}, T_{H/\Delta})$ as measure-preserving systems.
    In particular, $\rho(x) \in \tilde{G}_{z_0}$ for almost every $x \in G/\Lambda$, so the map $\xi^{-1} \circ \rho : G/\Lambda \to \tilde{G}/\tilde{\Lambda}$ is defined almost everywhere and induces an isomorphism of measure-preserving systems.
    
    Let $\pi_1 : H/\Delta \to G/\Lambda$ and $\pi_2 : H/\Delta \to G'/\Lambda'$ be the coordinate projection maps.
    Note that the continuous surjection $\iota = \pi_1 \circ \xi$ is a left-inverse to $\rho$. 

    Now define $\tilde{\pi} : \tilde{G}/\tilde{\Lambda} \to G'/\Lambda'$ by $\tilde{\pi} = \pi_2 \circ \xi$.
    Then $\tilde{\pi}$, being the composition of two continuous surjective maps, is a continuous surjection.
    Moreover, for $z \in \tilde{G}/\tilde{\Lambda}$ and $\gamma \in \Gamma$,
    \begin{equation*}
        \tilde{\pi}(T_{\tilde{G}/\tilde{\Lambda}}z) = \pi_2(T_{H/\Delta}\xi(z)) = T_{G'/\Lambda'}\pi_2(\xi(z)) = T_{G'/\Lambda'}\tilde{\pi}(z),
    \end{equation*}
    so $\tilde{\pi}$ is a topological factor map.

    Finally, for almost every $z \in \tilde{G}/\tilde{\Lambda}$, we may write $z = \xi^{-1}(\rho(x))$ for some $x \in G/\Lambda$, whence
    \begin{equation*}
        \pi(\iota(z)) = \pi(\pi_1(\xi(z))) = \pi(\pi_1(\rho(x))) = \pi(x) = \pi_2(\rho(x)) = \pi_2(\xi(\rho(z)) = \tilde{\pi}(z).
    \end{equation*}
\end{proof}

\begin{proposition} \label{prop: topological inverse limit}
    Let $(X, \Sigma_X, \mu_X, T_X)$ be an ergodic Conze--Lesigne $\Gamma$-system.
    Then there exists a uniquely ergodic topological dynamical $\Gamma$-system $(\tilde{X}, T_{\tilde{X}})$ with unique invariant measure $\mu_{\tilde{X}}$ satisfying the following properties. 
    \begin{enumerate}
        \item $(\tilde{X}, T_{\tilde{X}})$ is an inverse limit of degree 2 nilpotent translational $\Gamma$-systems as a topological dynamical $\Gamma$-system.
        \item $(\tilde{X}, \Sigma_{\tilde{X}}, \mu_{\tilde{X}}, T_{\tilde{X}})$ is measurably isomorphic to $(X, \Sigma_X, \mu_X, T_X)$.
    \end{enumerate}
\end{proposition}

\begin{proof}
    The system $(X, \Sigma_X, \mu_X, T_X)$ is an inverse limit of ergodic degree 2 nilpotent translational $\Gamma$-systems $(G_n/\Lambda_n, \Sigma_{G_n/\Lambda_n}, \mu_{G_n/\Lambda_n}, T_{G_n/\Lambda_n})$ as a measure-preserving $\Gamma$-system due to Theorem \ref{thm: structure theorem}. 

    For $n \le m$, let $\pi^m_n$ denote the measurable factor map $\pi^m_n : G_m/\Lambda_m \to G_n/\Lambda_n$, and let $\pi_n$ be the measurable factor map $\pi_n : X \to G_n/\Lambda_n$.
    The strategy of proof is to replace each system $(G_n/\Lambda_n, \Sigma_{G_n/\Lambda_n}, \mu_{G_n/\Lambda_n}, T_{G_n/\Lambda_n})$ by a measurably isomorphic translational $\Gamma$-system so that the factor maps $\pi^m_n$ become continuous.

    We will construct by induction $\Gamma$-systems $(\tilde{G}_n/\tilde{\Lambda}_n, \Sigma_{\tilde{G}_n/\tilde{\Lambda}_n}, \mu_{\tilde{G}_n/\tilde{\Lambda}_n}, T_{\tilde{G}_n/\tilde{\Lambda}_n})$ with a measurable isomorphism $\iota_n : \tilde{G}_n/\tilde{\Lambda}_n \to G_n/\Lambda_n$ and continuous factor maps $\tilde{\pi}^m_n : \tilde{G}_m/\tilde{\Lambda}_m \to \tilde{G}_n/\tilde{\Lambda}_n$ such that $\iota_n \circ \tilde{\pi}^m_n = \pi^m_n \circ \iota_m$ almost everywhere for $n \le m$.
    Let
    \begin{equation*}
        (\tilde{G}_1/\tilde{\Lambda}_1, \Sigma_{\tilde{G}_1/\tilde{\Lambda}_1}, \mu_{\tilde{G}_1/\tilde{\Lambda}_1}, T_{\tilde{G}_1/\tilde{\Lambda}_1}) = (G_1/\Lambda_1, \Sigma_{G_1/\Lambda_1}, \mu_{G_1/\Lambda_1}, T_{G_1/\Lambda_1}).
    \end{equation*}
    Suppose $(\tilde{G}_i/\tilde{\Lambda}_i, \Sigma_{\tilde{G}_i/\tilde{\Lambda}_i}, \mu_{\tilde{G}_i/\tilde{\Lambda}_i}, T_{\tilde{G}_i/\tilde{\Lambda}_i})$ and the measurable isomorphisms $\iota_i$ and continuous factor maps $\tilde{\pi}^j_i$ have been defined for $i \le j \le n$.
    The map $\iota_n^{-1} \circ \pi^{n+1}_n : G_{n+1}/\Lambda_{n+1} \to \tilde{G}_n/\tilde{\Lambda}_n$ is a measurable factor map.
    By Lemma \ref{lem: make factor map continuous}, there is a degree 2 nilpotent translational $\Gamma$-system $(\tilde{G}_{n+1}/\tilde{\Lambda}_{n+1}, \Sigma_{\tilde{G}_{n+1}/\tilde{\Lambda}_{n+1}}, \mu_{\tilde{G}_{n+1}/\tilde{\Lambda}_{n+1}}, T_{\tilde{G}_{n+1}/\tilde{\Lambda}_{n+1}})$, a measurable isomorphism $\iota_{n+1} : \tilde{G}_{n+1}/\tilde{\Lambda}_{n+1} \to G_{n+1}/\Lambda_{n+1}$, and a topological factor map $\tilde{\pi}^{n+1}_n : \tilde{G}_{n+1}/\tilde{\Lambda}_{n+1} \to \tilde{G}_n/\tilde{\Lambda}_n$ such that $\tilde{\pi}^{n+1}_n = \iota_n^{-1} \circ \pi^{n+1}_n \circ \iota_{n+1}$.
    Let $\tilde{\pi}^{n+1}_i = \tilde{\pi}^n_i \circ \tilde{\pi}^{n+1}_n$ for $i < n$ and $\tilde{\pi}^{n+1}_{n+1} = \id_{\tilde{G}_{n+1}/\tilde{\Lambda}_{n+1}}$.
    By the induction hypothesis, $\tilde{\pi}^{n+1}_i$ is a topological factor map for each $i \le n+1$, and
    \begin{equation*}
        \iota_i \circ \tilde{\pi}^{n+1}_i = \iota_i \circ \tilde{\pi}^n_i \circ \tilde{\pi}^{n+1}_n = \pi^n_i \circ \iota_n \circ \iota_n^{-1} \circ \pi^{n+1}_n \circ \iota_{n+1} = \pi^n_i \circ \pi^{n+1}_n \circ \iota_{n+1} = \pi^{n+1}_i \circ \iota_{n+1}.
    \end{equation*}
    This completes the induction.

    By construction, the maps $\tilde{\pi}^m_n$ satisfy the composition rule $\tilde{\pi}^m_n \circ \tilde{\pi}^k_m = \tilde{\pi}^k_n$ for $n \le m \le k$, so we may define the inverse limit $(\tilde{X}, T_{\tilde{X}})$ of the systems $(\tilde{G}_n/\tilde{\Lambda}_n, T_{\tilde{G}_n/\tilde{\Lambda}_n})$.
    Each of the systems $(\tilde{G}_n/\tilde{\Lambda}_n, T_{\tilde{G}_n/\tilde{\Lambda}_n})$ is uniquely ergodic by Corollary \ref{cor: ergodic, minimal, ue}, so the system $(\tilde{X}, T_{\tilde{X}})$ is also uniquely ergodic with unique invariant measure $\mu_{\tilde{X}}$ determined by the identity
    \begin{equation*}
        \int_{\tilde{X}} f \circ \tilde{\pi}_n~d\mu_{\tilde{X}} = \int_{\tilde{G}_n/\tilde{\Lambda}_n} f~d\mu_{\tilde{G}_n/\tilde{\Lambda}_n}
    \end{equation*}
    for $f \in C(\tilde{G}_n/\tilde{\Lambda}_n)$ and $n \in \N$ (cf. \cite[Proposition 22, \S 6.4]{host-kra-book}).
    Therefore, the measure-preserving $\Gamma$-system $(\tilde{X}, \Sigma_{\tilde{X}}, \mu_{\tilde{X}}, T_{\tilde{X}})$ is the inverse limit of the measure-preserving $\Gamma$-systems $(\tilde{G}_n/\tilde{\Lambda}_n, \Sigma_{\tilde{G}_n/\tilde{\Lambda}_n}, \mu_{\tilde{G}_n/\tilde{\Lambda}_n}, T_{\tilde{G}_n/\tilde{\Lambda}_n})$.
    Note that the following diagram (in the category of measure-preserving systems) commutes for $n \le m$:
    \begin{center}
        \begin{tikzcd}
            \tilde{X} \arrow[d, "\tilde{\pi}_m" left] & X \arrow[d, "\pi_m" right] \\
            \tilde{G}_m/\tilde{\Lambda}_m \arrow[r, "\iota_m"] \arrow[d, "\tilde{\pi}^m_n" left] & G_m/\Lambda_m \arrow[d, "\pi^m_n" right] \\
            \tilde{G}_n/\tilde{\Lambda}_n \arrow[r, "\iota_n"] & G_n/\Lambda_n
        \end{tikzcd}
    \end{center}
    Therefore, by the universal property for inverse limits (cf.~\cite[\S 6.3]{host-kra-book}), there is a (unique) measurable isomorphism $\iota : \tilde{X} \to X$ such that $\iota_n \circ \tilde{\pi}_n = \pi_n \circ \iota$ for every $n \in \N$.
\end{proof}

\begin{theorem} \label{thm: topological factors}
    Let $(X, T_X)$ be a topological dynamical $\Gamma$-system, and let $a \in X$ be a transitive point.
    Suppose $\mu_X$ is an ergodic $T_X$-invariant Borel probability measure on $X$ and $a \in \gen(\mu_X,\Phi)$ for some F{\o}lner sequence $\Phi = (\Phi_N)$.
    Then there exists an extension $\pi : (\tilde{X}, T_{\tilde{X}}) \to (X, T_X)$, a transitive point $\tilde{a} \in \tilde{X}$, a F{\o}lner sequence $\Psi = (\Psi_N)$, and an ergodic $T_{\tilde{X}}$-invariant Borel probability measure $\mu_{\tilde{X}}$ on $\tilde{X}$ satisfying the following properties. 
    \begin{enumerate}
        \item $\tilde{a} \in \gen(\mu_{\tilde{X}},\Psi)$. 
        \item The map $\pi\colon (\tilde{X}, \Sigma_{\tilde{X}}, \mu_{\tilde{X}}, T_{\tilde{X}})\to (X, \Sigma_X, \mu_X, T_X)$ establishes an isomorphism of measure-preserving $\Gamma$-systems. 
        \item There is a uniquely ergodic topological dynamical $\Gamma$-system $(Z_2, T_{Z_2})$ that is an inverse limit of degree 2 nilpotent translational $\Gamma$-systems such that $(Z_2, \Sigma_{Z_2}, \mu_{Z_2}, T_{Z_2})$ is measurably isomorphic to the Conze--Lesigne factor of $(\tilde{X}, \Sigma_{\tilde{X}}, \mu_{\tilde{X}}, T_{\tilde{X}})$,
        \item There is a uniquely ergodic rotational $\Gamma$-system $(Z, T_Z)$ such that the Kronecker factor of $(\tilde{X}, \Sigma_{\tilde{X}}, \mu_{\tilde{X}}, T_{\tilde{X}})$ is isomorphic to $(Z, \Sigma_Z, \mu_Z, T_Z)$ as measure-preserving $\Gamma$-systems. 
        \item There are topological factor maps $\tilde{\pi}_{Z_2} : \tilde{X} \to Z_2$ and $\tilde{\pi}_Z : \tilde{X} \to Z$.
    \end{enumerate}
\end{theorem}

\begin{proof}
    By the Halmos--von Neumann theorem, we may assume that the Kronecker factor $(Z, \Sigma_Z, \mu_Z, T_Z)$ arises from a uniquely ergodic rotational $\Gamma$-system $(Z, T_Z)$.
    By Proposition \ref{prop: topological inverse limit}, we may assume that the Conze--Lesigne factor of $(X, \Sigma_X, \mu_X, T_X)$ is of the form $(Z_2, \Sigma_{Z_2}, \mu_{Z_2}, T_{Z_2})$, where $(Z_2, T_{Z_2})$ is uniquely ergodic and equal to an inverse limit (as a topological dynamical system) of degree 2 nilpotent translational systems.
    
    Let $\pi_Z : X \to Z$ and $\pi_{Z_2} : X \to Z_2$ be the (measurable) factor maps.
    Then by \cite[Lemma 3.5]{cm}, there exists a point $(z_1, z_2) \in Z \times Z_2$ and a F{\o}lner sequence $\Psi = (\Psi_N)$ such that
    \begin{equation} \label{eq: limit from Host-Kra lemma 2}
        \lim_{N \to \infty} \frac{1}{|\Psi_N|} \sum_{\gamma \in \Phi_N} f(T_X^{\gamma}a) g_1(T_Z^{\gamma}z_1) g_2(T_{Z_2}^{\gamma}z_2) = \int_X f \cdot (g_1 \circ \pi_Z) \cdot (g_2 \circ \pi_{Z_2})~d\mu_X
    \end{equation}
    for every $f \in C(X)$, $g_1 \in C(Z)$, and $g_2 \in C(Z_2)$.
    Let $\tilde{X} = X \times Z \times Z_2$ with the measure $\mu_{\tilde{X}}$ defined by
    \begin{equation*}
        \int_{\tilde{X}} f \otimes g_1 \otimes g_2~d\mu_{\tilde{X}} = \int_X f \cdot (g_1 \circ \pi_Z) \cdot (g_2 \circ \pi_{Z_2})~d\mu_X
    \end{equation*}
    for $f \in C(X)$, $g_1 \in C(Z)$, and $g_2 \in C(Z_2)$.
    
    Let $\pi : \tilde{X} \to X$ be the projection onto the first coordinate.
    Note that $x \mapsto (x,\pi_Z(x), \pi_{Z_2}(x))$ is an almost sure inverse to $\pi$, so $\pi : \tilde{X} \to X$ is a measurable isomorphism between $(\tilde{X}, \Sigma_{\tilde{X}}, \mu_{\tilde{X}}, T_{\tilde{X}})$ and $(X, \Sigma_X, \mu_X, T_X)$.
    That is, (2) holds.
    Since $(\tilde{X}, \Sigma_{\tilde{X}}, \mu_{\tilde{X}}, T_{\tilde{X}})$ and $(X, \Sigma_X, \mu_X, T_X)$ are isomorphic, they have the same (measurable) Kronecker factor and Conze--Lesigne factor, so items (3) and (4) hold.
    The maps $\tilde{\pi}_Z\colon \tilde{X}\to Z$ and $\tilde{\pi}_{Z_2}\colon \tilde{X}\to Z_2$ given by projection onto the second coordinate and third coordinate respectively are topological factor maps, so (5) holds.
    Finally, letting $\tilde{a} = (a, z) \in \tilde{X}$, we have property (1) by \eqref{eq: limit from Host-Kra lemma 2}.
\end{proof}


\end{document}